\pdfoutput=1
\documentclass[11pt]{article} 
\def\arXiv{1} 

\def\usejohntable{1}

\usepackage{enumitem}
\usepackage{amssymb}
\usepackage{amsbsy}
\usepackage{amsmath}
\usepackage{amsfonts}
\usepackage{latexsym}
\usepackage{graphicx}
\usepackage{color}
\usepackage{xifthen}
\usepackage{xspace}
\usepackage{mathtools}
\usepackage{bbm}
\usepackage{multirow}
\usepackage[normalem]{ulem}
\usepackage{array}
\usepackage[utf8]{inputenc}
\usepackage[T1]{fontenc}
\usepackage{etoolbox}
\newtoggle{restatements}

\newtoggle{heavyplots}

\usepackage{xfrac}
\usepackage{booktabs}
\newcommand{\notarxiv}[1]{foo}
\newcommand{\arxiv}[1]{ba}
\ifdefined \arXiv
	\renewcommand{\arxiv}[1]{#1}%
	\renewcommand{\notarxiv}[1]{\ignorespaces}%
\else%
	\renewcommand{\arxiv}[1]{\ignorespaces}%
	\renewcommand{\notarxiv}[1]{#1}%
\fi%

\usepackage{titlesec}

\usepackage{amsthm}
\usepackage{thmtools}
\usepackage{thm-restate}
\usepackage[linesnumbered,ruled,vlined]{algorithm2e}

\notarxiv{
\usepackage[utf8]{inputenc} %
\usepackage[T1]{fontenc}    %
\usepackage{url}            %
\usepackage{booktabs}       %
\usepackage{amsfonts}       %
\usepackage{nicefrac}       %
\usepackage{microtype}      %

\usepackage{tabularx}
\usepackage[noend]{algpseudocode}
\usepackage{stackengine}
\usepackage{comment}
\definecolor{darkblue}{rgb}{0,0,.5}
\usepackage[colorlinks=true,allcolors=darkblue]{hyperref}
}

\arxiv{
	\titlespacing*{\paragraph}{0pt}{6pt plus 2pt minus 2pt}{6pt}
	\usepackage[dvipsnames]{xcolor}
	\usepackage[top=1in, right=1in, left=1in, bottom=1in]{geometry}
	\usepackage[numbers,square]{natbib}
	\usepackage{url}
	\usepackage[hidelinks]{hyperref}
	\hypersetup{
		colorlinks=true,
		linkcolor=blue!70!black,
		citecolor=blue!70!black,
		urlcolor=blue!70!black}
        \usepackage{comment}
}

\usepackage{empheq}

\usepackage{cases}

\definecolor{innerboxcolor}{rgb}{.9,.95,1}
\definecolor{outerlinecolor}{rgb}{.6,0,.2}

\theoremstyle{plain}

\newtheorem{lemma}{Lemma}
\newtheorem{claim}{Claim}
\newtheorem{proposition}{Proposition}

\newtheorem{definition}{Definition}

\theoremstyle{definition}
\newtheorem{remark}{Remark}

\newtheorem*{example*}{Example}

\theoremstyle{plain}
\newtheorem{innercustomthm}{Theorem}
\newenvironment{customthm}[1]
{\renewcommand\theinnercustomthm{#1}\innercustomthm}
{\endinnercustomthm}

\newtheorem{innercustomprop}{Proposition}
\newenvironment{customprop}[1]
{\renewcommand\theinnercustomprop{#1}\innercustomprop}
{\endinnercustomprop}

\newtheorem{innercustomclaim}{Claim}
\newenvironment{customclaim}[1]
{\renewcommand\theinnercustomclaim{#1}\innercustomclaim}
{\endinnercustomclaim}

\newtheorem{assumption}{Assumption}

\usepackage{cleveref}

\newcommand{\mc}[1]{\mathcal{#1}}

\newcommand{\what}[1]{\widehat{#1}}  %
\newcommand{\lone}[1]{\norm{#1}_1} %
\newcommand{\linf}[1]{\norm{#1}_\infty} %
\newcommand{\R}{\mathbb{R}} %
\newcommand{\N}{\mathbb{N}} %
\DeclareMathOperator{\E}{\mathbb{E}} %
\newcommand{\Var}{\mathrm{Var}}
\newcommand{\var}{\mathrm{Var}}
\renewcommand{\P}{\mathbb{P}}	%

\newcommand{\simiid}{\stackrel{\textup{iid}}{\sim}}

\SetCommentSty{mycommfont}
\SetKwInput{KwInput}{Input}                %
\SetKwInput{KwOutput}{Output}              %
\SetKwInput{KwReturn}{Return}              %

\makeatletter
\let\oldnl\nl%
\newcommand{\nonl}{\renewcommand{\nl}{\let\nl\oldnl}}%
\makeatother

\DeclareMathOperator*{\argmax}{arg\,max}

\DeclareMathOperator*{\argmin}{arg\,min}

\providecommand{\sign}{\mathop{\rm sign}}

\newcommand{\hinge}[1]{\left({#1}\right)_+} %

\newcommand{\abs}[1]{}
\let\abs\undefined
\newcommand{\brk}[1]{}
\let\brk\undefined
\newcommand{\prn}[1]{}
\let\prn\undefined
\newcommand{\crl}[1]{}
\let\crl\undefined

\DeclarePairedDelimiter{\abs}{\lvert}{\rvert} 
\DeclarePairedDelimiter{\brk}{[}{]}
\DeclarePairedDelimiter{\crl}{\{}{\}}
\DeclarePairedDelimiter{\prn}{(}{)}

\DeclarePairedDelimiter{\norm}{\|}{\|}
\DeclarePairedDelimiter{\tri}{\langle}{\rangle}

\DeclarePairedDelimiter{\ceil}{\lceil}{\rceil}
\DeclarePairedDelimiter{\floor}{\lfloor}{\rfloor}

\providecommand{\minimize}{\mathop{\rm minimize}}

\newcommand{\half}{\frac{1}{2}}

\newcommand{\defeq}{\coloneqq}
\newcommand{\eqdef}{\eqqcolon}
\newcommand{\grad}{\nabla}
\renewcommand{\d}{\mathrm{d}}
\newcommand{\del}{\partial}

\newcommand{\xset}{\mathcal{X}}

\def\ie{{i.e.\ }}

\newcommand{\eps}{\epsilon}

\renewcommand{\floor}[1]{\lfloor #1 \rfloor}

\newcommand{\innermid}{\nonscript\;\delimsize\vert\nonscript\;}
\newcommand{\activatebar}{%
	\begingroup\lccode`\~=`\|
	\lowercase{\endgroup\let~}\innermid 
	\mathcode`|=\string"8000
}

\newcommand{\inner}[2]{\left<#1,#2\right>}

\NewDocumentCommand{\prox}{ m m 
O{\alpha}}{\mathrm{Prox}_{#1}^{#3}(#2)}

\newcommand{\simplexN}{\Delta^N}

\newcommand{\cs}{\chi^2}
\renewcommand{\div}{\mathrm{D}}
\newcommand{\divcs}{\div_{\cs}}

\newcommand{\phidiv}[2]{\div_\phi\left({#1}, {#2}\right)}

\newcommand{\divpsi}{\div_{\psi}}
\newcommand{\divkl}{\div_{\mathrm{kl}}}
\newcommand{\uset}{\mc{U}}
\newcommand{\usetCVaR}{\uset^\alpha_{\textup{CVaR}}}
\newcommand{\usetCS}{\uset^\rho_{\cs}}
\newcommand{\Phat}[1]{\what{P}\brk{#1}}

\newcommand{\ones}{\boldsymbol{1}}

\newcommand{\loss}{\ell}
\newcommand{\opt}{^\star}
\newcommand{\x}{x}
\newcommand{\X}{\mathcal{X}}
\newcommand{\s}{s}
\renewcommand{\S}{S}
\renewcommand{\ss}{\mathbb{S}} %
\renewcommand{\L}{\mc{L}} %
\newcommand{\CSbdd}{C}
\newcommand{\Lcvar}{\L_{\textup{CVaR}}}
\newcommand{\LcvarR}{\L_{\textup{kl-CVaR}}}
\newcommand{\Lcs}{\L_{\cs}}
\newcommand{\Llam}{\L_{\cs\textup{-pen}}}
\newcommand{\bL}{\overline{\L}}
\newcommand{\bLcvar}{\bL_{\textup{CVaR}}}
\newcommand{\bLcvarR}{\bL_{\textup{kl-CVaR}}}
\newcommand{\bLcs}{\bL_{\cs}}
\newcommand{\bLlam}{\bL_{\cs\textup{-pen}}}

\newcommand{\icdfLip}{G_{\textup{icdf}}}

\newcommand{\wset}{\mathcal{W}}
\newcommand{\rset}{\mathcal{R}}

\newcommand{\err}{\mathrm{err}_T}

\newcommand{\Fhat}{\what{F}}

\newcommand{\kmax}{n}
\newcommand{\jmax}{j_{\max}}
\newcommand{\mld}{\what{\mc{D}}}
	
\newcommand{\betarv}{\mathsf{Beta}}
\newcommand{\binrv}{\mathsf{Bin}}
\newcommand{\bernrv}{\mathsf{Bernoulli}}

\newcommand{\unirv}{\mathsf{Unif}}

\newcommand{\ml}{\what{\mc{M}}}
\newcommand{\funct}{\mathsf{F}}

\newcommand{\biasbound}{\mathrm{bb}}
\newcommand{\softind}{\mc{I}_\alpha}

\newcommand{\indic}[1]{1_{\{#1\}}}
\newcommand{\convind}{\mathbb{I}}

\newcommand{\lambdaH}{\overline{\lambda}}
\newcommand{\lambdaL}{\underline{\lambda}}
\newcommand{\pind}[1]{^{(#1)}}
\newcommand{\LcsRange}[1]{\L_{\cs [#1]}}

\makeatletter
\long\def\@makecaption#1#2{
  \vskip 0.8ex
  \setbox\@tempboxa\hbox{\small {\bf #1:} #2}
  \parindent 1.5em  %
  \dimen0=\hsize
  \advance\dimen0 by -3em
  \ifdim \wd\@tempboxa >\dimen0
  \hbox to \hsize{
    \parindent 0em
    \hfil 
    \parbox{\dimen0}{\def\baselinestretch{0.96}\small
      {\bf #1.} #2
    } 
    \hfil}
  \else \hbox to \hsize{\hfil \box\@tempboxa \hfil}
  \fi
}
\makeatother

\notarxiv{
 \titlespacing*{\section}
 {0pt}{6pt plus 2pt minus 1pt}{3pt plus 0pt minus 3pt}
 \titlespacing*{\subsection}
 {0pt}{3pt plus 0pt minus 3pt}{0pt plus 0pt minus 3pt}
 \titlespacing*{\paragraph}
 {0pt}{0pt plus 0pt minus 0pt}{6pt}
}

\newcommand\blfootnote[1]{%
	\begingroup
	\renewcommand\thefootnote{}\footnote{#1}%
	\addtocounter{footnote}{-1}%
	\endgroup
}

\title{Large-Scale Methods for \\
  Distributionally Robust Optimization}

\arxiv{
\author{Daniel Levy\thanks{ Equal contribution.} ~~~  Yair Carmon$^*$ ~~~ 
John Duchi ~~~ Aaron Sidford\\
	\texttt{\{\href{mailto:danilevy@stanford.edu}{danilevy},%
		\href{mailto:jduchi@stanford.edu}{jduchi},%
		\href{mailto:sidford@stanford.edu}{sidford}\}@stanford.edu,
		\href{mailto:ycarmon@cs.tau.ac.il}{ycarmon@cs.tau.ac.il}
	}}
\date{}
}
\notarxiv{
\author{%
	 Daniel Levy\thanks{ Equal contribution. Code is available on GitHub at 
	\url{https://github.com/daniellevy/fast-dro/}.}, Yair Carmon$^*$, John C.\ Duchi and Aaron Sidford\\
	Stanford University\\
	\texttt{\{\href{mailto:danilevy@stanford.edu}{danilevy},%
		\href{mailto:jduchi@stanford.edu}{jduchi},%
		\href{mailto:sidford@stanford.edu}{sidford}\}@stanford.edu,, \href{mailto:ycarmon@cs.tau.ac.il}{ycarmon@cs.tau.ac.il}}\\
}
}

\begin{document}

\maketitle

\arxiv{\blfootnote{Code and data are available on GitHub at 
	\url{https://github.com/daniellevy/fast-dro/}.}}

\begin{abstract}%
  We propose and analyze algorithms for distributionally robust 
  optimization of convex losses with conditional value at risk (CVaR) and 
  $\chi^2$ divergence uncertainty sets. We prove that our algorithms 
  require a number of gradient evaluations independent of training set 
  size and number of parameters, making them suitable for large-scale 
  applications. For $\chi^2$ uncertainty sets these are the first such 
  guarantees in the literature, and for CVaR our guarantees scale linearly in 
  the uncertainty level rather than quadratically as in previous work. We also 
  provide lower bounds proving the worst-case optimality of our algorithms 
  for CVaR and a penalized version of the $\chi^2$ problem. Our primary 
  technical contributions are novel bounds on the bias of batch robust risk 
  estimation and the variance of a multilevel Monte Carlo gradient estimator 
  due to~\citet{BlanchetGl15}. Experiments on MNIST and 
   ImageNet  confirm the theoretical scaling of our algorithms, 
  which are 9--36 times more efficient than full-batch methods.
\end{abstract}

\arxiv{%
\section{Introduction}

The growing role of machine learning in high-stakes decision-making raises
the need to train reliable models that perform robustly across
subpopulations and environments~\cite{BuolamwiniGe18, FusterGoRaWa18,
  TorralbaEf11, RechtRoScSh19, HendrycksDi19, OakdenDuCaRe20, KalraPa16}.
Distributionally robust optimization (DRO)~\cite{Ben-TalHeWaMeRe13,
  Shapiro17} shows promise as a way to address this challenge, with recent
interest in both the machine learning
community~\cite{SinhaNaDu18,WangGuHaCoGuJo20, DuchiNa20, StaibJe19,
  HashimotoSrNaLi18,OrenSaHaLi19} and 
 in operations
research~\cite{DelageYe10, Ben-TalHeWaMeRe13, BertsimasGuKa18,
  EsfahaniKu18}. Yet while DRO has had substantial impact in operations
research, a lack of scalable optimization methods has hindered its adoption 
in common machine learning practice. 

In contrast to empirical risk minimization (ERM), which
minimizes an expected loss $\E_{S\sim P_0}\ell(x;S)$ over $x\in\X\subset \R^d$ with 
respect to a training distribution
$P_0$, DRO minimizes the expected loss with respect to the worst 
distribution in an uncertainty set $\uset(P_0)$, that is, its goal is
to solve
\begin{equation}\label{eq:pop-dro}
  \minimize_{\x \in \X} \ \L(x; P_0)  \defeq \sup_{Q \in 
  \mc{U}(P_0)}
   \E_{S\sim Q} \ell(x; S).
\end{equation}
The literature considers several uncertainty 
sets~\cite{Ben-TalHeWaMeRe13, BertsimasGuKa18, BlanchetKaMu19, 
EsfahaniKu18}, and we focus on two particular choices: (a) the set of 
distributions with bounded likelihood ratio to $P_0$,  so that 
$\L$ becomes the conditional value at risk 
(CVaR)~\cite{RockafellarUr00,ShapiroDeRu09}, and (b) the set of 
distributions with bounded $\cs$ divergence  to 
$P_0$~\cite{Ben-TalHeWaMeRe13, Csiszar67}. 
Some of our results extend to more general $\phi$-divergence (or R\'{e}nyi
divergence) balls~\cite{ErvenHa14}.
Minimizers of these objectives enjoy favorable statistical
properties~\cite{DuchiNa20, HashimotoSrNaLi18}, but finding them is more
challenging than standard ERM. More specifically, stochastic gradient methods
solve ERM with a number of $\grad \ell$ computations independent of both $N$,
the support size of $P_0$ (i.e., number of data points), and $d$, the dimension
of $x$ (i.e., number of parameters). These guarantees do not directly apply to
DRO because the supremum over $Q$ in~\eqref{eq:pop-dro} makes cheap
sampling-based gradient estimates biased. As a consequence, existing techniques
for minimizing the $\cs$ objective~\cite{Ben-TalGhNe09, DelageYe10,
  Ben-TalHeWaMeRe13,BertsimasGuKa18,NamkoongDu16,DuchiNa20} have $\grad \ell$
evaluation complexity scaling linearly (or worse) in either $N$ or $d$, which is
prohibitive in large-scale applications.

In this paper, we consider the setting in which $\ell$ is a Lipschitz convex
loss, a prototype
case for stochastic optimization and machine 
learning~\cite{Zinkevich03, NemirovskiJuLaSh09}, and we propose
methods for solving the problem~\eqref{eq:pop-dro} with $\grad \ell$
 complexity independent of sample size $N$ and dimension $d$, 
and  with
optimal (linear) dependence on the uncertainty set size. 

Let us define the three objectives we consider. For ease of comparison to prior
work, we focus in the introduction on the case where $P_0$ is the uniform
distribution on the points $\{s_i\}_{i=1}^N$. However, our developments in the
remainder of the paper make no assumptions on $P_0$, and our results hold for
non-uniform distributions with infinite support. Let
$\simplexN \defeq \{q \in \R^N_{\ge 0} \mid \ones^T q = 1\}$ denote the
probability simplex in $\R^N$. The first first objective is the
\emph{conditional value at risk (CVaR)} at level $\alpha$, corresponds to the
uncertainty set
$\mc{U}(P_0) = \{q \in \simplexN \mid \linf{q} \le \frac{1}{\alpha N}\}$,
\begin{equation}
  \label{eqn:cvar-def}
  \Lcvar(x; P_0)
  \defeq \sup_{q \in \simplexN} 
  \crl[\bigg]{\sum_{i = 1}^N q_i \ell(x; s_i)~\textrm{s.t.}~\linf{q} \le 
  \tfrac{1}{\alpha N}}
  = \inf_{\eta \in \R}\left\{\frac{1}{\alpha N } \sum_{i=1}^N 
    \hinge{\loss(x; s_i) - \eta}  + \eta \right\},
\end{equation}
where the equality is a standard
duality relationship~\cite{Ben-TalHeWaMeRe13, ShapiroDeRu09}.
The second is the $\chi^2$-constrained objective, where
 the $\chi^2$ divergence is $\divcs(Q,P)
= \half \int (\frac{\d Q}{\d P} - 1)^2 \d P$. For $q\in \Delta^N$ we slightly overload notation to write
\begin{equation*}
  \divcs(q) \defeq \divcs\prn*{\sum_{i=1}^N q_i \delta_{s_i}, P_0}
  = \frac{1}{2 N} \sum_{i = 1}^N (N q_i - 1)^2,
\end{equation*}
so that 
$\mc{U}(P_0) = \{q \in \simplexN \mid \divcs(q) \le \rho\}$ for a
constraint $\rho \ge 0$, and
the \emph{$\cs$-constrained} objective is
\begin{equation}
  \label{eqn:chi-square-def}
  \Lcs(x; P_0) \defeq \sup_{q \in \simplexN} \bigg\{ \sum_{i = 1}^N
  q_i \loss(x; s_i) ~ \mbox{s.t.}~ \divcs(q) \le \rho \bigg\}.
\end{equation}
Finally, the \emph{penalized $\chi^2$ objective} replaces the hard
constraint~\eqref{eqn:chi-square-def} with regularization,
\begin{equation}
  \label{eqn:chi-square-reg-def}
  \Llam(x; P_0)
  \defeq \sup_{q \in \simplexN} \bigg\{ \sum_{i = 1}^N q_i \loss(x; s_i)
  - \lambda \divcs(q) \bigg\}.
\end{equation}

We develop sampling-based algorithms for each of the
objectives~\eqref{eqn:cvar-def}--\eqref{eqn:chi-square-reg-def}. In 
Table~\ref{table:summary} we
summarize their complexities and compare them to previous work.
Each entry of the table shows the number of (sub)gradient evaluations to
obtain a point with optimality gap $\epsilon$; for reference, recall that for 
ERM the stochastic subgradient method requires order $\epsilon^{-2}$ 
evaluations, independent of $d$ and $N$.  We discuss related work further in
Section~\ref{sec:related} after outlining our approach.

\begin{table}\label{table:summary}
  \begin{center}
    \ifdefined\usejohntable
    \begin{tabular}{cccc}
      \toprule
      & CVaR at level $\alpha$
      & $\cs$ constraint $\rho$
      & $\cs$ penalty $\lambda$\\
      \midrule
      Objective
      & $\Lcvar$~\eqref{eqn:cvar-def}
      & $\Lcs$~\eqref{eqn:chi-square-def}
      & $\Llam$~\eqref{eqn:chi-square-reg-def} \\
      \midrule
      Subgradient method
      & $N\eps^{-2}$
      & $N\eps^{-2}$
      & $N\eps^{-2}$ \\
      Dual SGM~[Appendix~\ref{app:dual-sgm}]
      & $\alpha^{-2}\epsilon^{-2}$
      & -
      & $\lambda^{-2}\epsilon^{-2}$ \\
      Subsampling~\cite{DuchiNa20}
      & -
      & $\rho^2d\epsilon^{-4}$
      & - \\
      Stoch.\ primal-dual~\cite{CuriLeJeKr19,NamkoongDu16}
      & $N\epsilon^{-2}$
      & $N\rho\epsilon^{-2}$
      & - \\ \midrule
      \textbf{Ours}
      & $\alpha^{-1} \epsilon^{-2}$ (Thm.~\ref{thm:ml})
      & $\rho\epsilon^{-3}$ (Thm.~\ref{thm:doubling})
      & $\lambda^{-1}\epsilon^{-2}$ (Thm.~\ref{thm:ml}) \\
      Lower Bound
      & $\alpha^{-1}\epsilon^{-2}$ (Thm.~\ref{thm:lb})
      & $\rho\epsilon^{-2}$~\cite{DuchiNa20}
      & $\lambda^{-1}\epsilon^{-2}$ (Thm.~\ref{thm:lb})\\
      \bottomrule
    \end{tabular}
    \else
    \begin{tabular}{cccc}
      \toprule
      & CVaR at level $\alpha$
      & $\cs$ constraint $\rho$
      & $\cs$ penalty $\lambda$\\
      \midrule
      Objective $\L(x; P_0)=$
      & $\displaystyle{\sup_{\norm{q}_\infty \le \frac{1}{\alpha N}}} 
      q^\top\ell(\x)$
      & $\displaystyle{\sup_{\divcs(q) \le \rho}} 
      q^\top\ell(\x)$
      & $\displaystyle{\sup_{q\in\Delta^N}} {q^\top\ell(\x) - 
        \lambda\divcs\prn*{q}}$\\
      \midrule
      Subgradient method
      & $N\eps^{-2}$
      & $N\eps^{-2}$
      & $N\eps^{-2}$ \\
      Dual SGM~[Appendix~\ref{app:dual-sgm}]
      & $\alpha^{-2}\epsilon^{-2}$
      & -
      & $\lambda^{-2}\epsilon^{-2}$ \\
      Subsampling~\cite{DuchiNa20}
      & -
      & $\rho^2d\epsilon^{-4}$
      & - \\
      Stoch.\ primal-dual~\cite{CuriLeJeKr19,NamkoongDu16}
      & $N\epsilon^{-2}$
      & $N\rho\epsilon^{-2}$
      & - \\ \midrule
      \textbf{Ours}
      & $\alpha^{-1} \epsilon^{-2}$ (Thm.~\ref{thm:ml})
      & $\rho\epsilon^{-3}$ (Thm.~\ref{thm:doubling})
      & $\lambda^{-1}\epsilon^{-2}$ (Thm.~\ref{thm:ml}) \\
      Lower Bound
      & $\alpha^{-1}\epsilon^{-2}$ (Thm.~\ref{thm:lb})
      & $\rho\epsilon^{-2}$~\cite{DuchiNa20}
      & $\lambda^{-1}\epsilon^{-2}$ (Thm.~\ref{thm:lb})\\
      \bottomrule
    \end{tabular}
    \fi
    \vspace{3pt}
    \caption{Number of $\grad \ell$ evaluations to obtain $\E[\L(\x;P_0)] -
      \inf_{x'\in\X}\L(\x';P_0) \le \epsilon$ when $P_0$ is uniform on $N$
      training points. For simplicity we omit the Lipschitz constant of
      $\ell$, the size of the domain $\X$, and logarithmic factors.  
    }
\end{center}
\end{table}

We employ two gradient estimation strategies; 
the first uses a biased subsampling approximation to the objective
$\L$, and the second uses an essentially unbiased
 multi-level Monte Carlo~\cite{Giles08, Giles15}
gradient estimator. We begin by describing the former, which we develop in
Section~\ref{sec:batch}. Let $\what{P}_n$ be uniform distribution on a
random mini-batch of size $n$ (typically much smaller than $N$) sampled 
i.i.d.\ 
from $P_0$, and define the
surrogate objective $\bL(x;n)=\E \L(x;\what{P}_n)$, where the expectation 
is over the mini-batch samples.
In contrast to the full objective~\eqref{eq:pop-dro}, 
it is straightforward to 
obtain unbiased gradient estimates for $\bL$---using the mini-batch 
estimator $\grad \L(x;\what{P}_n)$---%
and to optimize it efficiently with stochastic gradient methods. 

We establish that $\bL$ is a useful surrogate for $\L$ by proving uniform 
bounds on the
error $|\L(x;P_0)-\bL(x;n)|$. For CVaR~\eqref{eqn:cvar-def} we prove a bound
scaling as $1/\sqrt{n}$ and extend it to other objectives, 
including~\eqref{eqn:chi-square-def}, via
the Kusuoka representation~\cite{Kusuoka01}. Notably, for the penalty
version of the $\cs$ objective~\eqref{eqn:chi-square-reg-def}  
we prove a stronger bound scaling
as $1/n$.

This analysis implies that, for large enough mini-batch size $n$, an
$\frac{\epsilon}{2}$-minimizer of $\bL$ is also an $\epsilon$-minimizer of
$\L$. Further, for CVaR and the $\cs$ penalized objective, we show that the
variance of the gradient estimator decreases as $1/n$, and we use Nesterov
acceleration to decrease the required number of (stochastic) gradient steps.

To obtain algorithms with improved oracle complexities, in 
Section~\ref{sec:multilevel} we present a theoretically more efficient 
multi-level Monte Carlo (MLMC)~\cite{Giles08, Giles15}  gradient estimator 
which is a slight modification of the general technique 
of~\citet{BlanchetGl15}. The resulting estimator is unbiased for $\grad 
\bL(x;n)$ but requires only a \emph{logarithmic number of samples} in $n$ 
in expectation. (In contrast, the above-mentioned mini-batch estimator 
requires $n$ samples). For CVaR and $\cs$ penalty we control the second 
moment 
of the gradient estimator, resulting in complexity bounds scaling with 
$\epsilon^{-2}$. In Section~\ref{sec:lowerbounds} we prove that these 
rates  are worst-case optimal up to logarithmic factors.

Unfortunately, direct application of the MLMC estimator for the $\cs$-constrained objective~\eqref{eqn:chi-square-def} 
demonstrably 
fails to achieve a second moment bound. Instead, in~\Cref{sec:doubling} 
we optimize its 
Lagrange dual---the $\cs$ 
penalty---with respect to $x$ and Lagrange multiplier $\lambda$. Using a 
 doubling 
scheme on the $\lambda$ domain, we obtain a complexity guarantee 
scaling as 
$\epsilon^{-3}$. 

Section~\ref{sec:experiments} presents experiments where we use
DRO to train linear models for digit classification (on a mixture between
MNIST~\cite{LeCunEtAl95} and typed digits~\cite{deCamposBaVa09}), and
ImageNet~\cite{RussakovskyDeSuKrSaMaHuKaKhBeBeFe15}.  To the best of 
our knowledge, the latter is the largest DRO problem solved to date.
 In both experiments DRO provides generalization improvements over
ERM, and we show that our stochastic gradient estimators require far fewer
$\grad \ell$ computations---between 9$\times$ and 36$\times$---than 
full-batch
methods.
Our experiments also reveal two facts that our theory only hints at. First,
using the mini-batch gradient estimator the error 
due to the difference
between $\bL(x;n)$ and $\L(x;P_0)$ becomes negligible even for batch sizes as
small as 10. Second, while the MLMC estimator avoids these errors 
altogether, its increased variance makes it practically inferior to the
mini-batch estimator with properly tuned batch size and learning rate. Our 
code, which is available at \url{https://github.com/daniellevy/fast-dro/},
implements our gradient estimators in
PyTorch~\cite{PaszkeGrChChYaDeLiDeAnLe17} and combines them seamlessly with the
framework's optimizers; we show an example code snippet in
Appendix~\ref{app:experiments-pytorch}. 

We conclude the paper in Section~\ref{sec:conclusion} with some remarks 
and directions for future research.

\subsection{Related work}\label{sec:related}
Distributionally robust optimization grows from the robust optimization
literature in operations research~\cite{Ben-TalHeWaMeRe13, Ben-TalGhNe09,
  BertsimasBrCa11, BertsimasGuKa18}, and the fundamental uncertainty about
the data distribution at test time makes its application to machine learning
natural. Experiments in the papers \cite{NamkoongDu16, FanLyYiHu17,
  DuchiNa20, HashimotoSrNaLi18, CuriLeJeKr19, KawaguchiLu20} show promising
results for CVaR~\eqref{eqn:cvar-def} and
$\cs$-constrained~\eqref{eqn:chi-square-def} DRO, while other works
highlight the importance of incorporating additional constraints into the
uncertainty set
definition~\cite{HuNiSaSu18,DuchiHaNa20,OrenSaHaLi19,SagawaKoHaLi20}.
Below, we review the
prior art on solving these DRO problems at scale.
\paragraph{Full-batch subgradient method.}
When $P_0$ has support of size $N$ it is possible to compute a 
subgradient of the objective $\L(x;P_0)$ by evaluating $\ell(x; s_i)$ and 
$\grad \ell(x; s_i)$ for $i = 1, \ldots, N$, computing the
$q \in \simplexN$ attaining the supremum~\eqref{eq:pop-dro},
whence $g = \sum_{i = 1}^N q_i \grad \loss(x; s_i)$ is a subgradient
of $\L$ at $x$.
As the Lipschitz constant of $\L$ is at most that of $\ell$, we may use
these subgradients in the subgradient method~\cite{Nesterov04} and find an
$\epsilon$ approximate solution in order $\epsilon^{-2}$ steps. This 
requires order $N\epsilon^{-2}$ evaluations of $\grad \ell$, regardless of the
uncertainty set.

\paragraph{CVaR.}
Robust objectives of the form~\eqref{eq:pop-dro} often admit tractable 
expression in terms of joint minimization over $x$ and the Lagrange 
multipliers associated with the constrained maximization over 
$Q$~\cite[e.g.,][]{RockafellarUr00,Shapiro17}.
For CVaR, this dual formulation (the second
equality~\eqref{eqn:cvar-def}) is an ERM problem in
$x$ and $\eta\in \R$, which we can solve in 
time independent of $N$ using stochastic gradient methods. We refer to 
this as ``dual SGM,'' providing the associated complexity 
bounds in Appendix~\ref{app:dual-sgm}.
\citet{FanLyYiHu17} apply dual SGM for learning linear classifiers, and  
\citet{CuriLeJeKr19} compare it to their proposed stochastic primal-dual
method based on determinantal point processes. 
While the latter performs better in practice, its 
worst-case guarantees scale roughly as $N\eps^{-2}$, similarly to the full-batch
method. \citet{KawaguchiLu20} propose to only use gradients from the 
highest $k$
losses in every batch, which is essentially identical to our mini-batch
estimator for CVaR; they do not, however, relate their algorithm to CVaR
optimization. We contribute to this line of work by obtaining tight
characterizations of the mini-batch and MLMC gradient estimators, resulting in
optimal complexity bounds scaling as $\alpha^{-1}\epsilon^{-2}$.

\paragraph{DRO with $\cs$ divergence.}
Similar dual formulations exist for both the constrained and penalized $\cs$
objectives~\eqref{eqn:chi-square-def} and~\eqref{eqn:chi-square-reg-def},
and dual SGM provides similar guarantees to CVaR for the penalized $\cs$
objective~\eqref{eqn:chi-square-reg-def}.  For the constrained
problem~\eqref{eqn:chi-square-def}, the additional Lagrange multiplier
associated with the constraint induce a so-called ``perspective
transform''~\cite{Ben-TalHeWaMeRe13, DuchiNa20}, making the method
unstable. Indeed, \citet{NamkoongDu16} report that
it fails to converge in practice and instead propose a stochastic
primal-dual method with convergence rate $(1+\rho N)\epsilon^{-2}$. Their
guarantee is optimal in the weak regularization regime where $\rho \lesssim
1/N$ , but is worse than the full-batch method in the
setting where $\rho \gtrsim 1$.  \citet{HashimotoSrNaLi18} propose a
different scheme alternating between ERM on $x$ and line search over a
Lagrange multiplier, but do not provide complexity bounds.
\citet{DuchiNa20} prove that for a sample of size $N' \approx \rho^2 d 
\epsilon^{-2}$ the empirical objective %
 converges to $\L(x; P_0)$ uniformly in $x\in\X$; substituting 
 $N'$ into the full-batch complexity bound implies a rate of $\rho^2 d 
 \epsilon^{-4}$. This  guarantee is independent of $N$, but features an 
 undesirable dependence on $d$.
\citet{GhoshSqWo18} use the mini-batch gradient estimator and 
gradually increase the batch size to $N$ as optimization progresses; they 
do not provide convergence rate bounds. We establish concrete 
rates for fixed batch sizes independent of $N$.

\paragraph{MLMC gradient estimators.}
Multi-level Monte Carlo techniques~\cite{Giles08,Giles15} facilitate the 
estimation 
of expectations of the form $\E \funct(S_1,\ldots,S_n)$, where the $S_i$ 
are i.i.d. %
 In this work we leverage a variant of a particular MLMC 
estimator proposed by~\citet{BlanchetGl15}. 
Prior work \cite{BlanchetKa20} uses
the estimator of~\cite{BlanchetGl15} in a DRO formulation of 
semi-supervised learning with Wasserstein uncertainty sets 
and $\funct(\cdot)$  a ratio of expectations,
as opposed to a supremum of expectations in our setting.
}
\notarxiv{%
\section{Introduction}

The growing role of machine learning in high-stakes decision-making raises
the need to train reliable models that perform robustly across
subpopulations and environments~\cite{BuolamwiniGe18, FusterGoRaWa18,
  TorralbaEf11, RechtRoScSh19, HendrycksDi19, OakdenDuCaRe20, KalraPa16}.
Distributionally robust optimization (DRO)~\cite{Ben-TalHeWaMeRe13,
  Shapiro17} shows promise as a way to address this challenge, with recent
interest in both the machine learning
community~\cite{SinhaNaDu18,WangGuHaCoGuJo20, DuchiNa20, StaibJe19,
  HashimotoSrNaLi18,OrenSaHaLi19} and 
 in operations
research~\cite{DelageYe10, Ben-TalHeWaMeRe13, BertsimasGuKa18,
  EsfahaniKu18}. Yet while DRO has had substantial impact in operations
research, a lack of scalable optimization methods has hindered its adoption 
in common machine learning practice. 

In contrast to empirical risk minimization (ERM), which
minimizes an expected loss $\E_{S\sim P_0}\ell(x;S)$ over $x\in\X\subset \R^d$ with 
respect to a training distribution
$P_0$, DRO minimizes the expected loss with respect to the worst 
distribution in an uncertainty set $\uset(P_0)$, that is, its goal is
to solve
\begin{equation}\label{eq:pop-dro}
  \minimize_{\x \in \X} \ \L(x; P_0)  \defeq \sup_{Q \in 
  \mc{U}(P_0)}
   \E_{S\sim Q} \ell(x; S).
\end{equation}
The literature considers several uncertainty 
sets~\cite{Ben-TalHeWaMeRe13, BertsimasGuKa18, BlanchetKaMu19, 
EsfahaniKu18}, and we focus on two particular choices: (a) the set of 
distributions with bounded likelihood ratio to $P_0$,  so that 
$\L$ becomes the conditional value at risk 
(CVaR)~\cite{RockafellarUr00,ShapiroDeRu09}, and (b) the set of 
distributions with bounded $\cs$ divergence  to 
$P_0$~\cite{Ben-TalHeWaMeRe13, Csiszar67}. 
Some of our results extend to more general $\phi$-divergence (or R\'{e}nyi
divergence) balls~\cite{ErvenHa14}.
Minimizers of these objectives enjoy favorable statistical
properties~\cite{DuchiNa20, HashimotoSrNaLi18}, but finding them is more
challenging than standard ERM. More specifically, stochastic gradient methods
solve ERM with a number of $\grad \ell$ computations independent of both $N$,
the support size of $P_0$ (i.e., number of data points), and $d$, the dimension
of $x$ (i.e., number of parameters). These guarantees do not directly apply to
DRO because the supremum over $Q$ in~\eqref{eq:pop-dro} makes cheap
sampling-based gradient estimates biased. As a consequence, existing techniques
for minimizing the $\cs$ objective~\cite{Ben-TalGhNe09, DelageYe10,
  Ben-TalHeWaMeRe13,BertsimasGuKa18,NamkoongDu16,DuchiNa20} have $\grad \ell$
evaluation complexity scaling linearly (or worse) in either $N$ or $d$, which is
prohibitive in large-scale applications.

In this paper, we consider the setting in which $\ell$ is a Lipschitz convex
loss, a prototype
case for stochastic optimization and machine learning~\cite{Zinkevich03,
  NemirovskiJuLaSh09}, and we propose methods for solving the
problem~\eqref{eq:pop-dro} with $\grad \ell$ complexity independent of sample
size $N$ and dimension $d$, and with optimal (linear) dependence on the
uncertainty set size. In Table~\ref{table:summary} we summarize their
complexities and compare them to previous work.  Each entry of the table shows
the number of (sub)gradient evaluations to obtain a point with optimality gap
$\epsilon$; for reference, recall that for ERM the stochastic subgradient method
requires order $\epsilon^{-2}$ evaluations, independent of $d$ and $N$.  We
discuss related work further in Section~\ref{sec:related} after outlining our
approach.

\begin{table}\label{table:summary}
  \begin{center}
\begin{tabular}{cccc}
  \toprule
  & CVaR at level $\alpha$
  & $\cs$ constraint $\rho$
  & $\cs$ penalty $\lambda$\\
  \midrule
   Objective $\L(x; P_0)=$
  & $\displaystyle{\sup_{\norm{q}_\infty \le \frac{1}{\alpha N}}} 
  q^\top\ell(\x)$
  & $\displaystyle{\sup_{\divcs(q) \le \rho}} 
  q^\top\ell(\x)$
  & $\displaystyle{\sup_{q\in\Delta^N}} {q^\top\ell(\x) - 
  \lambda\divcs\prn*{q}}$\\
  \midrule
  Subgradient method
  & $N\eps^{-2}$
  & $N\eps^{-2}$
  & $N\eps^{-2}$ \\
  Dual SGM~[Appendix~\ref{app:dual-sgm}]
  & $\alpha^{-2}\epsilon^{-2}$
  & -
  & $\lambda^{-2}\epsilon^{-2}$ \\
  Subsampling~\cite{DuchiNa20}
  & -
  & $\rho^2d\epsilon^{-4}$
  & - \\
  Stoch.\ primal-dual~\cite{CuriLeJeKr19,NamkoongDu16}
  & $N\epsilon^{-2}$
  & $N\rho\epsilon^{-2}$
  & - \\ \midrule
  \textbf{Ours}
  & $\alpha^{-1} \epsilon^{-2}$ (Thm.~\ref{thm:ml})
  & $\rho\epsilon^{-3}$ (Thm.~\ref{thm:doubling})
  & $\lambda^{-1}\epsilon^{-2}$ (Thm.~\ref{thm:ml}) \\
  {Lower Bound}
  & $\alpha^{-1}\epsilon^{-2}$ (Thm.~\ref{thm:lb})
  & $\rho\epsilon^{-2}$~\cite{DuchiNa20}
  & $\lambda^{-1}\epsilon^{-2}$ (Thm.~\ref{thm:lb})\\
  \bottomrule
\end{tabular}
\vspace{3pt}
\caption{Number of $\grad \ell$ evaluations to obtain
  $\E[\L(\x;P_0)] - \inf_{x'\in\X}\L(\x';P_0) \le \epsilon$ when $P_0$ is 
  uniform on $N$ training points. For simplicity we omit the Lipschitz 
  constant of $\ell$, the size of the domain $\X$, and logarithmic factors. 
  We define $\ell_i(x)\defeq\ell(x;\S_i)$ and
  $\divcs(q)  \defeq \frac{N}{2}\norm{ q - \frac{1}{N}\mathbf{1}}_2^2$. The 
  suprema are over $q$ in the simplex.  
}
\end{center}
\end{table}

We begin our development in Section~\ref{sec:batch} by considering the 
surrogate objective\arxiv{\linebreak} $\bL(x;n)=\E \L(x;\what{P}_n)$ 
corresponding to the 
average empirical robust objective over random batches of size $n$ 
sampled from $P_0$. 
In contrast to~\eqref{eq:pop-dro}, 
it is straightforward to 
obtain unbiased gradient estimates for $\bL$---using the mini-batch 
estimator $\grad \L(x;\what{P}_n)$---%
and to optimize it efficiently with stochastic gradient methods. To obtain 
guarantees for the true objective $\L$, we establish uniform bounds on the 
error $|\L(x;P_0)-\bL(x;n)|$. For CVaR we prove a bound scaling as 
$1/\sqrt{n}$ and extend it to other uncertainty sets, including $\cs$ balls, 
via 
the Kusuoka representation~\cite{Kusuoka01}. Notably, for the penalty 
version of the $\cs$ objective (Table~\ref{table:summary} right column) we 
prove a stronger bound scaling as $1/n$.
This analysis implies that, for large enough batch size $n$, an 
$\epsilon/2$-minimizer of 
 $\bL$ is also an 
$\epsilon$-minimizer of $\L$. Furthermore, for CVaR and $\cs$ penalty we 
show that the variance of the gradient estimator decreases as $1/n$, and 
we use Nesterov acceleration to decrease the required number of gradient 
steps.

To obtain stronger guarantees, in Section~\ref{sec:multilevel} 
we present a theoretically more efficient multi-level Monte 
Carlo (MLMC)~\cite{Giles08, Giles15}  gradient estimator which is a slight 
modification 
of the general technique of~\citet{BlanchetGl15}. The resulting estimator is 
unbiased for 
$\grad \bL(x;n)$ but requires only a \emph{logarithmic number of 
samples} in $n$. For CVaR and $\cs$ penalty we control the second 
moment 
of 
the gradient estimator, resulting in complexity bounds scaling with 
$\epsilon^{-2}$. We further prove that these 
rates  are worst-case optimal up to logarithmic factors.

Unfortunately, direct application of the MLMC estimator for the $\cs$ 
uncertainty set (Table~\ref{table:summary} center column) demonstrably fails for 
certain inputs. Instead, in~\Cref{sec:doubling} we optimize its 
Lagrange dual---the $\cs$ 
penalty---with respect to $x$ and Lagrange multiplier $\lambda$. Using a 
 doubling 
scheme on the $\lambda$ domain, we obtain a complexity guarantee 
scaling as 
$\epsilon^{-3}$.

Section~\ref{sec:experiments} presents experiments where we use
DRO to train linear models for digit classification (on a mixture between
MNIST~\cite{LeCunEtAl95} and typed digits~\cite{deCamposBaVa09}), and
ImageNet~\cite{RussakovskyDeSuKrSaMaHuKaKhBeBeFe15}.  To the best of 
our knowledge, the latter is the largest DRO problem solved to date.
 In both experiments DRO provides generalization improvements over
ERM, and we show that our stochastic gradient estimators require far less
$\grad \ell$ computations---between 9$\times$ and 40$\times$--- than full-batch
methods.
Our experiments also reveal two facts that our theory only hints at. First,
using the mini-batch gradient estimator the error floor due to the difference
between $\bL(x;n)$ and $\L(x;P_0)$ becomes negligible even for batch sizes as
small as 10. Second, while the MLMC estimator avoids these error floors
altogether, its increased variance makes it practically inferior to the
mini-batch estimator with properly tuned batch size and learning rate. Our 
code implements our gradient estimators in
PyTorch~\cite{PaszkeGrChChYaDeLiDeAnLe17} and combines them seamlessly with the
framework's optimizers; we show an example code snippet in
Appendix~\ref{app:experiments-pytorch}. 

\subsection{Related work}\label{sec:related}
Distributionally robust optimization grows from the robust optimization
literature in operations research~\cite{Ben-TalHeWaMeRe13, Ben-TalGhNe09,
  BertsimasBrCa11, BertsimasGuKa18}, and the fundamental uncertainty about
the data distribution at test time makes its application to machine learning
natural. Experiments in the papers \cite{NamkoongDu16, FanLyYiHu17,
  DuchiNa20, HashimotoSrNaLi18, CuriLeJeKr19, KawaguchiLu20} show promising
results for CVaR and
$\cs$-constrained DRO, while other works
highlight the importance of incorporating additional constraints into the
uncertainty set
definition~\cite{HuNiSaSu18,DuchiHaNa20,OrenSaHaLi19,SagawaKoHaLi20}.
Below, we review the
prior art on solving these DRO problems at scale.
\paragraph{Full-batch subgradient method.}
When $P_0$ has support of size $N$ it is possible to compute a 
subgradient of the objective $\L(x;P_0)$ by evaluating $\ell(x; s_i)$ and 
$\grad \ell(x; s_i)$ for $i = 1, \ldots, N$, computing the
$q \in \simplexN$ attaining the supremum~\eqref{eq:pop-dro},
whence $g = \sum_{i = 1}^N q_i \grad \loss(x; s_i)$ is a subgradient
of $\L$ at $x$.
As the Lipschitz constant of $\L$ is at most that of $\ell$, we may use
these subgradients in the subgradient method~\cite{Nesterov04} and find an
$\epsilon$ approximate solution in order $\epsilon^{-2}$ steps. This 
requires order $N\epsilon^{-2}$ evaluations of $\grad \ell$, regardless of the
uncertainty set.

\paragraph{CVaR.}
Robust objectives of the form~\eqref{eq:pop-dro} often admit tractable 
expression in terms of joint minimization over $x$ and the Lagrange 
multipliers associated with the constrained maximization over 
$Q$~\cite[e.g.,][]{RockafellarUr00,Shapiro17}.
For CVaR, this dual formulation is an ERM problem in
$x$ and $\eta\in \R$, which we can solve in 
time independent of $N$ using stochastic gradient methods. We refer to 
this as ``dual SGM,'' providing the associated complexity 
bounds in Appendix~\ref{app:dual-sgm}.
\citet{FanLyYiHu17} apply dual SGM for learning linear classifiers, and  
\citet{CuriLeJeKr19} compare it to their proposed stochastic primal-dual
method based on determinantal point processes. 
While the latter performs better in practice, its 
worst-case guarantees scale roughly as $N\eps^{-2}$, similarly to the full-batch
method. \citet{KawaguchiLu20} propose to only use gradients from the 
highest $k$
losses in every batch, which is essentially identical to our mini-batch
estimator for CVaR; they do not, however, relate their algorithm to CVaR
optimization. We contribute to this line of work by obtaining tight
characterizations of the mini-batch and MLMC gradient estimators, resulting in
optimal complexity bounds scaling as $\alpha^{-1}\epsilon^{-2}$.

\paragraph{DRO with $\cs$ divergence.}
Similar dual formulations exist for both the constrained and penalized $\cs$
objectives,
and dual SGM provides similar guarantees to CVaR for the penalized $\cs$
objective. For the constrained\notarxiv{-$\cs$}
problem\arxiv{~\eqref{eqn:chi-square-def}}, the additional Lagrange multiplier
associated with the constraint induces a ``perspective
transform''~\cite{Ben-TalHeWaMeRe13, DuchiNa20}, making the method
unstable. Indeed, \citet{NamkoongDu16} report that
it fails to converge in practice and instead propose a stochastic
primal-dual method with convergence rate $(1+\rho N)\epsilon^{-2}$. Their
guarantee is optimal in the weak regularization regime where $\rho \lesssim
1/N$ , but is worse than the full-batch method in the
setting where $\rho \gtrsim 1$.  \citet{HashimotoSrNaLi18} propose a
different scheme alternating between ERM on $x$ and line search over a
Lagrange multiplier, but do not provide complexity bounds.
\citet{DuchiNa20} prove that for a sample of size $N' \approx \rho^2 d 
\epsilon^{-2}$ the empirical objective %
 converges to $\L(x; P_0)$ uniformly in $x\in\X$; substituting 
 $N'$ into the full-batch complexity bound implies a rate of $\rho^2 d 
 \epsilon^{-4}$. This  guarantee is independent of $N$, but features an 
 undesirable dependence on $d$.
\citet{GhoshSqWo18} use the mini-batch gradient estimator and 
gradually increase the batch size to $N$ as optimization progresses; they 
do not provide convergence rate bounds. We establish concrete 
rates for fixed batch sizes independent of $N$.

\paragraph{MLMC gradient estimators.}
Multi-level Monte Carlo techniques~\cite{Giles08,Giles15} facilitate the 
estimation 
of expectations of the form $\E \funct(S_1,\ldots,S_n)$, where the $S_i$ 
are i.i.d. %
 In this work we leverage a variant of a particular MLMC 
estimator proposed by~\citet{BlanchetGl15}. 
Prior work \cite{BlanchetKa20} uses
the estimator of~\cite{BlanchetGl15} in a DRO formulation of 
semi-supervised learning with Wasserstein uncertainty sets 
and $\funct(\cdot)$  a ratio of expectations,
as opposed to a supremum of expectations in our setting. %
}

\vspace{-0.2cm} %

\section{Preliminaries}\label{sec:prelims}

We collect notation, establish a few assumptions, and provide the most
important definitions for the remainder of the paper in this section.

\paragraph{Notation.} 
We denote the optimization variable by $\x\in\R^d$, and use $\s$  (or 
$\S$ when it is random) for a 
data sample in $\ss$.
We use $z_l^m$ as shorthand for the 
sequence $z_l,\ldots,z_m$. For fixed $x$ we denote the cdf of $\ell(x,S)$ by
$F(t) \defeq\P(\ell(x,S)\le t)$ and its inverse by $F^{-1}(u) \defeq 
\inf\crl{t:F(t) > u}$, leaving the dependence on $x$ and $P_0$ implicit.
We use $\norm{\cdot}$ to denote Euclidean norm, but remark that many of 
our results carry over to general norms. 
We let $\Delta^m$ denote the simplex in $m$ dimensions.
We write $\indic{A}$ for the indicator of event $A$, i.e., 1 if $A$ holds and 0
otherwise, and write $\convind_{\mc{C}}$ for the infinite indicator of the set
$\mc{C}$, $\convind_{\mc{C}}(x)=0$ if $x\in \mc{C}$ and
$\convind_{\mc{C}}(x)=\infty$ otherwise. The Euclidean projection to a set
$\mc{C}$ is $\Pi_{\mc{C}}$. We use $\grad$ to denote gradient with respect to
$x$, or, for non-differentiable convex functions, an arbitrary subgradient. We denote the positive part of 
$t\in\R$  by $\hinge{t}\defeq \max\crl{t,0}$. Finally, $f \lesssim g$ means
that there exists $C\in\R_+$, independent of any problem parameters, such that
$f\le Cg$ holds; we also write $f\asymp g$ if
$f\lesssim g \lesssim f$.

\paragraph{Assumptions.}
Throughout, we assume that the domain $\X$ is closed convex and satisfies
$\norm{x-y}\le R$ for all $x,y\in\X$. Moreover, we assume the loss function
$\ell:\X\times\ss\to [0, B]$ is convex and $G$-Lipschitz in $\x$, i.e.,
$0\le \ell(x,s)\le B$ and $|\ell(x;s)-\ell(y;s)|\le G\norm{x-y}$ for
$x,y\in\X$ and $\s\in\ss$.\footnote{Our results hold also when $B$ denotes
$\sup_{x\in\X,s,s'\in\ss} \crl{\ell(x;s)-\ell(x;s')}$. The Lipschitz loss  
and bounded domain assumptions imply $B\le B_0 + GR$ if 
$\inf_{x\in\X}\ell(x;s)-\inf_{x'\in\X}\ell(x';s')\le B_0$ for all $s,s'\in\ss$, 
which typically holds with $B_0\approx 0$  in regression and classification 
problems.} 
In 
some cases,
we entertain two additional assumptions:
\begin{assumption}
  \label{assumption:grad-lipschitz}
  The gradient $\grad \ell(x, s)$ is $H$-Lipschitz in $x$.
\end{assumption}
\begin{assumption}
  \label{assumption:smooth-icdf}
  The inverse cdf $F^{-1}$ of $\ell(x; S)$ is $\icdfLip$-Lipschitz
  for each $x \in \X$.
\end{assumption}
\noindent
Most of our bounds do not require 
Assumptions~\ref{assumption:grad-lipschitz} and~\ref{assumption:smooth-icdf}.
Moreover, 
in Appendix~\ref{app:batch-extra-assumptions} we argue 
that these assumptions are frequently not restrictive.

\paragraph{The distributionally robust objective.}
We consider a slight generalization of $\phi$-divergence
distributionally robust optimization (DRO).
For a convex $\phi : \R_+ \to \R \cup \{+\infty\}$
satisfying $\phi(1) = 0$,
the $\phi$-divergence between distributions $P$ and
$Q$ absolutely continuous w.r.t.\ $P$ by
\arxiv{
\begin{equation*}
  \phidiv{Q}{P}
  \defeq \int \phi\prn[\bigg]{\frac{\d Q}{\d P}(s) } \d P(s).
\end{equation*}}
\notarxiv{$\phidiv{Q}{P} \defeq \int \phi\prn{\tfrac{\d Q}{\d P}(s)}\d P(s)$.}
Then, for convex $\phi, \psi$ with $\phi(1) = \psi(1) = 0$, a constraint
radius $\rho\ge0$, and penalty $\lambda \ge 0$ \emph{the general form of the objectives we consider} is
\begin{equation}\label{eq:pop-dro-reg}
  \L(\x;P) \defeq \sup_{Q:\div_\phi(Q, P) \le 
  \rho} \crl[\Big]{\E_Q\brk{\ell(\x;\S)} - \lambda\div_\psi(Q, P)}.
\end{equation}

\arxiv{The form~\eqref{eq:pop-dro-reg} allows us to redefine the
  objectives~\eqref{eqn:cvar-def}--\eqref{eqn:chi-square-reg-def} for general
  $P_0$ (nonuniform and with infinite support):} \notarxiv{As previewed, we
  consider the following objectives for general $P_0$ (nonuniform with infinite
  support):}
 \begin{itemize}[leftmargin=*]
 \item \textbf{$\cs$ constraint.} $\Lcs$ corresponds to $\phi(t) = \cs(t) 
 \defeq \frac{1}{2}(t-1)^2$ and
   $\psi = 0$.
 \item \textbf{$\cs$ penalty.} $\Llam$ corresponds to $\phi = 0$ and
   $\psi(t) = \cs(t) = \frac{1}{2}(t-1)^2$.
 \item \textbf{Conditional value at risk $\boldsymbol{\alpha\in(0,1]}$ 
 (CVaR).} $\Lcvar$ 
 corresponds to 
 $\phi=0$ and 
 $\psi = \convind_{[0, 1/\alpha)}$.
\end{itemize}
\noindent Additionally, define the following smoothed version of the CVaR objective, which we use in Section~\ref{sec:batch}.
\begin{itemize}[leftmargin=*]
 \item \textbf{KL-regularized CVaR.} $\LcvarR$ corresponds to $\phi=0$ 
 and  and $\psi(t) = \convind_{[0, 1/\alpha]}(t) + t\log t -t +1$. 
 \end{itemize}
 \noindent
In Appendix~\ref{app:prelims} we present additional standard 
formulations and useful properties of these objectives.

With mild
abuse of notation, for a sample $\s_1^n\in\ss^n$, we let
\begin{equation}\label{eq:finite-rob}
\L(x;s_1^n) \defeq \L(x; \Phat{s_1^n}) =
\sup_{q\in\Delta^n: \sum_{i\le n}\frac{1}{n}\phi(nq_i) \le \rho}
\bigg\{\sum_{i = 1}^n \prn*{q_i \ell(\x;s_i) - \tfrac{1}{n}\psi(nq_i)} \bigg\}
\end{equation}
denote the loss with respect to the empirical distribution on $s_1^n$. 
Averaging the robust objective over random batches of size $n$, we define the surrogate objective
\begin{equation}\label{eq:batch-obj}
\bL(x;n) \defeq\E_{S_1^n \sim P_0^n}\L(x; S_1^n).
\end{equation}

\paragraph{Complexity metrics.}
We measure complexity of our methods by the number of
computations of $\grad \loss(x; s)$
they require to reach a solution with accuracy
$\epsilon$. We can bound (up to a constant factor) 
the runtime of every method we consider 
by our 
complexity measure multiplied by $d+\mathsf{T}_{\mathrm{eval}}$, where
$\mathsf{T}_{\mathrm{eval}}$ denotes the time to evaluate $\ell(x;s)$ and
$\grad\ell(x;s)$ at a single point $x$ and sample $s$, and is typically
$O(d)$. (In the problems we study, solving the 
problem~\eqref{eq:batch-obj} given $\ell(x;S_1^n)$  takes $O(n\log n)$ 
time; see Appendix~\ref{app:prelims-compute}).

\section{Mini-batch gradient estimators}\label{sec:batch}

In this section, we develop and analyze stochastic subgradient methods 
using the subgradients of the mini-batch loss~\eqref{eq:finite-rob}. That 
is, we estimate $\grad \L(x;P_0)$ by sampling a mini-batch $S_1,\ldots, 
S_n \simiid P_0$ and computing
\begin{equation*}
\grad \L(x;S_1^n) = \sum_{i=1}^n q_i\opt \grad \ell(x;S_i),
\end{equation*}
where $q\opt \in\Delta^n$ attains the supremum in 
Eq.~\eqref{eq:finite-rob}. By definition~\eqref{eq:batch-obj} of the 
surrogate objective $\bL$, we have that $\E \grad \L(x;S_1^n) = \grad 
\bL(x;n)$. Therefore, we expect stochastic subgradient methods using 
$\grad \L(x;S_1^n)$ to minimize $\bL$. However, in general, $\bL(x;n) \ne \L(x;P_0)$ 
and  
$\E \grad \L(x;S_1^n) \ne \grad \L(x;P_0)$.

To show that the 
mini-batch gradient estimator is nevertheless effective for minimizing 
$\L$, 
we proceed in three steps. %
First, 
in Section~\ref{sec:batch-bias} we prove 
uniform bounds on the bias $\L-\bL$ that tend to zero with $n$. Second, 
in 
Section~\ref{sec:batch-variance} we complement them with $1/n$ variance 
bounds on $\grad \L(x;S_1^n)$. Finally, Section~\ref{sec:batch-rates} puts 
the pieces together: we 
apply the SGM guarantees to bound the complexity of minimizing $\bL$ to 
accuracy $\epsilon/2$, using Nesterov acceleration to 
exploit our variance bounds, and choose the mini-batch size $n$ large 
enough to guarantee (via our bias bounds) that the resulting solution is 
also an $\epsilon$ minimizer of the original objective $\L$.

\subsection{Bias analysis}\label{sec:batch-bias}

\begin{restatable}[Bias of the batch estimator]{proposition}
  {restatePropBatchBias}\label{prop:batch-bias}
  For all $x\in\X$ and $n\in\N$ we have
    \begin{empheq}[left={0 \le \L(x;P_0)-\bL(x;n) \lesssim \empheqlbrace}]{align}
      & B \min \crl[\big]{ 1, (\alpha n)^{-1/2} } & & 
      \hspace{-40pt}\mbox{for~~} \L = 
      \Lcvar\label{eq:cvar-bias}\\
      & B \sqrt{ {(1+\rho) (\log n)/n}} & & \hspace{-40pt}\mbox{for~~} \L = 
      \Lcs 
      \label{eq:cs-bias}\\ 
      & {B^2}(\lambda n)^{-1} & & \hspace{-40pt}\mbox{for~~} \L=\Llam 
      \label{eq:lam-bias}\\
      & {\icdfLip}\,{n^{-1}}  & & \hspace{-40pt}\mbox{for any 
        loss~\eqref{eq:pop-dro-reg},}\label{eq:icdf-bias}
    \end{empheq}
where the bound~\eqref{eq:icdf-bias} holds under Assumption~\ref{assumption:grad-lipschitz}.
\end{restatable}

We present the proof in Appendix~\ref{prf:prop-batch-bias} and make a 
few remarks before proceeding to discuss the main proof ideas. 
First, the bounds~\eqref{eq:cvar-bias},~\eqref{eq:cs-bias} 
and~\eqref{eq:lam-bias} are all tight up to constant or logarithmic factors 
when $\ell(x,S)$ has a Bernoulli distribution, and so are unimprovable 
without further assumptions (see Proposition~\ref{prop:lb-bias} in 
Appendix~\ref{sec:worst-case-bias}). One such assumption is that 
$\ell(x;S)$ has $\icdfLip$-Lipschitz inverse-cdf, and it allows us to obtain 
a general $1/n$ bias bound~\eqref{eq:icdf-bias} independent of the 
uncertainty set size. As we discuss in Appendix~\ref{app:batch-icdf}, this 
assumption has natural relaxations for uniform distributions with finite 
supports and, for CVaR at level $\alpha$, we only need the inverse cdf 
$F^{-1}(\beta)$ to be Lipschitz around $\beta=\alpha$, a common 
assumption in the risk estimation literature~\cite{TrindadeUrShZr07}.  
\notarxiv{\vspace{-6pt}}
\begin{proof}[Proof sketch]
  To show that $\L(x; P_0) \ge \bL(x;n)$ for every loss of the 
  form~\eqref{eq:pop-dro-reg}, we use Lagrange duality to write 
 \begin{equation*}
   \L(x;P_0) = \inf_{\eta,\nu} \E_{S_1^n\sim P_0^n}\frac{1}{n} \sum_{i=1}^n 
   \Upsilon(x;\eta,\nu;S_i)
  ~~\mbox{and}~~
  \bL(x;n) = \E_{S_1^n\sim P_0^n} \inf_{\eta,\nu} 
   \frac{1}{n} \sum_{i=1}^n \Upsilon(x;\eta,\nu;S_i),
 \end{equation*}
  for some $\Upsilon:\X\times \R \times \R_+ \times \ss \to \R$. 
  This exposes the fundamental source of the mini-batch estimator bias: 
  when infimum and 
  expectation do not commute (as is the case in general), exchanging them 
  strictly decreases the 
  result.

  Our upper bound analysis begins with CVaR, where 
  $\Lcvar=\frac{1}{\alpha}\int\indic{\beta\ge 1-\alpha}F^{-1}(\beta)\d 
  \beta$ and $\bLcvar=\frac{1}{\alpha}\int 
  \mc{I}_\alpha(\beta)F^{-1}(\beta)\d \beta$, with $F^{-1}$ the inverse cdf of 
  $\ell(x,S)$ and $\mc{I}_\alpha$ a ``soft step function'' that we write in 
  closed form as a sum of Beta densities. To obtain the 
  bound~\eqref{eq:cvar-bias} we express $\int (\indic{\beta\ge 
    1-\alpha}-\mc{I}_\alpha(\beta))_+\d\beta$ as a sum of binomial tail 
  probabilities and apply Chernoff bounds. For CVaR only, the improved 
  bound~\eqref{eq:icdf-bias} follows from arguing that replacing 
  $F^{-1}(\beta)$ with $\icdfLip\cdot \beta$ overestimates the bias, and 
  showing that $\int (\indic{\beta\ge 1-\alpha}-\mc{I}_\alpha(\beta)) 
  \beta\d\beta \le (n+1)^{-1}$ for any $\alpha$.

  To transfer the CVaR bounds to other objectives we express the 
  objective~\eqref{eq:pop-dro-reg} as a weighted CVaR average over 
  different $\alpha$ values, essentially using the Kusuoka representation of 
  coherent risk measures~\cite{Kusuoka01}.
  Given any bias bound $\biasbound(\alpha)$ for CVaR 
  at level $\alpha$, this expression implies the bound
  $\L - \bL \le \sup_{w\in\wset(\L)} \int \biasbound(\alpha)\d 
  w(\alpha)$, 
  where $\wset(\L)$ is a set of probability measures. %
  Substituting $\biasbound(\alpha) = 1/\sqrt{n\alpha}$ and using the 
  Cauchy-Schwartz inequality gives the bound~\eqref{eq:cs-bias}, while 
  substituting 
  $\biasbound(\alpha) = \icdfLip /n$ shows this bound in fact holds for 
  any 
  $\L$, as 
  we claim in~\eqref{eq:icdf-bias}. 

  Showing the bound~\eqref{eq:lam-bias} requires a fairly different argument. 
  Our proof uses the dual representation of $\Llam$ as a minimum of an 
  expected risk over a Lagrange multiplier $\eta$ imposing the constraint 
  that 
  $q$ in~\eqref{eq:finite-rob} sums to $1$ (or that $Q$ 
  in~\eqref{eq:pop-dro-reg} integrates to $1$). Using convexity with respect 
  to $\eta$ we relate the value of the risk at $\eta_n$ (the minimizer for 
  sample $S_1^n$) to $\eta\opt$ (the population minimizer), which on 
  expectation are $\bLlam$ and $\Llam$, respectively. We then apply 
  Cauchy-Schwartz and bound the variance of $\eta_n$ with the Efron-Stein 
  inequality~\cite{EfronSt81} to obtain a $1/n$ bias bound.
\end{proof}

\subsection{Variance analysis}\label{sec:batch-variance}

With the bias bounds in Proposition~\ref{prop:lb-bias} established, we analyze
the variance of the stochastic gradient estimators $\grad \L(x; S_1^n)$.  More
specifically, we prove that the variance of the mini-batch gradient estimator
decreases as $1/n$ for penalty-type robust objectives (with $\phi=0$) for which
the maximizing $Q$ has bounded $\cs$ divergence from $P_0$, which we call ``$\cs$-bounded objectives'' (see \Cref{app:prelims-cs-bounded}). Noting that $\LcvarR$ (with $\Lcvar$ as a
special case) and $\Llam$ are $\cs$-bounded yields the following.
\begin{restatable}[Variance of the batch estimator]
  {proposition}{restatePropVarianceGrad}\label{prop:variance-grad}
  For all $n\in\N$, $x\in\X$, and $S_1^n \sim P_0^n$, 
  \begin{equation*}
    \Var \brk[\Big]{\grad\LcvarR(x; S_1^n)}
    \lesssim \frac{G^2}{\alpha n} 
    \mbox{~~and~~}
    \Var \brk[\Big]{\grad \Llam(x; S_1^n)}
    \lesssim \frac{G^2  (1+B/\lambda)}{n}.
  \end{equation*}
\end{restatable}
\noindent
(Note that the variance bound on $\LcvarR$ is independent of $\lambda$ and therefore holds also for $\Lcvar$ where $\lambda=0$). \arxiv{
 
}We prove \Cref{prop:variance-grad} in Appendix~\ref{prf:prop-variance-grad} and
provide a proof sketch below.\footnote{In the appendix we provide bounds on the
  variance of $\L(x;S_1^n)$ in addition to $\grad \L(x;S_1^n)$.} Unfortunately,
the bounds do not extend to the $\cs$ constrained
formulation\arxiv{~\eqref{eqn:chi-square-def}}: in
Appendix~\ref{prf:prop-variance-grad} (Proposition~\ref{prop:lb-variance}) we
prove that for any $n$ there exist $\ell$, $P_0$, and $x$ such that
$\Var[\grad \Lcs(x;P_0)] \gtrsim \rho$. Whether
Proposition~\ref{prop:variance-grad} holds when adding a $\cs$ penalty to the
$\cs$ constraint remains an open question.

\notarxiv{\vspace{-6pt}}
\begin{proof}[Proof sketch]
  The Efron-Stein inequality~\cite{EfronSt81} is $\Var[\grad \L(x;S_1^n)]
  \le \frac{n}{2} \E{} \norm{ \grad \L(x;S_1^n) - \grad
    \L(x;\tilde{S}_1^n)}^2$, where $S_1^n$ and $\tilde{S}_1^n$ are identical
  except in a random entry $I \in [n]$ for which $\tilde{S}_I$ is an i.i.d.\
  copy of $S_I$. We bound $\norm{\grad \L(x;S_1^n) - \grad
    \L(x;\tilde{S}_1^n)} \le G q_I + G \lone{q-\tilde{q}}$ with the
  triangle inequality, where $q$ and $\tilde{q}$ attain the maximum
  in~\eqref{eq:finite-rob} for $S$ and $\tilde{S}$, respectively. The crux
  of our proof is the equality $\lone{q-\tilde{q}}=2|q_I - \tilde{q}_I|$,
  which holds since increasing one coordinate of
  $\ell(x;S_1),\ldots,\ell(x;S_n)$ must decrease all other coordinates in
  $q$. Noting
  that $\E{}(q_I-\tilde{q}_I)^2 \le 4\E (q_I-1/n)^2= \frac{8}{n^2}\E\divcs(
  q, \frac{1}{n}\ones)$, the results follow by observing that  $\divcs( q, 
  \frac{1}{n}\ones)$ is bounded by $1/\alpha$ and
  $B/\lambda$ for $\LcvarR$ and
  $\Llam$, respectively. 
\end{proof}

\subsection{Complexity guarantees}\label{sec:batch-rates}

With the bias and variance guarantees established, we now provide bounds on 
the complexity of minimizing $\L(\x;P_0)$ to arbitrary accuracy $\epsilon$ 
using standard gradient methods with the gradient estimator $\tilde{g}(x) = 
\grad \L(x;S_1^n)$. 
(Recall from Section~\ref{sec:prelims} that we measure complexity
by the number of individual first order evaluations
$(\loss(x; s), \grad \loss(x; s))$.)
Writing $\Pi_{\X}$ for the Euclidean projection onto $\X$, the stochastic
gradient method (SGM) with fixed step-size $\eta$ and $\x_0 \in \X$ 
iterates
\begin{equation}\label{eq:sgd}
x_{t+1} = \Pi_{\X}\prn{x_t - \eta\tilde{g}(x_t)},~\mbox{and}~~
\bar{x}_{t} %
= \frac{1}{t}\sum_{\tau\le t} x_\tau.
\end{equation}
We also consider Nesterov's accelerated gradient 
method~\cite{Nesterov83,Lan12}.  For
$x_0 = y_0 = z_0 \in \X$, a fixed step-size $\eta>0$ and a sequence
$\crl{\theta_t}$, we iterate
\begin{equation}\label{eq:agd}
z_{t+1} = \Pi_{\X}(z_t - \tfrac{\eta}{\theta_t} \tilde{g}(x_t)),~
y_{t+1} = \theta_t z_{t+1} + (1-\theta_t) y_t,~\mbox{and}~~
x_{t+1} =  \theta_{t+1} z_{t+1} + (1-\theta_t) y_{t+1}.
\end{equation}

\notarxiv{In Appendix~\ref{app:gangster}, we}
\arxiv{We now} state the rates of convergence of the iterations~\eqref{eq:sgd}
and~\eqref{eq:agd} following the analysis in~\cite{Lan12}, with a small
variation where the stochastic gradient estimates are unbiased for a uniform
approximation of the true objective with additive error $\delta$.
\arxiv{We provide a short proof in Appendix~\ref{app:gangster}.}
\arxiv{
\begin{restatable}[Convergence of stochastic gradient methods~{\cite[][Corollary 1]{Lan12}}]
  {proposition}{restatePropGangster}\label{prop:gangster}
Let $F:\X\to\R$ and $\overline{F}:\X\to\R$ satisfy $0\le 
F(x)-\overline{F}(x)\le
\delta$ for all $\X\in\R$. Assume that $\overline{F}$ is convex and that a
stochastic gradient estimator $\tilde{g}$ satisfies $\E{} \tilde{g}(x)
\in \del \overline{F}(x)$ and $\E{} \norm{\tilde{g}(x)}^2 \le \Gamma^2$ for
all $x\in\X$. For $T\in\N$, the iterate $\bar{x}_T$ in the
sequence~\eqref{eq:sgd} with 
 $\eta \asymp \frac{R}{ T^{1/2}\Gamma}$
satisfies
\begin{equation}\label{eq:sgd-rate}
\E F(\bar{x}_T) - \inf_{x'}F(x') \lesssim \delta + \frac{\Gamma
  R}{\sqrt{T}}.
\end{equation}
If in addition $\grad \overline{F}$ is $\Lambda$-Lipschitz and $\Var
\brk*{\tilde{g}(x)} \le \sigma^2$ for all $x\in\X$, the iterate $y_T$
in the sequence~\eqref{eq:agd} with $\eta \asymp
\min\{\frac{1}{\Lambda}, \frac{R}{T^{3/2}\sigma}\}$ and $\theta_t =
\frac{2}{t+1}$ satisfies
\begin{equation}\label{eq:agd-rate}
\E F(y_T) - \inf_{x'}F(x') \lesssim \delta + \frac{\Lambda R^2}{T^2} +
\frac{\sigma R}{\sqrt{T}}.
\end{equation}
\end{restatable}
}
Since our gradient estimator has norm bounded by $G$, SGM
allows us to find an $\epsilon$-minimizer of $\bL$ in
$T\asymp {(GR)^2}/{\epsilon^2}$ steps. Therefore, choosing $n$ large enough in
accordance to Proposition~\ref{prop:batch-bias} guarantees that we find an
$\epsilon$-minimizer of $\L$. The accelerated scheme~\eqref{eq:agd} admits
convergence guarantees that scale with the gradient estimator variance
instead of its second moment, allowing us to leverage
Proposition~\ref{prop:variance-grad} to reduce $T$ to the order of
$1/\epsilon$. The accelerated guarantees require the loss $\L$ to have 
order $1/\epsilon$-Lipschitz gradients---fortunately, this holds for 
$\Llam$ and $\LcvarR$.

\begin{restatable}{claim}{restateClaimSmoothObj}\label{claim:smooth-obj}
  Let Assumption~\ref{assumption:grad-lipschitz} hold.
  For all $P$, $\grad \LcvarR(x; P)$ and $\grad \Llam(x; P)$ are 
  $(\frac{G^2}{\lambda}+H)$-Lipschitz in $x$, and $0\le 
  \Lcvar(x;P)-\LcvarR(x;P)\le \lambda \log({1}/{\alpha})$ for all $x$.
\end{restatable}
\noindent
See proof in Appendix~\ref{app:prelims-smoothness}. Thus, to minimize
$\Lcvar$ we instead minimize $\LcvarR$ and choose
$\lambda\asymp \epsilon/\log({1}/{\alpha})$ to satisfy the smoothness
requirement while incurring order $\epsilon$ approximation error. For $\Llam$
with $\lambda \ge \epsilon$ we get sufficient smoothness for 
free.\footnote{%
	We can also handle the case $\lambda < \epsilon$ by adding a 
	KL-divergence term to $\psi$ for $\Llam$.}

As computing every gradient estimator requires $n$ evaluations
of $\grad\ell$, the total gradient complexity is $nT$, and we have the following
suite of guarantees (see Appendix~\ref{prf:thm-complexity-batch} for proof).
\begin{restatable}{theorem}{restateThmBatchComplexity}
  \label{thm:batch-complexity}
  Let Assumptions~\ref{assumption:grad-lipschitz}
  and~\ref{assumption:smooth-icdf} hold, possibly trivially (with  $H=\infty$ or $\icdfLip=\infty$).  Let $\epsilon\in(0,B)$ and write
  $\nu=\frac{H}{G^2}\epsilon$. With suitable choices of the batch size $n$
  and iteration count $T$, the gradient methods~\eqref{eq:sgd}
  and~\eqref{eq:agd} find $\bar{x}$ satisfying $\E \L(\bar{x},P_0) -
  \inf_{x'\in\X}\L(x';P_0) \le \epsilon$ with complexity $n T$ admitting the
  following bounds.
  \begin{itemize}[leftmargin=*]
  \item For $\L=\Lcvar$, we have
    $nT \lesssim 
    \frac{(GR)^2}{\alpha\epsilon^2}
    \prn*{1 + 
      \min\crl[\Big]{
	\frac{\alpha\icdfLip\sqrt{\log\frac{1}{\alpha} + \nu}}{GR},
	\frac{B^2\sqrt{\log\frac{1}{\alpha}+\nu}}{GR \epsilon},
        \frac{B^2}{\epsilon^2}
      }
    }
    $.
  \item For $\L=\Llam$ with $\lambda \le B$, we have
    $nT \lesssim 
    \frac{(GR)^2B}{\lambda\epsilon^2}
    \prn*{1 + 
      \min\crl[\Big]{
	\frac{B}{GR}\sqrt{\frac{\epsilon(1+\nu)}{\lambda}}, 
	\frac{B}{\epsilon}
    }}$.
  \item For $\L=\Lcs$, we have
    $nT \lesssim 
    \frac{(1+\rho)(GR)^2 
      B^2}{\epsilon^4}\log\frac{(1+\rho)B^2}{\epsilon^2}$.
  \item For any loss of the from~\eqref{eq:pop-dro-reg}, we have $nT 
    \lesssim \frac{(GR)^2 \icdfLip }{\epsilon^3}$.
  \end{itemize}
\end{restatable}
The smoothness parameter $H$ only appears in rates resulting from Nesterov
acceleration.  Even there, $H$ appears in lower-order terms in $\epsilon$
since $\nu = \frac{H}{G^2} \epsilon$. We also
note that the final $\icdfLip\epsilon^{-3}$ rate holds even when the
uncertainty set is the entire simplex; therefore, when $\icdfLip<\infty$ it
is possible to approximately minimize the maximum 
loss~\cite{ShalevWe16} in sublinear
time. Theorem~\ref{thm:batch-complexity} achieves the claimed rates
of convergence in Table~\ref{table:summary} in certain settings. In 
particular, it recovers the rates for $\Lcvar$
and $\Llam$ (the first and last column of the table)
when $\nu\lesssim 1$, $\lambda \gtrsim (B/(GR))^2\epsilon$, and $\alpha
\lesssim GR/\icdfLip$.
In the next section, we show how to attain the claimed optimal
rates for $\Lcvar$ and $\Llam$ without conditions,
returning to address the rates for the constrained $\chi^2$ objective
$\Lcs$ in \Cref{sec:doubling}.

\section{Multi-level Monte Carlo (MLMC) gradient 
estimators}\label{sec:multilevel}

In the previous section, we optimized the mini-batch surrogate
$\bL(x; n)$ to the risk $\L(x; P_0)$, using
Proposition~\ref{prop:batch-bias} to guarantee the surrogate's
fidelity for sufficiently large $n$. The increasing (linear) complexity of 
computing the estimator
$\grad \L(x; S_1^n)$ as $n$ grows limits the (theoretical) efficiency of the
method. To that end, in this section we revisit a multi-level Monte Carlo
(MLMC) gradient estimator of \citet{BlanchetGl15} to form an 
unbiased approximation to $\grad \bL(x; n)$ whose sample complexity is
{logarithmic} in $n$. We provide new bounds on the variance of this 
MLMC
estimator, leading immediately to improved (and, as we shall see, optimal)
efficiency estimates for stochastic gradient methods using it.

\notarxiv{
	To define the estimator, let $J\sim \min\{\mathsf{Geo}(1/2), \jmax\}$ be 
	a truncated geometric random variable supported on 
	$\{1,\ldots,\jmax\}$, and let $q(j) = \P(J=j) = 2^{-j + \indic{j=\jmax}}$.
  For a realization of $J$ we draw  a sample of size $2^J n_0$
  and compute the 
  multi-level Monte-Carlo estimator 
  as follows:
  \begin{equation}%
    \ml[\grad \L] \defeq \grad \L(x;S_1^{n_0}) + \frac{1}{q(J)} 
    \mld_{2^J n_0},~\mbox{where}~
    \mld_k \defeq \grad \L(x;S_1^{k}) - \frac{
      \grad \L(x;S_1^{k/2}) + \grad \L(x;S_{k/2+1}^{k})
    }{2}.
    \label{eqn:mlmc}
  \end{equation}
}
\arxiv{
  To define the estimator, let $J\sim \min\{\mathsf{Geo}(1/2), \jmax\}$ be a 
  truncated geometric random variable supported on $\{1,\ldots,\jmax\}$, 
  and let $q(j) = \P(J=j) = 2^{-j + \indic{j=\jmax}}$. Furthermore, for any 
  $k\in2\N$ we define the ``bias increment'' estimate
  \begin{equation*}
  \mld_k \defeq \grad \L(x; S_1^k) - \frac{\grad \L(x; S_1^{k/2})
  	+ \grad \L(x; S_{k/2 + 1}^k)}{2}.
  \end{equation*}	
 For a given minimum sample size parameter $n_0 \ge 1$, we define 
 $\ml[\grad \L]$, the MLMC estimator of $\grad \L$, via
  \begin{align}
    \nonumber
    \mbox{Draw}~ & J \sim \min\left\{\mathsf{Geo}(1/2), \jmax \right\}
    ~~ \mbox{and}~ S_1,\ldots, S_{2^J n_0} \simiid P_0 \\
    \mbox{Estimate}~
    & \ml[\grad \L] \defeq \grad \L(x; S_1^{n_0}) + \frac{1}{q(J)}
    \mld_{2^J n_0}.
     \label{eqn:mlmc} 
  \end{align}
}

Our estimator differs from the proposal~\cite{BlanchetGl15} in two aspects:
the distribution of $J$ and the option to set $n_0>1$. As we further discuss
in Appendix~\ref{app:compare-with-blanchet}, the former difference is
crucial for our setting, while the latter is pratically and theoretically
helpful yet not crucial.  The following properties of the MLMC estimator are 
key to our
analysis (see Appendix~\ref{sec:prf-mlmc-bounds} for proofs).

\begin{restatable}{claim}{restateClaimMLMC}\label{claim:mlmc}
  The estimator $\ml[\grad \L]$ with parameters $n = 2^{\jmax} n_0$ satisfies
  \begin{equation*}
    \E \ml[\grad L] = \E \grad \L(x; S_1^{\kmax})=\grad \bL(x;n),
    ~
    \text{requiring expected sample size}~\E 2^J n_0 =
    n_0(1+\log_2(\kmax/n_0)).
  \end{equation*}
\end{restatable}

\begin{restatable}[Second moment of MLMC gradient estimator]
  {proposition}{restatePropMomentMLMC}\label{prop:ml-mom}
	For all $x\in\X$, the multi-level Monte Carlo estimator with parameters 
	$n$ and $n_0$ satisfies
	\begin{equation*}
	\E{}\,\norm[\Big]{\ml\brk[\big]{\grad\Lcvar}}^2 \lesssim 
	\prn*{1+\frac{\log\frac{\kmax}{n_0}}{\alpha n_0}}G^2
	~~\mbox{and}~~
	\E{}\,\norm[\Big]{\ml\brk[\big]{\grad\Llam}}^2 \lesssim 
	\prn*{1+\frac{B\log\frac{\kmax}{n_0}}{\lambda n_0}}G^2.
	\end{equation*}
\end{restatable}

Claim~\ref{claim:mlmc} follows from a simple calculation, while the core of
Proposition~\ref{prop:ml-mom} is a sign-consistency argument for simplifying a
1-norm, similar to the proof of Proposition~\ref{prop:variance-grad}.  \arxiv{%
  Specifically, for $q$ and $q'$ attaining the maximum~\eqref{eq:finite-rob} for
  samples $S_1^k$ and $S_1^{k/2}$, respectively, we show that
  $\E \norm{\mld_k}^2 \lesssim G^2 \E \norm{q_1^{k/2}-\half q'}^2_1$.  Then, we
  argue that $\norm{q_1^{k/2}-\half q'}_1 = | \ones^\top q_1^{k/2} - \half |$ as
  $q_i-\half q'_i$ has the same sign for $i\le k/2$. This implies that
  $\E \norm{\mld_k}^2$ scales as $1/k$, and the desired bound on the expected
  gradient estimator norm follows by direct calculation. The proof extends to  any unconstrained $\cs$-bounded objective (see \Cref{app:prelims-cs-bounded}), including $\LcvarR$ (independently of $\lambda$). 
}%

Further paralleling Proposition~\ref{prop:variance-grad},
we obtain similar bounds on the MLMC estimates of $\Lcvar$ and $\Llam$ (in
addition to their gradients), and demonstrate that similar bounds fail to hold
for $\grad\Lcs$ (Proposition~\ref{prop:lb-cs-mlmc} in
Appendix~\ref{sec:prf-mlmc-bounds}).  Therefore, directly using the MLMC
estimator on $\grad\Lcs$ cannot provide guarantees for minimizing $\Lcs$;
instead, in \Cref{sec:doubling} we develop a doubling scheme that minimizes the
dual objective $\Llam(x;P_0) + \lambda \rho$ jointly over $x$ and
$\lambda$. This scheme relies on MLMC estimators for both the gradient
$\grad \Llam$ and the derivative of $\Llam$ with respect to $\lambda$.

Proposition~\ref{prop:ml-mom} guarantees that the second moment of our
gradient estimators remain bounded by a quantity that depends
logarithmically on $n$. For these estimators,
Proposition~\ref{prop:gangster} thus directly provides complexity guarantees
to minimize $\Lcvar$ and $\Llam$. We also provide a high probability bound on
the total complexity of the algorithm using a one-sided Bernstein
concentration bound. We state the guarantee below and present a short proof
in Appendix~\ref{prf:mlmc-complexity}. %

\begin{restatable}[MLMC complexity
  guarantees]{theorem}{restateThmMLMC}
  \label{thm:ml}
  For $\epsilon \in (0,B)$, set $n\asymp \frac{B^2}{\alpha \epsilon^2}$,
  $1\lesssim n_0 \lesssim \frac{\log n}{\alpha}$ and
  $T\asymp\frac{(GR)^2}{n_0\alpha\epsilon^2}\log^2{n}$. The stochastic 
  gradient
  iterates~\eqref{eq:sgd} with $\tilde{g}(x) = \ml[\grad \Lcvar(x;\cdot)]$
  satisfy
  $\E[\Lcvar(\bar{x}_T;P_0)] - \inf_{\x \in \X}\Lcvar(x; P_0) \le \epsilon$ with
  complexity at most
  \begin{equation*}
    n_0\log_2\prn*{\frac{n}{n_0}}T +
    5\sqrt{(n\log n)^2 + n_0n T \log n}
    \lesssim
    \frac{(GR + B)^2}{\alpha\epsilon^2}\log^2\frac{B^2}{\alpha\epsilon^2}
    \mbox{~~w.p~~}\ge 1-\frac{1}{n}.
  \end{equation*}
  The same conclusion holds when replacing $\Lcvar$ with $\Llam$ and
  $\alpha^{-1}$ with $1+B/\lambda$.

\end{restatable}

\arxiv{\section{Lower bounds}\label{sec:lowerbounds}}
\notarxiv{\paragraph{Lower bounds.}}
We match the guarantees of Theorem~\ref{thm:ml} with lower bounds that hold in a
standard \emph{stochastic oracle} model~\cite{NemirovskiYu83,Lan12,BraunGuPo17},
where algorithms interact with a problem instance by iteratively
querying $x_t\in \X$  (for $t\in\N$) and observing $\ell(x_t;S)$ and $\grad 
\ell(x_t;S)$ with
$S\sim P_0$ (independent of $x_t$).  All algorithms we consider fit into this
model, with each gradient evaluation corresponding to an oracle
query. Therefore, to demonstrate that our MLMC guarantees are unimprovable in
the worst case (ignoring logarithmic factors), we formulate a lower bound on the
number of queries any oracle-based algorithm requires.

\begin{restatable}[Minimax lower bounds]{theorem}{restateThmLowerBounds}
  \label{thm:lb} 
  Let $G,R,\alpha,\lambda>0$, $\epsilon\in(0, GR/64)$, and
  sample space $\ss = [-1, 1]$. There exists a numerical constant
  $c > 0$ such that the following holds.
  \begin{itemize}[leftmargin=*]
  \item For each $d \ge 1$, domain $\mc{X} = \{x \in \R^d \mid \norm{x} 
  \le
    R\}$, and any algorithm, 
    there exists a distribution $P_0$ on $\ss$ and convex
    $G$-Lipschitz loss $\loss : \mc{X} \times \ss \to [0, GR]$ such that
    \begin{equation*}
      T \le c \frac{(GR)^2}{\alpha \epsilon^2}
      \mbox{~~implies~~}
      \E\brk{\Lcvar(x_T;P_0)}-\inf_{x'\in\X}\Lcvar(x';P_0) > 
      \epsilon.
    \end{equation*}
  \item There exists $d_\epsilon \lesssim (GR)^2 \epsilon^{-2}
    \log \frac{GR}{\epsilon}$ such that for
    $\mc{X} = \{x \in \R^d \mid \norm{x} \le R\}$,
    the same conclusion holds when replacing
    $\Lcvar$ with $\Llam$ and $\alpha$ with $\lambda / (GR)$.
  \end{itemize}
\end{restatable}

We present the proof in Appendix~\ref{sec:prf-lowerbounds}\arxiv{ and provide a  
sketch below}. Our proof for the penalized $\cs$ lower bound 
leverages a classical high-dimensional hard instance construction for 
oracle-based optimization, 
while our 
proof 
for CVaR is information-theoretic. Consequently, the CVaR 
lower bound is stronger: it holds for $d=1$ and extends 
to a global model where at every round the oracle provides the entire 
function $\ell(\cdot; S)$ rather than $\ell(x;S)$ and $\grad\ell(x;S)$ at the 
query point $x$.
\arxiv{
\begin{proof}[Proof sketch]
The proof of the CVaR lower bound relies on the classical reduction from 
optimization to testing~\cite[Chapter 5]{Duchi18} in conjunction with 
the Le Cam method~\cite{Yu97}. More precisely, we construct a pair of 
distributions $P_{-1}$ and $P_1$ that are statistically hard to 
distinguish yet are such 
that $\L(\cdot;P_{-1})$ and $\L(\cdot;P_1)$ have well-separated values at 
their respective minima. Our construction takes the loss to be 
$\ell(x;s)=x\cdot s$, and the distributions $P_{\pm1}$ to be perturbations 
of $\bernrv(\alpha)$, similarly to the lower bound of~\citet{DuchiNa20} for 
constrained-$\cs$.

Unlike the CVaR and constrained-$\cs$ objectives, the penalized-$\cs$ 
objective with the loss $\ell(x; s)=x\cdot s$ is not positively homogeneous 
in $x$, 
making the Le Cam lower bound strategy difficult to apply. Instead, 
we appeal to a classical high-dimensional hard instance construction for 
convex optimization~\cite{NemirovskiYu83,BraunGuPo17}. Choosing the 
sample space $\ss = \crl{0, 1}$, we construct $\ell(x;s)$ such that 
$\ell(x;1)$ is equal to the hard instance at $x$ and $\ell(x;0)=-GR$ is 
uninformative. We show that the robust loss is 
(up to an additive constant) equal to the hard instance and thus minimizing 
it requires sampling $S=1$ roughly $\Omega(\epsilon^{-2})$ times; setting 
$\P(S=1) = \lambda/GR$ thus establishes the desired lower bound.
\end{proof}}

 \arxiv{%
\section{A doubling scheme for minimizing $\Lcs$}\label{sec:doubling}

The remaining technical contribution in the paper is to revisit
the constrained $\cs$ objective~\eqref{eqn:chi-square-def}, which
is resistant to many of the techniques we have thus far developed.
In this section, we leverage duality relationships to approximate
the constrained objective~\eqref{eqn:chi-square-def} via
its penalized counterpart~\eqref{eqn:chi-square-reg-def}, $\Llam$.
We adjust notation to make the dependence 
of $\Llam^\lambda$ on $\lambda$ explicit, and defer all proofs to 
\Cref{app:doubling}. 

Our starting point is the recognition that, by duality
(cf.~\cite[Sec.~3.2]{Shapiro17}),
\begin{equation*}
  \Lcs(x; P_0) = \inf_{\lambda \ge 0} \left\{ \Llam^\lambda(x; P_0)
  + \lambda \rho \right\}
  = \inf_{\lambda \ge 0} \sup_{Q \ll P_0}
  \left\{\E_Q \loss(x; S) - \lambda
  \left[\divcs(Q, P_0) - \rho \right] \right\}
\end{equation*}
for any distribution $P_0$. For $0 \le \lambdaL \le \lambdaH$, we
may thus consider the approximation
\begin{equation*}
  \LcsRange{\lambdaL, \lambdaH}(x;P_0) \defeq 
  \min_{\lambda\in [\lambdaL, \lambdaH]} 
  f_\rho(x,\lambda)
  ~~\mbox{where}~~
  f_\rho(x,\lambda) \defeq \Llam^\lambda(x;P_0) + \lambda \rho.
\end{equation*}
By restricting $\lambda$ to an appropriate range, we can then approximate 
$\Lcs$ by
its truncated version, as the next lemma 
shows.

\begin{restatable}{lemma}{restateCSDualRange}\label{lem:cs-dual-range}
  For all $P_0$, $\rho$ and $\epsilon$,
  \begin{equation*}
    \min_{x\in\X} \LcsRange{\frac{\epsilon}{2\rho}, \frac{B}{\rho}}(x;P_0)
    \le \min_{x'\in\X} \Lcs(x';P_0) + \frac{\epsilon}{2}.
  \end{equation*}
\end{restatable}

Our strategy is therefore to jointly minimize
$f_\rho(x,\lambda)=\Llam^{\lambda}(x;P_0)+ \lambda \rho$ over both $x \in
\mc{X}$ and $\lambda \in [\lambdaL, \lambdaH]$ (rather than $[0, \infty]$),
using the approximation guarantee in Lemma~\ref{lem:cs-dual-range} to argue
that the restriction of $\lambda$ will have limited effect on the quality of
the resulting solution. We iterate the projected stochastic gradient method
with the multi-level Monte Carlo (MLMC) gradient estimator~\eqref{eqn:mlmc}
via \notarxiv{
  \begin{equation}
  \label{eq:x-lambda-sgd}
  x_{t+1} = \Pi_{\X}\prn*{x_{t} - \gamma_x \ml\brk[\big]{\grad 
      \Llam^{\lambda_t}(x_t)}},~
  \lambda_{t+1} = \Pi_{[\lambdaL,\lambdaH]}\prn*{\lambda_t - 
    \gamma_\lambda
    \ml\brk[\big]{\tfrac{\del}{\del \lambda} 
      \Llam^{\lambda_t}(x_t)+\rho}}.
\end{equation}
}
\arxiv{
  \begin{equation}
    \label{eq:x-lambda-sgd}
    \begin{split}
      x_{t+1} & = \Pi_{\X}\prn*{x_{t} - \gamma_x \ml\brk[\big]{\grad 
          \Llam^{\lambda_t}(x_t)}} \\
      \lambda_{t+1} & = \Pi_{[\lambdaL,\lambdaH]}\prn*{\lambda_t - 
        \gamma_\lambda
        \ml\brk[\big]{\tfrac{\del}{\del \lambda} 
          \Llam^{\lambda_t}(x_t)+\rho}}.
    \end{split}
  \end{equation}
}
If we can bound the moments of the MLMC-approximated gradients
$\ml$, we can then leverage standard stochastic gradient
analyses to prove convergence. We use the following
bound.
\begin{restatable}{lemma}{restateLemLambdaMom}\label{lem:ml-lambda-mom}
	We have 
	\begin{equation*}
	\E \prn*{\ml\brk[\big]{\tfrac{\del}{\del \lambda} 
			\Llam^{\lambda}(x;\cdot)+\rho}}^2 \lesssim 
	\frac{B^2}{\lambda^2}\prn*{
		1+\frac{B \log\frac{n}{n_0}}{\lambda n_0}} + \rho^2.
	\end{equation*}
\end{restatable}

\noindent
Therefore, 
we may find an $\epsilon$ approximate minimizer with complexity roughly 
$B^3\lambdaH^2/(\lambdaL^3\epsilon^2)$:
\begin{restatable}{lemma}{restateLemXLamSGD}\label{lem:x-lambda-sgd}
  Fix $\epsilon\in (0,B)$ and $\lambdaH \ge \lambdaL > 0$. For a suitable 
  setting of the parameters $n_0,n,T,\gamma_x$ and $\gamma_\lambda$, 
  the average $\bar{x}_T = \sum_{t\le T} x_t$ 
  of the iterates~\eqref{eq:x-lambda-sgd} satisfies
  $\E \LcsRange{\lambdaL,\lambdaH}(\bar{x}_T;P_0) \le \min_{x\in\X}  
  \LcsRange{\lambdaL,\lambdaH}(x;P_0) + \epsilon$, with complexity 
  \begin{equation*}
    \lesssim 
     \prn*{1+\frac{B}{\lambdaL}}\frac{(GR)^2 + B^2 \lambdaH^2 / 
      \lambdaL^2 + 
      \lambdaH^2 \rho^2}{\epsilon^2}\log^2\prn*{1+ 
      \frac{B}{\lambdaL\epsilon}}~~\mbox{with probability 
    }\ge 1-\frac{\eps^2}{B^2}.
  \end{equation*}
\end{restatable}

Directly substituting $\lambdaL  = \frac{\epsilon}{2\rho}$ and 
$\lambdaH=\frac{B}{\rho}$ results in a guarantee scaling as 
$\epsilon^{-5}$, which is worse than the mini-batch rate of 
$\epsilon^{-4}$. To improve on this, we divide $[\frac{\epsilon}{2\rho}, 
\frac{B}{\rho}]$ into $K = \log_2 \frac{B}{\epsilon}$ sub-intervals 
$[\lambda\pind{i+1},\lambda\pind{i}]$ satisfying 
$\lambda\pind{i+1}/\lambda\pind{i} = 2$.
We then perform the stochastic gradient method~\eqref{eq:x-lambda-sgd}
on each of these intervals $[\lambda^{(i + 1)}, \lambda^{(i)}]$ in turn,
yielding estimates $\bar{x}^{(i)}$ that
are each $\lesssim\epsilon$-suboptimal for the approximate
objective $\LcsRange{\lambda\pind{i+1},\lambda\pind{i}}$.
Using the bounded ratio $\lambda\pind{i+1} / \lambda\pind{i} = 2$,
this requires complexity roughly
$1/(\lambda\pind{i+1} \epsilon^2) \lesssim \rho /\epsilon^3$,
giving the following theorem.
\begin{restatable}{theorem}{restateThmDoubling}\label{thm:doubling}
	Fix $\epsilon\in (0,B)$, and for $i\in\N$ set 
	$\lambda\pind{i}=\frac{B}{\rho}2^{-i+1}$ and let $\bar{x}\pind{i}$ be an 
	$\epsilon/2$-approximate minimizer of 
	$\LcsRange{\lambda\pind{i+1},\lambda\pind{i}}$ computed via 
	stochastic 
	gradient iterations according to Lemma~\ref{lem:x-lambda-sgd}. 
	Then, 
	for $1+ K=\ceil{\log_2 \frac{2B}{\epsilon}}$ and some $i\opt \le K$ we 
	have
	$\E \Lcs(\bar{x}\pind{i\opt}; P_0) \le \min_{x\in\X} \Lcs(x;P_0) + 
	\epsilon$. 
	Computing $\bar{x}\pind{1},\ldots, \bar{x}\pind{K}$ requires a total 
	number of $\grad \ell$ evaluations 
	\begin{equation*}
	  \lesssim \frac{(GR)^2 (\rho B + \epsilon \log_2 
		\frac{B}{\epsilon})}{\epsilon^3} \log^2 
	\prn*{1+\frac{\rho 
			B}{\epsilon^2}}~~\mbox{with probability 
	}\ge 1-\frac{\epsilon}{B}.
	\end{equation*}
\end{restatable}

The index $i^\star$ is independent of randomness in our procedure, but we 
do not know it in advance. Instead, we may estimate the minimized objective for each $i$ and select the index 
 with the lowest estimate. Let $\hat{\lambda}\pind{i}$ be the average of the $\lambda$ iterations of our stochastic gradient method~\eqref{eq:x-lambda-sgd} for a particular interval $[\lambda\pind{i+1},\lambda\pind{i}]$. Our bias and variance bounds on $\Llam$ (Proposition~\ref{prop:batch-bias} and Proposition~\ref{prop:variance-extended} in the appendix) imply the we can estimate\footnote{
	To obtain an estimate that has error $\lesssim \epsilon$ with high probability, we can use the median of a logarithmic number of iid copies of the batch estimator.}  $f_\rho(\bar{x}\pind{i},\hat{\lambda}\pind{i})$ to accuracy $\lesssim \epsilon$ with a sample of size $\asymp B^2 / (\lambda\pind{i} \epsilon^2) \asymp  2^{i-K} B^3 \rho\epsilon^{-3}$.
Taking $i\opt$ to be the index $i$ minimizing this estimate, it is straightforward to argue that $\E \Lcs(\bar{x}\pind{i\opt}; P_0) - \min_{x\in\X} \Lcs(x;P_0) \lesssim 
\epsilon$. 
Therefore, the cost of selecting the best $i$ is at most the cost of performing the optimization.

Theorem~\ref{thm:doubling} provides a rigorous guarantee on the 
complexity of minimizing $\Lcs$ with a fixed constraint $\rho$ by 
optimizing the parameter $\lambda$ of $\Llam^\lambda$. In practice, we 
usually have no prior knowledge of $\rho$, so it will often make sense to 
directly tune $\lambda$ according to validation criteria rather than a target 
$\rho$.
We also note that~\citet{DuchiNa20} prove a lower bound of order
${\rho}{\epsilon^{-2}}$, which is smaller than our  
${\rho}{\epsilon^{-3}}$ rate. Establishing  the optimal rate for this 
problem 
remains an open question.

}
\section{Experiments}\label{sec:experiments}

We test our theoretical predictions with experiments on two
datasets. Our main focus is measuring how the total work in solving the DRO
problems depends on different gradient estimators. In particular, we quantify
the tradeoffs in choosing the mini-batch size $n$ in the estimator
$\grad \L(x;S_1^n)$ of Section~\ref{sec:batch} and the effect of using the MLMC
technique of Section~\ref{sec:multilevel}. To ensure that we operate in
practically meaningful settings, our experiments involve heterogeneous data, and
we tune the DRO objective to improve the generalization performance of ERM on
the hardest subpopulation. We provide a full account of experiments in
Appendix~\ref{app:experiments} and summarize them below.

Our digit recognition
experiment  reproduces~\cite[Section 3.2]{DuchiNa20}, where the 
training data includes the 60K MNIST training images mixed with 
600 images of typed digits from~\cite{deCamposBaVa09}, while our 
ImageNet experiment uses the ILSVRC-2012 1000-way classification task. 
In each experiment we use DRO to learn linear classifiers on top of 
pretrained neural network features (i.e., training the head of the network), 
taking $\ell$ to be the logarithmic loss %
 with squared-norm regularization;
see Appendix~\ref{app:experiments-setup}.
Each experiments compares different gradient estimators for 
minimizing the  $\Lcvar$, $\Lcs$ and $\Llam$ objectives.  
Appendix~\ref{app:experiments-tuning}  
details our hyper-parameter settings and their tuning procedures. 

Figure~\ref{fig:main-experiment} plots the training objective as 
optimization progresses. In Appendix~\ref{app:experiments-results} we 
provide expanded figures that also report the robust generalization 
performance. We find that the benefits of DRO manifest mainly when the 
metric of interest is continuous (e.g., log loss) as opposed to the 0-1 
loss. %

\begin{figure}
  \begin{center} \arxiv{\includegraphics[width=\columnwidth]{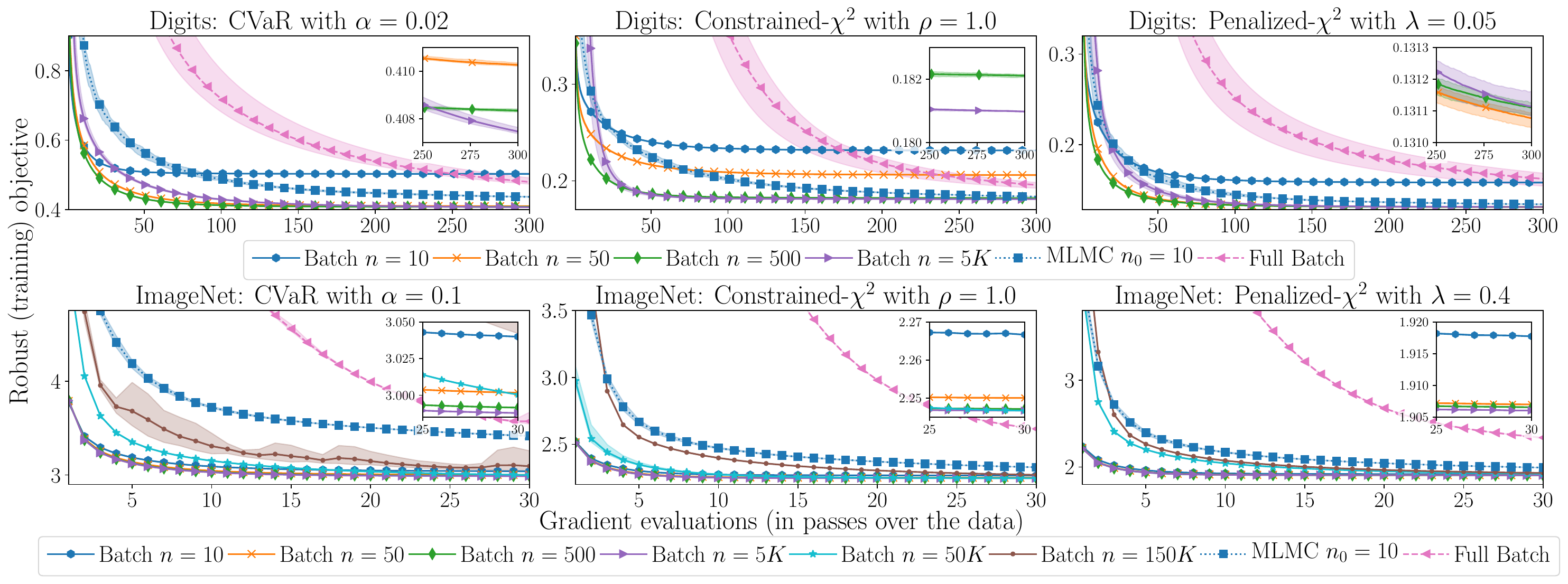}}%
  \notarxiv{\includegraphics[width=0.8\columnwidth]{figures/mnist_bs=10.pdf}}
  \notarxiv{\vspace{0pt}}
    \caption{Convergence of DRO objective in our digits and ImageNet
      classification experiments. Shaded areas indicate range of variability
      across 5 repetitions (minimum to maximum), and the zoomed-in regions
      highlight the (often very low) ``bias floor'' of small batch sizes.}
\label{fig:main-experiment}
\end{center} \notarxiv{\vspace{-0.6cm}} %
\end{figure}

\arxiv{\paragraph{Discussion.}}  Our analysis in Section~\ref{sec:batch-bias}
bounds the suboptimality of solutions resulting from using a mini-batch
estimators with batch size $n$, showing it must vanish as $n$ increases.
Figure~\ref{fig:main-experiment} shows that smaller batch sizes indeed converge
to suboptimal solutions, and that their suboptimality becomes negligible very
quickly: essentially every batch size larger than $10$ provides fairly small
bias (with the exception of $\Lcs$ in the digits experiment). The effect of bias
is particularly weak for $\Llam$, consistent with its superior theoretical
guarantees. We note, however, that the suboptimality we see in practice is far
smaller than the worst-case bounds in
Proposition~\ref{prop:batch-bias}\notarxiv{; we investigate this in detail in
  Appendix~\ref{app:experiments-discussion}.}\arxiv{. We investigate this in
  Appendix~\ref{app:experiments-discussion}, where we show that the bias
  $\L-\bL$ is in fact consistent with our theory, but the minimizers of $\L$ and
  $\bL$ are more similar than expected \emph{a priori}.}

While the MLMC estimator does not suffer from a bias floor (by design), it is
also much slower to converge. This may appear confusing, since the MLMC
convergence guarantees are optimal (for $\Lcvar$ and $\Llam$) while the
mini-batch estimator achieves the optimal rate only under certain
assumptions. Recall, however, that these assumptions are smoothness of the loss
(which holds in our experiments) and---for CVaR---sufficiently rapid 
decay of the
bias floor, which we verify empirically. 

For batch sizes in the range 50--5K, the traces in 
Figure~\ref{fig:main-experiment} look remarkably similar.
This is consistent with our theoretical analysis for $\Lcvar$ and $\Llam$, which
shows that the variance decreases linearly with the batch size and we may
therefore (with Nesterov acceleration) increase the step size proportionally and
expect the total work to remain constant. As theory predicts, this learning rate
increase is only possible up to a certain batch size (roughly 5K in our
experiments), after which larger batches become less efficient. Indeed, to reach
within 2\% of the optimal value, the full-batch method requires 27--36$\times$
more work than batch sizes 50--5K for ImageNet, and 9--16$\times$ more work for
the digits experiment (see Table~\ref{tab:speed-up-digits}
and~\ref{tab:speed-up-imagenet} for a precise breakdown of \notarxiv{these numbers}\arxiv{the number of epochs
required per algorithm for each robust objective}).

We also repeat our experiments with the dual SGM and prima-dual methods mentioned in Table~\ref{table:summary} and compare them with them our proposed method; see Appendix~\ref{app:experiments-baselines} for details.

\arxiv{We conclude the discussion by briefly touching upon the improvement that 
DRO yields
in terms of generalization metrics; we provide additional 
detail in Appendix~\ref{app:experiments-discussion}. In digit recognition 
experiment we observe that, compared to ERM with tuned $\ell_2$ 
regularization, DRO enables strictly better tradeoff between average and 
worst-subgroup performance. Specifically, it provides significant 
improvements in the worst sub-group loss---between 17.5\% and 27\% 
compared to ERM---with no negligible degradation in average loss and 
accuracy. It also provides minor gains in worst-group accuracy. For 
ImageNet the effect is more modest: in the worst-performing 10 classes we 
observe 
improvements of 5--10\% in log loss, as well as a roughly 4 point 
improvement in accuracy. These improvements, however, come at the
cost of degradation in average performance: the average loss increases by 
up to 10\% and the average accuracy drops by roughly 1 point.}

\arxiv{%
\paragraph{Runtime comparison.}
In Table~\ref{tab:compute} we report the gradient complexity and 
wallclock time to reach
accuracy within 2\% of the optimal value. For brevity, we show it
for a single robust objective (penalized-$\cs$), but we observe that
similar results across robust objectives.
We note that for small batch sizes the time per epoch is significantly larger 
than for larger batch sizes, this due in part to parallelization in evaluating 
$\ell$ and $\grad\ell$ and in part to logging and Python interpreter 
overhead, which increase linearly with the number of iterations. However, 
these effects diminish as the batch size grows, and for batch size 5K the 
wallclock time to reach an accurate solution is an order of magnitude 
smaller than with the full-batch method.
We run
our experiments with 4 Intel Xeon E5-2699 CPUs and 12--32Gb
of memory.
Increasing the number of CPUs or using GPUs would allow for greater 
parallelism and improve the runtime at greater batch sizes. However, 
increasing the model complexity (e.g., to a deep neural network) would 
have the opposite effect. Using 4 CPUs for linear classification gives 
roughly the same range of feasible batch sizes as a ResNet-50 on large GPU 
arrays.

\begin{table}
	\begin{center}
		\smaller
		\begin{tabular}{clcccccc}  
			\toprule
			& & \multicolumn{3}{c}{ImageNet times [minutes]} & 
			\multicolumn{3}{c}{Digits times [minutes]} \\
			\cmidrule(lr){3-5}
			\cmidrule(lr){6-8}
			\multicolumn{2}{c}{Algorithm} 
			& per epoch & to 2\% of opt &\# epochs  
			& per epoch & to 2\% of opt &\# epochs \\
			\midrule
			Batch
			& $n=10$ & $120 \pm 5$ & $850 \pm 30$ & 7 & $0.80 \pm 0.1$ & 
			$\infty$  & $\infty$ \\
			& $n=50$ & $23 \pm 0.7$ & $116 \pm 4$ & $\boldsymbol{5}$ & 
			$0.23 
			\pm 0.01$ & $24 \pm 1$  & $107\pm 1$ \\
			& $n=500$ & $5.9 \pm 0.2$ & $29 \pm 1$ & $\boldsymbol{5}$ & 
			$0.056 \pm 0.004$ & $5.8 \pm 0.4$  & $\boldsymbol{104} \pm 
			1$\\
			& $n=5K$ & $3.3 \pm 0.04$ & $\boldsymbol{16.5} \pm 0.2$ & 
			$\boldsymbol{5}$ & $0.033 \pm 0.004$ & 
			$\boldsymbol{4.4} \pm 0.7$  & $131 \pm 6$\\
			& $n=50K$ & $2.2 \pm 0.03$ & $50 \pm 0.9$ & $22$ & -- & --  & 
			--\\
			& $n=150K$ & $2.1 \pm 0.03$ & $55 \pm 0.7$ & $26$ & -- & -- & 
			--\\
			\midrule
			MLMC
			& $n_0 = 10$ & $16\pm 1$ & $\infty$ & $\infty$ & $0.34 \pm 
			0.02$ & 
			$\infty$ & $\infty$ \\
			\midrule
			\multicolumn{2}{c}{Full-batch} & $2.1$ & $380$ & $180$ & $0.022$ 
			& 
			$37.0$ & $1680$ \\
			\bottomrule
		\end{tabular}
		\caption{Comparison wallclock time (in minutes) of the different
			algorithms, in terms of time per epoch and time to reach within 2\% 
			of the
			best training loss. In the last two columns, we report the number of
			epochs required to reach within 2\% of the best training loss. We 
			report 
			$\infty$ for configurations that do not reach the 
			sub-optimality 
			goal for the duration of the experiment, and 
			omit standard deviations when then they are $0$.}
		\label{tab:compute}
	\end{center}
\end{table} %
}

 \arxiv{%
\section{Conclusion}\label{sec:conclusion}

This work provides rigorous convergence guarantees for solving 
large-scale convex $\phi$-divergence DRO problems with stochastic 
gradient methods, laying out a foundation for their use in practice; we 
conclude it by highlighting two directions for further research.

First, while our work resolves the optimal theoretical convergence rates for 
CVaR and $\cs$ penalty objectives, the corresponding result for $\cs$ 
constraint remains open. In particular, there is a gap between our 
$O(\rho\epsilon^{-3})$ upper and the $\Omega(\rho\epsilon^{-2})$ lower 
bound of \citet{DuchiNa20}. %
Moreover, combining the uniform convergence results in~\citet{DuchiNa20} 
with a cutting plane method gives complexity guarantees scaling a roughly 
as $\rho^2 d^2 \epsilon^{-2}$, so the $O(\rho\epsilon^{-3})$ rate can only 
be optimal in high-dimensional settings. 

Second, understanding the practical benefit of large-scale 
$\phi$-divergence DRO for machine learning requires further research.
Our experiments suggest that larger benefits are likely when (a) distinct 
subgroups are present in the data and (b) good calibration and hence low 
logarithmic loss (rather than simply high accuracy) is important. While our 
work 
focuses on convex losses $\ell$ for 
theoretical clarity and experimental simplicity, we note that all the 
algorithms we develop apply directly for non-convex losses. Furthermore, 
our bias and variance analyses are independent of the convexity of $\ell$, 
and our PyTorch implementation supports any prediction model via 
automatic differentiation. Therefore, a natural next step is to apply   
DRO for training modern predictors such as neural networks.
}

\notarxiv{\newpage

\section*{Broader Impact}

The robustness of machine learning (ML) models, or lack thereof, has 
far-reaching present and future societal consequences: in autonomous 
vehicles~\cite{KalraPa16,DaiVa18}, medical 
diagnosis~\cite{OakdenDuCaRe20}, facial 
recognition~\cite{BuolamwiniGe18}, credit scoring~\cite{FusterGoRaWa18}, 
and recidivism prediction~\cite{Chouldechova17, BarocasSe16}, failure of 
ML to perform robustly across sub-population or under distribution shift 
can have disastrous real-life consequences, particularly for members of 
underserved and/or under-represented groups.

Distributionally robust optimization (DRO) is emerging as a methodology 
for imposing the constraint that models perform uniformly well across 
subgroups, and several works conduct experiments 
demonstrating its benefit in promoting 
fairness~\cite{HashimotoSrNaLi18,DuchiHaNa20,WangGuHaCoGuJo20} and 
robustness~\cite{SinhaNaDu18,SagawaKoHaLi20} in ML. However, the 
computational experiments in these works are relatively small in scale, and 
there exist serious computational impediments to scaling up DRO. 
Consequently, the potential benefits of several DRO formulations 
remain unexplored.

The  main contribution of our work is in strengthening the theoretical and 
algorithmic foundations of two fundamental DRO formulations. In 
particular, for $\cs$-divergence uncertainty sets we give the first proof that 
stochastic gradient methods can scale to large data similarly to they way 
they scale for standard empirical risk minimization. We believe that our 
algorithms will serve a basis for future experimentation with CVaR and 
$\cs$ 
divergence DRO, and we hope that the resulting findings would lead to 
more robust and fair machine learning algorithms with positive societal 
impact. Towards that end, we will release an implementation of our DRO 
gradient estimators that integrates seamlessly into the PyTorch 
optimization framework and is therefore suitable for application in a wide 
range of ML tasks.

In addition, we believe that our work is a step towards a suite of 
algorithms capable of solving a broader class of DRO problems at scale, 
including e.g., uncertainty set with explicit group structure as proposed 
in~\cite{HuNiSaSu18,SagawaKoHaLi20}. We believe that such algorithm 
suite will empower machine learning researchers and engineers to create 
more reliable and ethical systems.

However, greater applicability and simplicity always comes with the risk of 
irresponsible and superficial use. In particular, we are concerned with the 
possibility that DRO might become a marketing scheme to sell off ML 
systems as robust without proper verification. Therefore, the development 
of robust training procedures must go hand-in-hand with the development 
of rigorous and independent evaluation methodologies for auditing of 
algorithms~\cite{HendrycksDi19,OKellySiNaDuTe18,KearnsNeRoWu18, 
CorbettGo18, MitchellWuZaBaVaHuSpRaGe19}.

 }

\arxiv{\subsection*{Acknowledgments}}
\notarxiv{\subsubsection*{Acknowledgments}} The authors would like to thank
Hongseok Namkoong for discussions and insights, as well as Nimit Sohoni for
comments on an earlier draft. \notarxiv{\subsubsection*{Sources of funding}}DL,
YC and JCD were supported by the NSF under CAREER Award CCF-1553086 and HDR
1934578 (the Stanford Data Science Collaboratory) and Office of Naval Research
YIP Award N00014-19-2288.  YC was supported by the Stanford Graduate Fellowship.
AS is supported by a Microsoft Research Faculty Fellowship, NSF CAREER Award CCF-1844855, NSF Grant CCF-1955039, a PayPal research gift, and a Sloan Research Fellowship.

\notarxiv{\subsubsection*{Competing interests}
The authors declare no competing interests.
}

\newpage

\bibliographystyle{abbrvnat-nourl} %

\newpage

\appendix
\part*{Appendix}
\section{Extended preliminaries}\label{app:prelims}

In this section we collect several basic results which we use in subsequent 
derivations in the paper: Section~\ref{app:prelims-char} gives several 
additional characterization of the robust objective $\L$, 
Section~\ref{app:prelims-compute} briefly discusses the computation of 
$\L$ and its costs, Section~\ref{app:dual-sgm} gives a short 
derivation of the complexity guarantees for ``dual SGM'' in 
Table~\ref{table:summary}, and 
Section~\ref{app:prelims-cs-bounded} introduces the 
notion of losses contained in a $\cs$ divergence ball.
Finally, Section~\ref{app:prelims-general} lists a few standard probabilistic 
bounds.

\subsection{Characterization of the robust 
objective}\label{app:prelims-char}

Here we give several equivalent characterizations of the robust objective
\begin{equation}\label{eq:pop-dro-reg-restated}
\L(\x;P) \defeq \sup_{Q\ll P:\div_\phi(Q, P) \le 
	\rho}\crl[\Big]{\E_{S\sim Q}\brk{\ell(\x;\S)} - \lambda\div_\psi(Q, P)}.
\end{equation}
where $\psi$, $\phi$ are closed convex functions from $\R_+$ to $\R$ 
satisfying $\psi(1)=\phi(1)=0$,
\begin{equation*}
\div_\phi(Q,P)\defeq \int \phi\prn*{\frac{\d Q}{\d P}}\d P,
~~\mbox{and}~~
\div_\psi(Q,P)\defeq \int \psi\prn*{\frac{\d Q}{\d P}}\d P.
\end{equation*}
For $\Phat{s_1^n}$ uniform on $s_1, s_2, \ldots, s_n$ (which we abbreviate 
$s_1^n$), we write 
\begin{equation}\label{eq:emp-dro-reg-restated}
\L(x;s_1^n) \defeq \L(x; \Phat{s_1^n}) =
\sup_{q\in\Delta^n: \sum_{i\le n}\frac{1}{n}\phi(nq_i) \le \rho}
\crl[\Bigg]{\sum_{i \le n}\prn*{q_i \ell(\x;s_i) - \tfrac{1}{n}\psi(nq_i)}}.
\end{equation}

\subsubsection{Inverse-cdf formulation}

Instead of expressing the objective in terms of distribution over $\ss$, we can
characterize the robust loss in terms of the inverse cdf of the distribution (over
$\R$) of $\ell(x;S)$. Let $F^{-1}$ denotes the inverse cdf of $\ell(x;S)$ under
$P$.  Note that $\ell(x;S)$ with $S\sim P$ is equal in distribution to
$F^{-1}(U)$ with $U\sim \unirv([0,1])$.  Therefore,
\begin{flalign}
\L(\x;P) &\defeq \sup_{Q':\div_\phi(Q', \unirv([0,1])) \le 
	\rho}\crl[\Big]{
\E_{U\sim Q'}\brk{F^{-1}(U)} - \lambda\div_\psi(Q', \unirv([0,1]))}
\nonumber \\&
= \sup_{r\in\rset} \int_0^1 \brk[\Big]{r(u) F^{-1}(u) - \lambda\psi(r(u))}\d u,
\label{eq:icdf-dro}
\end{flalign}
where the last equality follows from writing $r(u)=\frac{\d Q'}{\d 
\unirv([0,1])}(u)$, and the set $\rset$ is
\begin{equation}\label{eq:rset-def}
\rset \defeq \crl*{ r:[0,1]\to \R_+ ~\bigg|~ \int_0^1 r(u)\d u = 
1~~\mbox{and}~~\int_0^1 \phi(r(u))\d u \le \rho}.
\end{equation}

\subsubsection{Dual formulation}\label{app:prelims-dual}

We can convert the maximization over $r$ in Eq.~\eqref{eq:icdf-dro} (or $Q$
in~\eqref{eq:pop-dro-reg-restated}) with minimization over Lagrange
multipliers for the constraint that $r$ sums to 1 and the $\phi$-divergence
constraint, yielding
\begin{flalign}
  \L(\x;P) &= \inf_{\eta \in \R, \nu \ge 0}\Upsilon(x, \eta, \nu;P),~
  \mbox{where}~ \nonumber \\&
  \Upsilon(x, \eta, \nu;P) \defeq \int_0^1 \sup_{r\in \R_+}
  \brk[\Big]{r F^{-1}(u) - \eta (r-1) - \nu(\phi(r)-\rho) -\lambda\psi(r)}\d u,
  \label{eq:dual-exp-gen}
\end{flalign}
where the strong duality follows~\citet[Sec. 3.2]{Shapiro17}.
Writing $(g)^*[v] \defeq \sup_{t\in\mathrm{dom}(g)} \crl*{vt - g(t)}$ for the conjugate
function of $g$, we may express $\Upsilon$ as 
\begin{equation}\label{eq:pop-rob-dual}
\Upsilon(x, \eta, \nu;P) = 
\int_0^1 (\nu \phi + \lambda \psi)^*[F^{-1}(u)-\eta] \d u + \eta + \nu 
\rho
= \E (\nu \phi + \lambda \psi)^*[\ell(x;S)-\eta] + \eta + \nu\rho,
\end{equation} 
where the expectation is over $S \sim P$, \ie the distribution from which we observe
samples. On a finite sample $s_1^n$ we have
\begin{equation*}
\Upsilon(x, \eta,\nu; s_1^n) \defeq 
\Upsilon(x, \eta,\nu; \Phat{s_1^n})
= \frac{1}{n}\sum_{i\le n} (\nu \phi + \lambda \psi)^*[\ell(x;s_i)-\eta] + 
\eta + \nu\rho.
\end{equation*}

For pure-constraint objectives (with $\psi=0$), $\Upsilon$ simplifies to
\begin{equation}
\psi = 0 \implies 
\Upsilon(x, \eta,\nu;P) = \nu \E_{S\sim P} \phi^* 
\brk*{\frac{\ell(x;S)-\eta}{\nu}} 
+ \eta + \nu\rho.
\end{equation}
For pure-penalty objective (with $\phi = 0$) the Lagrange multiplier $\nu$ is unnecessary 
and we have
\begin{equation}\label{eq:penalty-dual}
\phi = 0 \implies 
\Upsilon(x, \eta;P) = \lambda \E_{S\sim P} \psi^* 
\brk*{\frac{\ell(x;S)-\eta}{\lambda}} 
+ \eta.
\end{equation}

Note that $\Upsilon$ is an expectation (i.e., an empirical risk) which means 
that to minimize $\L(x;P)$ we can, in principle, apply ERM jointly on 
$x,\eta$ and $\nu$, as we further discuss in \Cref{app:dual-sgm}.

Finally, we note that any $Q\opt$ attaining the supremum 
in~\eqref{eq:pop-dro-reg-restated} is of the form
\begin{equation*}
\frac{\d Q\opt}{\d P}(s) = {(\nu\opt + \lambda\psi)^*}'[\ell(x;s)-\eta\opt].
\end{equation*}
where $\eta\opt$ and $\nu\opt$ are optimal Lagrange multipliers 
in~\eqref{eq:dual-exp-gen} and ${(\nu\opt + \lambda\psi)^*}'$ is a 
subderivative of $(\nu\opt + \lambda\psi)^*$. For $\phi=0$ this 
specializes to
\begin{equation*}
\frac{\d Q\opt}{\d P}(s) = 
{\psi^*}'\brk*{\frac{\ell(x;s)-\eta\opt}{\lambda}}.
\end{equation*}
For a finite sample, we have
\begin{equation}\label{eq:emp-q-opt-expression}
q\opt_i = \frac{1}{n} {\psi^*}'\brk*{\frac{\ell(x;s_i)-\eta\opt}{\lambda}}.
\end{equation}

\subsubsection{Expressions for CVaR}

Recall that CVaR at level
$\alpha$ corresponds to $\phi=0$ and $\psi = \convind_{[0, 1/\alpha)}$.  The
  dual expression of CVaR simplifies to~\cite[Example 6.16]{ShapiroDeRu09}
\begin{equation*}
\Lcvar(x;P) = \inf_{\eta \in \R} \crl*{\frac{1}{\alpha} \E_{S\sim P} 
(\ell(x;S)-\eta)_+ + \eta}.
\end{equation*}
It also has a simple closed-form expression in terms of the inverse cdf of 
$\ell(x;S)$~\cite[Theorem 6.2]{ShapiroDeRu09}:
\begin{equation}\label{eq:cvar-cdf}
\Lcvar(x; P) = \frac{1}{\alpha}\int_{1-\alpha}^1 
F^{-1}(u)\mathrm{d}u.
\end{equation}
We note that this last expression is a direct consequence
of~\eqref{eq:icdf-dro}, since $\rset$ is the set of measures never exceeding
$\frac{1}{\alpha}$. On a finite sample $s_1^n$ this gives the closed-form
expression
\begin{equation}\label{eq:cvar-sample-closed-form}
\Lcvar(x;s_1^n) = \frac{1}{\alpha 
	n}\sum_{i=1}^{\floor{\alpha n}} \ell(x;s_{(i)})
+ 
\left(1 -
\frac{\floor{\alpha n}}{\alpha n}\right)
\ell(x;s_{(\floor{\alpha n}+1)}),
\end{equation}
where $s_{(1)}, \ldots, s_{(n)}$ are a permutation of $s_1^n$ satisfying 
$\ell(x;s_{(1)}) \ge \ell(x;s_{(2)}) \ge \cdots \ge \ell(x;s_{(n)})$. For $\alpha 
\le 1/n$ we simply have $\Lcvar(x;s_1^n) = \max_{i\le n} \ell(x;s_i)$.

The KL-divergence penalized CVaR at level $\alpha$ corresponds to $\psi(t) 
= \convind_{[0, 1/\alpha]}(t) + t\log t -t +1$, for which 
\begin{equation*}
\psi^*[v] = \begin{cases}
e^v- 1 & v < \log \frac{1}{\alpha} \\
\frac{1}{\alpha} - 1 + \frac{1}{\alpha}(v - \log\frac{1}{\alpha}) & 
\text{otherwise}, \\
\end{cases}
\end{equation*}
and the dual expression for $\LcvarR$ is given by~\eqref{eq:penalty-dual}. 
In the special case $\alpha \le 1/n$ the CVaR constraint becomes inactive, 
and we can minimize over $\eta$ in closed form to obtain the standard ``soft max'' objective 
$\LcvarR(x;s_1^n) = \lambda \log\prn[\big]{\frac{1}{n}\sum_{i\le n} 
\exp\prn*{\ell(x;s_i)/\lambda}}$.

\subsubsection{Expressions for $\Llam$ and $\Lcs$}
The penalized version of the $\cs$ objective corresponds to $\phi(t)=0$ 
and $\psi(t) = \half (t-1)^2$. Note that $\div_\phi(Q,P)$ is invariant under 
$\psi(t)\mapsto \psi(t) + c \cdot (t-1)$ for any $c\in\R$ because $\int 
(\frac{\d Q}{\d P}-1)\d P=0$. We find it more convenient to work 
with $\psi(t) = \half (t-1)^2 + (t-1) = \frac{1}{2}(t^2 - 1)$, for which the 
conjugate is simply  $\psi^*[v] = \half ((v)^2_+ +1)$. The 
dual form~\eqref{eq:penalty-dual} gives
\begin{equation}\label{eq:lam-dual}
\Llam(x;P) = \inf_{\eta\in\R} \crl*{
	\frac{1}{2\lambda}\E_{S\sim P}  {(\ell(x;S)-\eta)_+^2}  + \frac{\lambda}{2} 
	+ 
	\eta}.
\end{equation}
The infimum is attained at the $\eta\opt$ solving $\E (\ell(x;S)-\eta\opt)_+ 
= \lambda$. In other words, 
\begin{equation*}
\eta\opt = \E\brk*{\ell(x;S) \mid 
\ell(x;S) \ge \eta\opt} - \frac{\lambda}{\P(\ell(x;S)\ge\eta\opt)}
=
\Lcvar^{F(\eta\opt)}(x;P) - \frac{\lambda}{1-F(\eta\opt)},
\end{equation*}
where $F(t) = \P(\ell(x;S)\le t)$ is the cdf of $\ell(x;S)$. Letting 
$\mathfrak{G}(\eta\opt)$ denote the event that $\ell(x;S)\ge \eta\opt$, 
substituting back to 
the expression for $\Llam$ gives
\begin{equation*}
\Llam(x;P) = \E\brk*{\ell(x;S) \mid 
	\mathfrak{G}(\eta\opt)} + \frac{1}{2\lambda}
	\Var\brk*{\ell(x;S)\mid \mathfrak{G}(\eta\opt)} + 
	\frac{\lambda}{2}\prn*{\frac{1}{\P(\mathfrak{G}(\eta\opt))}-1}^2.
\end{equation*}
In words, $\Llam$ is a sum of a CVaR (at level $F(\eta\opt)$), a conditional 
variance regularization term and an outage probability regularization term. 
This expression simplifies considerably when $\lambda$ is sufficiently 
large. Specifically, we have,
\begin{flalign}\label{eq:lam-edge-case}
\lambda \ge B &\implies \lambda \ge \E\ell(x;S)-F^{-1}(0) 
\nonumber \\& 
\implies \eta^\star = \E\ell(x;S) - \lambda
~~\mbox{and}~~\P(\mathfrak{G}(\eta\opt))=1  
\nonumber\\& 
\implies \Llam(x;P) = \E\ell(x;S) + \frac{1}{2\lambda} \Var[\ell(x;S)].
\end{flalign}
That is, for sufficiently large $\lambda$ the objective $\Llam$ is simply the 
empirical risk with variance regularization (see also~\cite{DuchiNa20}).

For a finite sample we have
\begin{equation*}
\Llam(x;s_1^n) = 
	\frac{1}{2\lambda n}\sum_{i\le n} {(\ell(x;s_i)-\eta\opt_n)_+^2}  + 
	\frac{\lambda}{2} 
	+ 
	\eta\opt_n.
\end{equation*}
Where $\eta\opt_n$ is the solution to $\sum_{i\le n} 
{(\ell(x;s_i)-\eta\opt_n)_+} = n\lambda$, or equivalently
\begin{equation}\label{eq:lam-emp-eta-expression}
\eta\opt_n = \frac{1}{i\opt} \sum_{i\le i\opt} \ell(x;s_{(i)}) - \frac{\lambda 
n}{i\opt}~~\mbox{for the unique $i\opt$ such that}~~\ell(x;s_{(i\opt+1)})\le 
\eta\opt_n \le  ~\ell(x;s_{(i\opt)}),
\end{equation}
where $\{\ell(x;s_{(i)})\}$ are the sorted $\{\ell(x;s_i)\}$ and 
$\ell(x;s_{(n+1)})\defeq -\infty$.

An expression for $\Lcs$ follows via~\eqref{eq:lam-dual}
\begin{equation}\label{eq:cs-dual}
\Lcs(x;P) = \inf_{\lambda \ge 0} \crl*{\Llam(x;P) + \lambda\rho}
= \inf_{\eta \in \R} \crl*{ \sqrt{1+2\rho} \sqrt{ \E_{S\sim P}  
{(\ell(x;S)-\eta)_+^2}} + \eta },
\end{equation}
and the maximizing $Q$ is
\begin{equation}\label{eq:cs-opt-q}
\frac{\d Q\opt}{\d P}(s) = \frac{(\ell(x;s)-\eta\opt)_+}{
	\E_{S\sim P} (\ell(x;S)-\eta\opt)_+}.
\end{equation}

\subsubsection{Expression for $\grad\L$}

Let $Q\opt$ by a
distribution attaining the supremum in~\eqref{eq:pop-dro-reg-restated} and
recall that $\grad \ell(x;s)$ denotes an element in the sub-differential of
$\ell(x;s)$ w.r.t.\ $x$. Then the following vector is a subgradient of
$\L$~\cite[Corollary 4.4.4]{HiriartUrrutyLe93},
\begin{equation*}
\grad \L(x;P) = \E_{S\sim Q\opt} \grad \ell(x;S).
\end{equation*}
Similarly, for a sample of size $n$ and a maximizing $q\opt$, we have
\begin{equation}\label{eq:gradient-expression-emp}
\grad \L(x;s_1^n) = \sum_{i\le n} q\opt_i \grad \ell(x;s_i).
\end{equation}

\subsubsection{Smoothness of $\Llam$ and 
$\LcvarR$}\label{app:prelims-smoothness} 
The smoothness of $\L$ (i.e., Lipschitz continuity of its gradient) plays a 
role in our mini-batch gradient estimator complexity 
guarantees.  
 When the penalty term $\psi$ is strongly
convex, the maximizing $Q\opt$ (or $q\opt$) is unique, and if $\grad \ell$ is
$H$-Lipschitz then $\L$ is differentiable~\cite[Corollary
  4.4.5]{HiriartUrrutyLe93}. In particular, writing $Q\opt_x$ for the maximizing
$Q$ at point $x$, we have
\begin{flalign}
&\norm{\grad \L(x;P) - \grad \L(y;P)}
= \norm*{\int \crl*{\grad \ell(x;s) \d Q\opt_x(y) - 
\grad \ell(y;s) \d Q\opt_y(s)} }
\nonumber\\ & \quad \quad
\le \int {\norm{\grad \ell(x;s)-\grad \ell(y;s)} \d Q\opt_x(s)}
	+ \int \norm{\grad \ell(y;s)} \abs*{\frac{\d Q\opt_x}{\d P}(s) - 
	\frac{\d Q\opt_y}{\d P}(s)} \d P(s)
\nonumber\\ & \quad \quad
\le H \norm{x-y} + G \norm*{ Q\opt_x - Q\opt_y}_1.
\label{eq:smoothness-bound}
\end{flalign}

Therefore, if $Q\opt_x$ is Lipschitz w.r.t.\ $x$ in the 1-norm then $\grad 
\L$ is Lipschitz as well. This is indeed the case $\Llam$ and $\LcvarR$. 

\restateClaimSmoothObj*
\begin{proof}
	Since entropy is 1-strongly-convex w.r.t.\ the 1-norm, for $\LcvarR$ 
we 
	have that the penalty $\lambda \psi$ is $\lambda$-strongly-convex 
	w.r.t.\ the 1-norm and therefore~\cite[Lemma 2]{ShalevSi06c}
	\begin{equation*}
	\norm*{ Q\opt_x - Q\opt_y}_1 \le 
	\frac{1}{\lambda}\norm{\ell(x;\cdot)-\ell(y;\cdot)}_\infty \le 
	\frac{G}{\lambda}\norm{x-y},
	\end{equation*}
	which by~\eqref{eq:smoothness-bound} implies that $\grad \LcvarR$ is 
	$(H + G^2/\lambda)$-Lipschitz as required. For $\Llam$, we find it 
	easier to argue for a finite sample $s_1^n$. 
	By~\eqref{eq:emp-dro-reg-restated} we have $q\opt_x = 
	\argmax_{q\in\Delta^n}\crl*{q\top \ell(x) - \half \lambda n 
	\norm{q}_2^2}$, where $\ell_i(x) = \ell(x;s_i)$. Therefore, by $\lambda 
	n$ -strong-convexity w.r.t.\ the 2-norm, we have
	\begin{equation*}
	\norm{q\opt_x-q\opt_y}_1 \le \sqrt{n} \norm{q\opt_x-q\opt_y}_2
	\le \frac{1}{\lambda \sqrt{n}} \norm{\ell(x)-\ell(y)}_2
	\le \frac{G}{\lambda}\norm{x-y},
	\end{equation*}
	establishing that $\grad \Llam$ is also 
	$(H + G^2/\lambda)$-Lipschitz. 
	
	Finally, we note that $\LcvarR(x;P)\le \Lcvar$ because 
	$D_\psi(Q,P)\ge 0$ for all $Q$. Conversely since any feasible $Q$ 
	satisfies $\d Q /\d P 
	\le 1/\alpha$ we have $D_\psi(Q,P) = \int \d Q \log 
	\frac{\d Q}{\d P} \le 
	\log\frac{1}{\alpha}$ and therefore $\LcvarR(x;P) \ge \Lcvar(x;P) - 
	\lambda \log\frac{1}{\alpha}$.
\end{proof}

\subsection{Computational cost}\label{app:prelims-compute}
To compute $\L(x;s_1^n)$ and its (sub)gradient from $\{\ell(x;s_i)\}_{i\le n}$
and $\{\grad \ell(x;s_i)\}_{i\le n}$ we compute $q\opt$ that
maximizes~\eqref{eq:emp-dro-reg-restated} and substitute it back
in~\eqref{eq:gradient-expression-emp}. The substitution requires $O(nd)$ work,
so it remains to account for the work in computing $q\opt$. 

For CVaR, this
clearly amounts to sorting $\{\ell(x;s_i)\}_{i\le n}$ and therefore takes
$O(n\log n)$ time. Similarly, for $\Llam$ we may find sort the losses and find
$i\opt$ in~\eqref{eq:lam-emp-eta-expression}, and hence $\eta\opt_n$ and
$q\opt$, in $O(n)$ time. Alternatively, for any objective with $\phi=0$
(including $\Llam$ and $\LcvarR$ we can bisect directly on $\eta$, either to
minimize the expression~\eqref{eq:penalty-dual} or to satisfy the the simplex
constraint
$\sum_{i\le n}q\opt_i = \frac{1}{n} \sum_{i\le n}
(\psi^*)'[(\ell(x;s_i)-\eta\opt)/\lambda] = 1$.

For $\Lcs$ we may find $q\opt$ by 
performing similar bisection over $\eta$ via the 
expression~\eqref{eq:cs-dual}, again either minimizing it or solving for the 
 condition $\frac{1}{n}\sum_{i\le n} (\ell(x;s_i)-\eta\opt)_+^2 = 
(1+2\rho) \prn[\big]{\frac{1}{n}\sum_{i\le n} 
(\ell(x;s_i)-\eta\opt)_+}^2$. 
Finding an 
$\varepsilon$ accurate solution via bisection requires roughly $n \log 
\frac{B}{\varepsilon}$ time.

Since we are interested in large-scale application, we assume that $d \gg 
\log(nB/\varepsilon)$ and therefore the time to compute the objective and 
its gradient is $O(nd)$.

For simplicity and stability, our code implements the computation of 
$q\opt$ using bisection over $\eta$ for each of $\Llam,\Lcs$ and 
$\LcvarR$.  

\subsection{Stochastic gradient method on the dual
  objective}\label{app:dual-sgm}

\newcommand{\etaL}{\underline{\eta}}
\newcommand{\etaH}{\overline{\eta}}

Here we discuss the convergence guarantees for a simple stochastic 
gradient method using the dual expression~\eqref{eq:dual-exp-gen} for 
$\L(x;P_0)$ in order to minimize it over $x$. While several works consider 
such methods (see \Cref{sec:related}), we could not find direct reference for 
their runtime guarantees, and we therefore briefly derive it below.

Focusing on objectives with 
$\phi=0$ (as in~\eqref{eq:penalty-dual}), and writing $\gamma_x$ 
and $\gamma_\eta$ for step sizes, 
we write the iterations on $x$ and the Lagrange multiplier $\eta$ as
\begin{flalign}
x_{t+1} &= 
\Pi_{\X}\prn*{x_t - \gamma_x \grad \Upsilon(x_t,\eta_t; P_0) }  = 
\Pi_{\X}\prn*{x_t - \gamma_x {\psi^*}'\brk*{
		 \frac{\ell(x;S_i) - \eta_t}{\lambda} } 
\grad \ell(x;S_i)},~~\mbox{and} \nonumber\\
\eta_{t+1} &= 
\Pi_{[\etaL, \etaH]}\prn*{\eta_t - \gamma_\eta \frac{\del }{\del\eta} 
\Upsilon(x_t,\eta_t; 
P_0)}
	= 
\Pi_{[\etaL, \etaH]}\prn*{\eta_t + 
\gamma_\eta {\psi^*}'\brk*{
	\frac{\ell(x;S_i) - \eta_t}{\lambda} } - \gamma_\eta},
\label{eq:dual-sgm}
\end{flalign}
Where $S_1,S_2,\ldots$ are drawn iid from $P_0$. 

For CVaR, we have $(\psi^*)'[v] = \frac{1}{\alpha}\indic{v \ge 0}$ and we 
may restrict $\eta$ to the range $[\etaL,\etaH]=[0, B]$, as the optimal 
$\eta$ is the value at risk level $\alpha$ and therefore in the range of 
$\ell$. For $\Llam$ we have $(\psi^*)'[v] = (v)_+$ and we may take 
$[\etaL,\etaH]=[-\lambda, B]$ due to the condition 
$\E (\ell(x;S)-\eta\opt)_+ 
= \lambda$. In these settings, the method~\eqref{eq:dual-sgm} has the 
following guarantee
\begin{claim}
	Let $\epsilon\in(0,B)$. 
	For CVaR and a suitable choice of $\gamma_x, \gamma_\eta$ the 
	average iterate 
	$\bar{x}_T = \frac{1}{T}\sum_{t\le T} x_t$ satisfies
	\begin{equation*}
	\E \Lcvar(\bar{x}_T; P_0)  - \min_{x'\in\X} \Lcvar(x';P_0) \le \epsilon
	~~\mbox{for}~~
	T \asymp \frac{(GR)^2 + B^2}{\alpha^2 \epsilon^2}.
	\end{equation*}     
	Similarly, for $\cs$ penalty we have
	\begin{equation*}
	\E \Llam(\bar{x}_T; P_0)  - \min_{x'\in\X} \Llam(x';P_0) \le \epsilon
	~~\mbox{for}~~
	T \asymp \frac{(GR)^2+B^2}{\epsilon^2}\prn*{1+\frac{B^2}{\lambda^2}}.
	\end{equation*}     
\end{claim}
\begin{proof}
	By Proposition~\ref{prop:gangster}, the expected sub-optimality of 
	$\bar{x}_T$ is $\lesssim (\Gamma_x R + \Gamma_\eta 
	(\etaH-\etaL)/\sqrt{T}$, 
	where $\Gamma_x^2$ (respectively $\Gamma_\eta$) is an upper bound 
	on the second moment of $\grad \Upsilon(x,\eta;S)$ (respectively  
	$\frac{\del}{\del\eta}  \Upsilon(x,\eta;S)$). 
	 For $\Lcvar$ we have $\Gamma_x\le G/\alpha$, $\Gamma_\eta \le 
	 1/\alpha$ and $\etaH-\etaL=B$. For $\Llam$ we have $\Gamma_x \le 
	 G(1+B/\lambda)$, $\Gamma_\eta = 1+B/\lambda$ and 
	 $\etaH-\etaL=B+\lambda$.
	 The result follows from substituting $T \asymp 
	 \prn*{\Gamma_x^2 R^2  + \Gamma_\eta^2 (\etaH-\etaL)^2} 
	 \epsilon^{-2}$.
\end{proof}

\subsection{Uncertainty sets contained in $\cs$ divergence 
balls}\label{app:prelims-cs-bounded}

A number of our results hold for general subclass of the 
objective~\eqref{eq:pop-dro-reg-restated} with the following property.
\begin{definition}[$\cs$-bounded objective] \label{def:cs-bounded}
	An objective $\L(x;P_0)$ is 
$\CSbdd$-$\cs$-bounded if for all $x$ and all $Q\opt$ attaining the 
supremum in~\eqref{eq:pop-dro-reg-restated} we have $\divcs(Q\opt, P_0) 
\le \CSbdd$.  
\end{definition}

The three objectives we focus on are $\cs$-bounded.
\begin{claim}\label{claim:cs-bdd}
The objectives $\LcvarR,\Lcs$ and $\Llam$ are $\cs$-bounded with 
constants $\CSbdd=\frac{1}{\alpha}-1$,  $\CSbdd=\rho$ and 
$\CSbdd=B/\lambda$, respectively.
\end{claim}
\begin{proof}
	That $\Lcs$ is $\rho$-$\cs$-bounded is obvious from definition. For 
	$\Llam$ we have
	\begin{flalign*}
	\Llam(x;P_0) &= \E_{S\sim Q\opt} \ell(x;S) - \lambda \divcs(Q\opt, P_0)
	\\& \ge \E_{S\sim P_0} \ell(x;S) - \lambda \divcs(P_0;P_0) = \E_{S\sim 
	P_0} 
	\ell(x;S) 
	\end{flalign*}
	and consequently
	\begin{equation*}
	\divcs(Q\opt, P_0) \le \frac{E_{S\sim Q\opt} \ell(x;S)-\E_{S\sim P_0} 
	\ell(x;S) }{\lambda} \le \frac{B}{\lambda}.
	\end{equation*} 
	Finally, for $\LcvarR$ every feasible $Q$ satisfies $\d Q/\d P_0 \le 
	1/\alpha$ and therefore
	\begin{equation*}
	\divcs(Q,P_0) = \int \prn*{\frac{\d Q}{\d P_0}(s)}^2 \d P_0(s) - 1
	\le \frac{1}{\alpha} \prn*{\frac{\d Q}{\d P_0}(s)}\d P_0(s) - 1
	= \frac{1}{\alpha}  - 1.
	\end{equation*}
\end{proof}

\subsection{General results}\label{app:prelims-general}
We conclude this section of the appendix by stating three general results 
that aid our analysis. First, we give a lemma stating that a binomial random 
variable with parameters $n$ and $\alpha$ has a constant probability of 
being at least $\sqrt{\alpha(1-\alpha)n}$ below its mean.

\begin{lemma}\label{lem:bin-tail-be}
  Let $n\in\N$ and $\alpha \in (0,1)$. There exists a numerical constant
  $C\in\R$ such that
  \begin{equation*}
    \P\prn[\big]{ \binrv(n,\alpha) \le n\alpha - \sqrt{n\alpha(1-\alpha)} } 
    \ge \P(\mc{N}(0,1) \le -1) -
    \frac{C}{\sqrt{\alpha(1-\alpha)n}}.
  \end{equation*}
\end{lemma}
\begin{proof}
  Note that
  $\P\prn[\big]{ \binrv(n,\alpha) \le n\alpha - \sqrt{n\alpha(1-\alpha)} } =
  \P(Y\sqrt{n} \le -1)$ where
  $Y=\frac{1}{n}\frac{\binrv(n,\alpha) - n\alpha}{\sqrt{\alpha(1-\alpha)}}$ is
  the mean of $n$ independent random variable with zero mean, unit variance, and
  absolute third moment
  $\rho = \frac{\alpha^2 + (1-\alpha)^2}{\sqrt{\alpha(1-\alpha)}} \le
  \frac{1}{\sqrt{\alpha(1-\alpha)}}$. The Berry-Esseen
  theorem~\cite[Theorem 3.4.17]{Durrett19} states that for such $Y$ we have
  $|\P(Y\sqrt{n} \le t) - \P(\mc{N}(0,1) \le t)| \le C\rho/\sqrt{n}$, for all
  $t\in \R$; substituting $t=1$ and $\rho\le \frac{1}{\sqrt{\alpha(1-\alpha)}}$
  concludes the proof.
\end{proof}

Second, we state the Efron-Stein inequality in vector form, which follows 
from applying the standard scalar bound element-wise.
\begin{lemma}[Efron-Stein inequality~{\cite[][Theorem 
3.1]{BoucheronLuMa13}}]
  \label{lem:efron-stein}
  Let $X_1^{n+1}$ be i.i.d random variables and $f:\mc{X}^n \to \R^m$. Let 
  $I$ be uniform on $\{1,\ldots,n\}$ and let $\tilde{X}_1^N$ be such that 
  $\tilde{X}_i = X_i$ for $i\ne I$ and $\tilde{X}_I = X_{n+1}$. Then
  \begin{equation}\label{eqn:efron-stein}
    \Var[ f(X_1^n) ]
    \le \frac{n}{2}
    \E{}\norm{f(X_1^n) - f(\tilde{X}_1^n)}^2.
  \end{equation}
\end{lemma}

Third, we give a general lemma on the variance of sampling without 
replacement, which we specialize to the simplex for later use.
\begin{lemma}\label{lem:subset-cs}
  Let $p \in \Delta^k$ and let $\mc{I}$ be a random subset of $[k]$ of size 
  $k/2$. Then 
  \begin{equation*}
  \E \prn[\Bigg]{\sum_{i\in\mc{I}} p_{i} - \half}^2 \le \half \norm[\Big]{p - 
  \tfrac{1}{k}\ones}^2 = \frac{1}{2k} \divcs(p, \tfrac{1}{k}\ones).
  \end{equation*}
\end{lemma}
\begin{proof}
  Let us denote $q = p - \tfrac{1}{k}\ones$. We have
  \begin{align*}
    \E\prn*{\sum_{i \le k}p_i\indic{i\in\mc{I}} - \frac{1}{2}}^2
    & = \E\prn*{\sum_{i\le k}q_i\indic{i\in\mc{I}}}^2  \\&\stackrel{(i)}{=} 
    \frac{1}{2}\sum_{i\le k}q_i^2
      + \sum_{i\neq j}q_i q_j
      \E\indic{i\in\mc{I}\mbox{~and~}j\in\mc{I}} \\
    & \stackrel{(ii)}{=} \frac{1}{2}\norm{q}^2 + \frac{k-2}{4(k-1)}
      \sum_{i\le k}\sum_{j\neq i} q_i q_j  = \frac{1}{2}\norm{q}^2 + 
      \frac{k-2}{4(k-1)}
      \sum_{i\le k}q_i(1-q_i) \\
    & \stackrel{(iii)}{=} \prn*{\frac{1}{2} - \frac{k-2}{4(k-1)}}\norm{q}^2  \le 
    \frac{1}{2}\norm{q}^2,
  \end{align*}
  where $(i)$ stems from $\P(i \in \mc{I}) = \frac{1}{2}$, $(ii)$ from
  $\P(i \in I\mbox{~and~}j\in I) = \frac{1}{2}\frac{k/2-1}{k-1}$ and $(iii)$
  from $\sum_{i\le k}q_i = 0$. Noting that $\divcs(p, \tfrac{1}{k}\ones) = 
  k\norm{q}^2$ concludes the proof.
\end{proof}
\section{Proofs from Section~\ref{sec:batch}}\label{prf:batch}
This section completes  the proof and discussion of the results in 
Section~\ref{sec:batch}. First, in Section~\ref{app:bias}, we 
prove the bias bounds in Proposition~\ref{prop:batch-bias} and argue their 
tightness in the worst case. 
Section~\ref{app:batch-extra-assumptions} provides additional discussion 
of the smoothness and Lipschitz inverse-cdf assumptions sometimes used 
in this section. Then, in Section~\ref{prf:prop-variance-grad} we bound the 
variance of the mini-batch estimators for $\cs$-bounded penalty 
objectives and their gradient, obtaining 
Proposition~\ref{prop:variance-grad}  as a corollary. We also argue that 
similar bounds do not hold for the $\cs$ constraint objective. In 
Section~\ref{app:gangster} we review the standard convergence guarantees 
for stochastic gradients iterations with and without Nesterov acceleration, 
and in Section~\ref{prf:thm-complexity-batch} we combine all these 
ingredients to prove Theorem~\ref{thm:batch-complexity}.

\subsection{Bias of batch estimator}\label{app:bias}
\subsubsection{Proof of Proposition~\ref{prop:batch-bias}}\label{prf:prop-batch-bias}

\restatePropBatchBias*

\begin{proof}
	We first show that the bound $\L \ge \bL$ holds for any loss of the 
	form~\eqref{eq:pop-dro-reg-restated} and then proceed to show each of 
	the bounds~\eqref{eq:cvar-bias}--\eqref{eq:icdf-bias}. We 
	remark here that the bound~\eqref{eq:cs-bias} actually holds for any 
	$\rho$-$\cs$-bounded objective (Definition~\ref{def:cs-bounded}).

\paragraph{Proof of $\L(x;P_0) \ge \bL(x;n)$.}
The dual expression~\eqref{eq:pop-rob-dual} gives
\begin{flalign*}
\L(x;P_0) &= \inf_{\eta\in\R, \nu\ge 0} \E_{S\sim P_0}\crl*{
 (\nu \phi + \lambda \psi)^*[\ell(x;S)-\eta] + \eta + \nu\rho}
\\&
=
\inf_{\eta\in\R, \nu\ge 0}  \E_{S_1^n\sim P_0^n} \crl[\Bigg]{
	\frac{1}{n}\sum_{i\le n} (\nu \phi + \lambda 
	\psi)^*[\ell(x;S_i)-\eta] + \eta + \nu\rho}
\\&
 \ge \E_{S_1^n\sim P_0^n} \inf_{\eta\in\R, \nu\ge 0} \crl[\Bigg]{
	\frac{1}{n}\sum_{i\le n} (\nu \phi + \lambda 
	\psi)^*[\ell(x;S_i)-\eta] + \eta + \nu\rho} = \E \L(x;S_1^n) = \bL(x;n),
\end{flalign*}
where the inequality follows from exchanging the expectation and the 
infimum.

  \paragraph{Proof of the CVaR bias bound~\eqref{eq:cvar-bias}.}
  By Eq.~\eqref{eq:cvar-sample-closed-form} we have
  \begin{equation*}
  \bLcvar(x;n) = \E \Lcvar(x;S_1^n) = \frac{1}{\alpha 
  	n}\sum_{i=1}^{\floor{\alpha n}} \E\ell(x;S_{(i)})
  + 
  \left(1 -
  \frac{\floor{\alpha n}}{\alpha n}\right)
  \E \ell(x;S_{(\floor{\alpha n}+1)}),
  \end{equation*}
  where $\ell(x;S_{(i)})$ is the $i$th order statistic of $\ell(x;S_1^n)$ (in
  decreasing order). Recalling that $F$ denotes the cdf of $\ell(x;S)$, we may
  write $\ell(x;S)=F^{-1}(U)$ with $U$ uniform on $[0,1]$. Therefore,
  $\ell(x;S_{(i)})=F^{-1}(U_{(i)})$ where $U_{(i)}\sim \betarv(n-i+1,i)$ is the
  $i$th order statistic of $n$ iid $\unirv([0,1])$ random
  variables~\cite[Sec. 4.6]{Pitman93}.  Taking expectation, we have
  \begin{equation*}
  \E_{S_1^n\sim P_0^n} \ell(x;S_{(i)}) = \int_0^1 
  F_Z^{-1}(u)f_{\betarv(n-i+1, i)}(u)\mathrm{d}u,
  \end{equation*}
  where $f_{\betarv(a, b)}$ is the density function of the Beta random 
  variable of parameters $a, b$. Substituting back, we have
  \begin{flalign}
  \bLcvar(x;n) &= \frac{1}{\alpha}\int_0^1 \softind(u) F^{-1}(u)\d 
  u,~~\mbox{where}~~\nonumber \\&
  \softind(u) = \frac{1}{ n}\sum_{i=1}^{\floor{\alpha n}} f_{\betarv(n-i+1, 
  i)}(u) 
+\left(\alpha - \frac{\floor{\alpha n}}{ n}\right) 
  f_{\betarv(n-\floor{\alpha n}, \floor{\alpha n}+1)}(u).
 \label{eq:cvar-bias-soft-indicator}
  \end{flalign}
  
  Using
  \begin{equation}\label{eq:sum-of-betas}
    \begin{split}
  \frac{1}{n}\sum_{i=1}^n f_{\mathsf{Beta}(n-i+1, i)}(u) &=
  \frac{1}{n}\sum_{i=1}^n \frac{n!}{(n-i)!(i-1)!}u^{n-i}(1-u)^{i-1} \\& =
  \sum_{i=0}^{n-1}{n-1 \choose i-1}(1-u)^iu^{n-1-i} = 1,
  \end{split}
  \end{equation}
  we have that
  \begin{equation*}
  1-\softind(u) \le \frac{1}{n}\sum_{i=\floor{\alpha n}+1}^{n} 
  f_{\betarv(n-i+1, i)}(u).
  \end{equation*}
  Recalling Eq.~\eqref{eq:cvar-cdf} for $\Lcvar(x;P_0)$, and recalling that 
  $F^{-1}(u)\in [0,B]$ for all $u$ by assumption, we bound the bias as
\begin{align}
  \Lcvar(x; P_0) - \bLcvar(x;n)
  &= \frac{1}{\alpha}\int_0^1 \brk*{\indic{u\ge 1-\alpha} - \softind(u)} 
  F^{-1}(u) \d u
  \nonumber \\& \le  \frac{B}{\alpha}\int_{1-\alpha}^1 \brk*{\indic{u\ge 
  1-\alpha} - 
  \softind(u)} \d u  \nonumber \\
  & \le \frac{B}{\alpha}\int_{1-\alpha}^1 
  \left(\frac{1}{n}\sum_{i=\floor{\alpha n}+1}^{n} f_{\betarv(n-i+1, 
  i)}(u)\right)\d u \nonumber \\
  & = \frac{B}{\alpha n} \sum_{i=\floor{\alpha n} + 1}^{n} 
  \P(\mathsf{Beta}(n-i+1, i) \ge 1-\alpha).
  \label{eq:cvar-bias-tail-sum}
\end{align}

To conclude, it suffices to bound the tail probability of the Beta random
variables. We have~\cite[see, e.g.,][Ex. 5 in Sec. 4.6]{Pitman93}
\begin{align*}
  \P(\betarv(n-i+1, i) \ge 1-\alpha)
  & = 1 - \P(\betarv(n-i+1, i) \le 1-\alpha) \\
  & = \P(\binrv(n;1-\alpha) \le n-i) = \P(\binrv(n;\alpha) \ge i),
\end{align*}
and the multiplicative Chernoff bound~\cite[Theorem 4.3]{MotwaniRa95} 
gives
\begin{equation*}
\P(\binrv(n;\alpha) \ge i) \le \exp\prn*{
-\frac{i-n\alpha}{3}\min\crl*{ \frac{i-n\alpha}{n\alpha}, 1}}.
\end{equation*}
Therefore, for $\alpha n \ge 9$,
\begin{align*}
  \sum_{i = \floor{\alpha n} + 1}^n\P(\binrv(n;\alpha) \ge i)
  & \le \sum_{i=\floor{\alpha n}+1}^{2\floor{\alpha 
  n}}\exp\prn*{-\frac{(i-n\alpha)^2}{3n\alpha}}
    + \sum_{i=2\floor{\alpha n} + 1} ^\infty \exp\prn*{-\frac{i-n\alpha}{3}} 
    \\
  & \le 1 + \int_0^\infty \exp\prn*{-\frac{u^2}{3n\alpha}}\mathrm{d}u +
    \exp\prn*{-\frac{2\floor{\alpha n} +1 - \alpha n}{3}}\sum_{i=0}^\infty 
    e^{-i/3} \\
  & \le 1 + \frac{\sqrt{3\pi\alpha n}}{2} + \frac{e^{-3}}{1 - e^{-1/3}}  \le 
  3\sqrt{\alpha n}.
\end{align*}
Substituting into~\eqref{eq:cvar-bias-tail-sum} and using $\Lcvar(x;P_0) 
\le B$ when $\alpha n \le 9$ gives the final bound
\begin{equation}\label{eq:cvar-bias-precise}
  \Lcvar(x; P_0) - \bLcvar(x;n) \le B \min\crl*{\frac{3}{\sqrt{\alpha n}}, 1}.
\end{equation}

\paragraph{Proof of the bound~\eqref{eq:cs-bias}.}
We start with the 
expression~\eqref{eq:icdf-dro} specialized for the $\Lcs$, 
	\begin{equation*}
	\Lcs(x;P_0) = \sup_{r\in \rset} \int_{0}^{1} r(\beta) 
	F^{-1}(1-\beta) \d \beta,
	\end{equation*}
	where
	\begin{equation*}
	\rset = \crl*{r:[0,1]\to \R_+ ~\Big\vert~ 
		\norm{r}_1 =1,~\norm{r}_2^2 \le 1+2\rho,~\mbox{and $r$ is 
			non-increasing}};
	\end{equation*}
	The restriction of $\rset$ to non-increasing 
	functions is ``free'' since $F^{-1}$ is  non-decreasing. Our strategy is to 
	relate $F^{-1}$ to CVaR and then apply the corresponding 
	bias bounds~\eqref{eq:cvar-bias}---this type of 
	transformation is closely related to the Kusuoka representation of 
	coherent risk measures~\cite{Kusuoka01}. Specifically, note that
	\begin{equation*}
	\Lcvar^\alpha = 
	\frac{1}{\alpha}\int_{0}^{\alpha} F^{-1}_Z(1-\beta)\d \beta
	\implies  F^{-1}_Z(1-\alpha) = \frac{\d}{\d\alpha} (\alpha 
	\Lcvar^\alpha).
	\end{equation*}
	Therefore, for any $r\in\rset$ integration by parts gives
	\begin{equation*}
	\int_{0}^{1} r(\beta) F_Z^{-1}(1-\beta) \d \beta
	= \int_{0}^{1} r(\alpha) \frac{\d}{\d\alpha} (\alpha 
	\Lcvar^\alpha)\d \alpha
	= r(1) \Lcvar^1 - \int_0^1 r'(\alpha) \alpha \Lcvar^\alpha d\alpha.
	\end{equation*}
	The CVaR bias bound~\eqref{eq:cvar-bias-precise} tells us that 
	$\Lcvar^\alpha \le 
	\bLcvar^\alpha + \biasbound(\alpha)$ where $\biasbound(\alpha) = 3B 
	\min 
	\crl*{\sqrt{\frac{1}{\alpha n}}, 1}$. Moreover, we may write  
	$\bLcvar^{\alpha} = \frac{1}{\alpha}\int_{0}^{\alpha} 
	\E \Fhat^{-1}(1-\beta)\d \beta$, where $\Fhat$ denotes the empirical 
	cdf of the losses $\ell(x;S_1),\ldots,\ell(x;S_n)$. Noting that 
	$r'(\alpha)\le 0$ 
	for all $\alpha$, we may write
	\begin{flalign*}
	- \int_0^1 r'(\alpha) \alpha \Lcvar^\alpha d\alpha
	&\le 
	- \int_0^1 r'(\alpha) \alpha \bLcvar^\alpha d\alpha
	- \int_0^1 r'(\alpha) \alpha \cdot \biasbound(\alpha) d\alpha
	\\&= \E \int_0^1 r(\beta) \Fhat^{-1}(1-\beta)\d \beta
	-r(1)\bLcvar^{1} + \int_0^1 [r(\alpha)-r(1)] (\alpha \cdot 
	\biasbound(\alpha))' 
	\d \alpha,
	\end{flalign*}	
	where in the final equality we used again integration by parts along with 
	$\E \Fhat^{-1}(1-\alpha) = \frac{\d}{\d \alpha}(\alpha \bLcvar)$. 

	Substituting back and using $\Lcvar^1 = \bLcvar^1 = \E \ell(x;S)$, we 
	obtain
	\begin{equation*}
	\int_{0}^{1} r(\beta) F_Z^{-1}(1-\beta) \d \beta
	- \E \int_0^1 r(\beta) \Fhat^{-1}(1-\beta)\d \beta
	\le \sup_{r\in\rset} \int_0^1 [r(\alpha)-r(1)] (\alpha \cdot 
	\biasbound(\alpha))' \d \alpha \eqdef E
	\end{equation*}
	Taking a supremum over $r\in\rset$, we 
	conclude that
	\begin{flalign}
	\Lcs(x;P_0) &= \sup_{r\in \rset} \int_{0}^{1} r(\beta) 
	F_Z^{-1}(1-\beta) \d \beta
	\le 
	\sup_{r\in \rset} \E \int_{0}^{1} r(\beta) 
	\Fhat^{-1}(1-\beta)\d \beta + E
	\nonumber \\&
	\le \E \sup_{r\in \rset} \int_{0}^{1} r(\beta) 
	\Fhat^{-1}(1-\beta)\d \beta + E
	= \bLcs(x;n) + E.
	\label{eq:gen-kusuoka}
	\end{flalign}
	
	It remains to bound the quantity $E$, which we do via the 
	the Cauchy-Schwarz 
	inequality and the definition of  $\rset$, which gives
	\begin{equation*}
	  E = \int_0^1 [r(\alpha)-r(1)] (\alpha \cdot \biasbound(\alpha))'
          \d \alpha 
	\le 
	\norm{r}_2 \norm{(\alpha \cdot \biasbound(\alpha)'}_2 \le 
	\sqrt{1+2\rho} 
	\cdot 
	\norm{(\alpha \cdot \biasbound(\alpha))'}_2
	\end{equation*}
	for all $r\in\rset$.  We calculate $(\alpha \cdot \biasbound(\alpha))' = 
	\biasbound(0) 
	\indic{\alpha 
	\le 1/n} + \half \biasbound(\alpha) \indic{\alpha > 1/n}$, so that
	\begin{equation*}
	\norm{(\alpha \cdot \biasbound(\alpha)'}_2^2 = \frac{\biasbound^2(0)}{n}\prn*{
	1 + \int_{1/n}^1 \frac{\d \beta}{4\beta}	
} \le (3B)^2 \cdot \frac{4+\log n}{4n}.
	\end{equation*}
	for all $r\in\rset$, giving the required bound.
	
	\begin{remark}
	The bound~\eqref{eq:gen-kusuoka} hold for any 
	loss~\eqref{eq:pop-dro-reg-restated} and not just $\Lcs$. Moreover, the 
	final bound using Cauchy-Schwarz is equally valid for any 
	$\rho$-$\cs$-bounded uncertainty set. In particular, consider the 
	Cressie-Read uncertainty sets~\cite{CressieRe84} corresponding to 
	$k$-norm the constraint $\norm{r}_k^2 \le 1+2\rho$. For 
	$k> 2$ they satisfy $\norm{r}_2^2 \le 1+2\rho$ and our bias bounds 
	holds (using H\"{o}lder's inequality instead of Cauchy-Schwarz 
	removes the logarithmic factor). For $k\in(1,2)$ H\"{o}lder's inequality 
	gives bounds decaying as $n^{-(k-1)/k}$. 
\end{remark}

  \paragraph{Proof of the bound~\eqref{eq:icdf-bias}}
  Starting with CVaR, we return to the
  expression~\eqref{eq:cvar-bias-soft-indicator} for the bias and note that
  	\begin{equation*}
  	\brk*{\indic{u\ge 1-\alpha} - \softind(u)} F^{-1}(u) 
  	\le \brk*{\indic{u\ge 1-\alpha} - \softind(u)} \crl*{F^{-1}(1-\alpha) + 
\icdfLip \cdot 
  	(u-(1-\alpha))}
  	\end{equation*}
  	holds for all $u$, because when $u<1-\alpha$ we have that 
$\indic{u\ge 
  	1-\alpha} - \softind(u)\le 0$ and so we increase the LHS by replacing 
  	$F^{-1}$ with an under-estimate, while for $u\ge 1-\alpha$ we have 
  	$1 - \softind(u) \ge 0$ due to~\eqref{eq:sum-of-betas} and we increase 
  	the LHS be replacing it with an $F^{-1}$ with an over-estimate. 
  	Substituting into~\eqref{eq:cvar-bias-soft-indicator} and calculating gives
  	\begin{flalign}
  	&\Lcvar(x;P_0)-\bLcvar(x;n) 
  	\nonumber\\&\quad
  	\le 
  	\frac{1}{\alpha}\int_{0}^{1}\left(\indic{u \ge 1-\alpha}
  	-\softind\left(u \right)\right)
  	\left[F^{-1}\left(1-\alpha\right)+
  	\icdfLip \cdot \left(u -\left[1-\alpha\right]\right)\right]\d u 
  	\nonumber\\&\quad
  	\overset{\left(i\right)}{=}
  	 \frac{\icdfLip}{\alpha}\int_{0}^{1}\left(\indic{u \ge1-\alpha}
  	-\softind\left(u \right)\right) u \d u 
  	\nonumber\\&\quad
  	\overset{\left(ii\right)}{=}
  	\icdfLip\left[\frac{1}{2\alpha}\left(1-\left(1-\alpha\right)^{2}\right)
  	-\frac{1}{\alpha n }\sum _{i=n -\left\lfloor \alpha n \right\rfloor 
  	+1}^{n }\frac{i}{n +1}-\left(1-\frac{\left\lfloor \alpha n \right\rfloor 
  	}{\alpha n }\right)\frac{n -\left\lfloor \alpha n \right\rfloor }{n +1}\right]
  	\nonumber\\&\quad 
  	=\icdfLip\left[1-\frac{\alpha}{2}-\frac{1}{\alpha 
  	n \left(n +1\right)}\frac{\left\lfloor \alpha n \right\rfloor 
  	}{2}\left(2n -\left\lfloor \alpha n \right\rfloor 
  	+1\right)-\left(1-\frac{\left\lfloor \alpha n \right\rfloor }{\alpha 
  	n }\right)\frac{n -\left\lfloor \alpha n \right\rfloor }{n +1}\right]
  	\nonumber\\&\quad
  	=\icdfLip\left[1-\frac{\alpha}{2}-\frac{1}{2}\frac{\left\lfloor \alpha 
  	n \right\rfloor }{\alpha n \left(n +1\right)}\left(\left\lfloor \alpha 
  	n \right\rfloor +1\right)-\frac{n -\left\lfloor \alpha n \right\rfloor 
  	}{n +1}\right]
  	\nonumber\\&\quad
  	=\icdfLip\left[\frac{1}{n +1}+\frac{\left\lfloor \alpha n \right\rfloor 
  	}{n +1}\left[1-\frac{\left\lfloor \alpha n \right\rfloor }{2\alpha 
  	n }\right]-\frac{\alpha}{2}-\frac{1}{2}\frac{\left\lfloor \alpha 
  	n \right\rfloor }{\alpha n \left(n +1\right)}\right]
  	\nonumber\\&\quad
  	\le\icdfLip\left[\frac{1}{n +1}+\frac{\alpha 
  	n }{2\left(n +1\right)}-\frac{\alpha}{2}\right]\le\frac{\icdfLip}{n +1}.
 	 \label{eq:cvar-bias-icdf}
  	\end{flalign}
  	Above, $(i)$ uses the fact that $\frac{1}{\alpha}\softind$ is a convex 
  	combination of densities to deduce that 
  	\[
  	\int_{0}^{1}\left(\indic{u \ge 
  	1-\alpha}
  	-\softind\left(u \right)\right)
  	\left[F^{-1}\left(1-\alpha\right)-
  	\icdfLip \cdot \left(1-\alpha\right)\right]\d u = 0,\]
  	 and $(ii)$ uses the 
  	definition~\eqref{eq:cvar-bias-soft-indicator} of $\softind$ along with 
  	the fact that $\E \betarv(a,b) = \frac{a}{a+b}$.
  	
  	This bound extends to any $\L$ of the 
  	form~\eqref{eq:pop-dro-reg-restated} via~\eqref{eq:gen-kusuoka}, 
  	since we have $\biasbound(\alpha) = \icdfLip/(n+1)$ independent of 
  	$\alpha$ and consequently $(\alpha \cdot \biasbound(\alpha) )' = 
  	\icdfLip/(n+1)$, giving
  	\begin{equation*}
  	E =  \sup_{r\in\rset} \int_0^1 [r(\alpha)-r(1)] (\alpha \cdot 
  	\biasbound(\alpha))' \d \alpha = 
  	\frac{\icdfLip}{n+1} \cdot \sup_{r\in\rset} \int_0^1 [r(\alpha)-r(1)] 
  	\d\alpha \le 
  	\frac{\icdfLip}{n+1},
  	\end{equation*}
  	since $\int r(\alpha) \d \alpha = 1$ for all $r\in\rset$ regardless of 
  	$\phi$ and $\psi$.

  \paragraph{Penalized-$\cs$} 	We use the shorthand $Z=\ell(x,S)$ and 
	for a sample $S_1^n$ we let $Z_i = \ell(x,S_i)$. By 
	Eq.~\eqref{eq:lam-dual},
	\begin{equation*}
	\Llam(x;P_0) = \Upsilon(\eta\opt; P_0) = \E 
	\frac{(Z-\eta^\star)_+^2}{2\lambda} + 
	\eta^*  + \frac{\lambda}{2},
	\end{equation*}
	where $\eta^*$ is the unique solution to $\E(Z-\eta^\star)_+ = 
	\lambda$. (We omit the dependence of $\Upsilon$ on $x$ as $x$ is 
	constant throughout). 
	Similarly, we have that
	\begin{equation*}
	\Llam(x; S_1^n) = \Upsilon(\eta_n;S_1^n),~~\mbox{where}~~
	\Upsilon(\eta_n;S_1^n)\defeq 
	\sum_{i=1}^n\frac{(Z_i-\eta)_+^2}{2\lambda n} + \eta  
	+ 
	\frac{\lambda}{2}
	\end{equation*}
	and $\eta_n$ is the unique solution to 
	$\frac{1}{n}\sum_{i=1}^n (Z_i-\eta)_+ = \lambda$.  Convexity of 
	$\Upsilon$  w.r.t.\ $\eta$ gives us
	\begin{equation*}
	\Upsilon(\eta_n;S_1^n) \ge \Upsilon(\eta^\star;S_1^n) + 
	\Upsilon'(\eta^\star; S_1^n)(\eta_n-\eta^\star).
	\end{equation*}
	Taking expectation, we observe that $\E\Upsilon(\eta_n;S_1^n) = 
	\bLlam(x;n)$ and 
	$\E \Upsilon(\eta^\star;S_1^n) = 
	\Upsilon(\eta^\star;P_0)=\Llam(x;P_0)$. 
	Therefore, 
	by the Cauchy-Schwarz inequality
	\begin{flalign}
	&\bLlam(x;n) -  \Llam(x;P_0) =\E 
	\Upsilon'(\eta^\star;S_1^n)(\eta_n-\eta^\star) 
	\nonumber \\ & \quad\quad\quad
	\stackrel{(\star)}{=} \E \Upsilon'(\eta^\star;S_1^n)(\eta_n-\E 
	\eta_n) \ge - \sqrt{\var \Upsilon'(\eta^\star;S_1^n)}\sqrt{\var \eta_n},
	\label{eq:cauchy-schwarz-ftw}
	\end{flalign}
	where $(\star)$ uses that
        $\E\Upsilon'(\eta^\star;S_1^n) = \E\Upsilon'(\eta\opt;P_0)=0$ by 
	the definition of $\eta^\star$, and therefore we may replace 
	$\eta_n-\eta\opt$ 
	with $\eta_n - \E \eta_n$. %
	We now proceed to bound each variance separately. First, we have
	\begin{flalign}
	\var \Upsilon'(\eta^\star;S_1^n) &= \E 
	\brk*{\frac{1}{n}{\sum_{i=1}^n\prn*{\frac{(Z_i-\eta^*)_+}{\lambda} - 
	1}}}^2 = \frac{1}{n} \E \prn*{\frac{(Z-\eta^*)_+}{\lambda} - 1}^2 
	\nonumber\\
    & = \frac{1}{n}\brk*{ \E \frac{(Z-\eta^\star)_+^2}{\lambda^2}  -1 }
    = \frac{1}{n}\brk*{ \frac{1}{\lambda}(2 \Llam(x;P_0) - 2\eta^* - 
    \lambda) - 1} \le \frac{2B}{\lambda n},
	\label{eq:var-psi-bound}
	\end{flalign}
	where in the final transition we used $\Llam(x;P_0) \le B$ and 
	$\eta^\star \ge -\lambda$ due to  $\E(Z-\eta^\star)_+ = 
	\lambda$ and $Z\ge 0$. 
	
	To handle the second variance we use the Efron-Stein inequality 
	(\Cref{lem:efron-stein}). Let 
	$I$ be uniformly distributed on $[n]$, and define 
	\begin{equation*}
		\tilde{Z}_1^n=(Z_1,\ldots,Z_{I-1}, Z'_I, Z_{I+1}, \ldots, Z_n),
	\end{equation*}
	where $Z'$ is an i.i.d. copy of $Z$. Let
	$\tilde{\eta}_n$ be the solution to $\frac{1}{n}\sum_{i=1}^n 
	(\tilde{Z}_i-\eta)_+ = \lambda$. Then,
	\begin{equation}\label{eq:efron-stein}
	\var \eta_n \le \frac{n}{2} \E (\eta_n -\tilde{\eta}_n)^2
	\end{equation}
	Define the random 
	set 
	\begin{equation*}
	\mc{A} \defeq \{ i\mid Z_i - \eta_n > 0\}.
	\end{equation*}
	Recalling that
	$\sum_{i=1}^n 
	(\tilde{Z}_i-\tilde{\eta}_n)_+ = \sum_{i=1}^n 
	({Z}_i-\eta_n)_+ = \lambda n$, we have
	\begin{flalign*}
	0 &= \sum_{i\in [n]} \crl*{
		({Z}_i-\eta_n)_+ - (\tilde{Z}_i-\tilde{\eta}_n)_+}
	\le \sum_{i\in \mc{A}} \crl*{
		({Z}_i-\eta_n) - (\tilde{Z}_i-\tilde{\eta}_n)_+}\\
	&\le \sum_{i\in \mc{A}} \crl*{
		({Z}_i-\eta_n) - (\tilde{Z}_i-\tilde{\eta}_n)}
	= |\mc{A}|(\tilde{\eta}_n - \eta_n) + (Z_I - Z_I')\indic{I\in \mc{A}},
	\end{flalign*}
	and therefore $\eta_n - \tilde{\eta}_n \le \frac{B\indic{I\in 
	\mc{A}}}{|\mc{A}|}$. Similarly defining $\tilde{\mc{A}} \defeq \{ i\mid 
	\tilde{Z}_i - \tilde{\eta}_n > 0\}$ and applying the same argument with 
	$\tilde{\eta}_n$ and $\eta_n$ swapped allows us to conclude that
	\begin{equation*}
	(\eta_n - \tilde{\eta}_n)^2 \le B^2 \max\crl*{
	\frac{\indic{I\in\mc{A}}}{|\mc{A}|^2}, 
	\frac{\indic{I\in\tilde{\mc{A}}}}{|\tilde{\mc{A}}|^2}}
	\le B^2 \prn*{
		\frac{\indic{I\in\mc{A}}}{|\mc{A}|^2} + 
		\frac{\indic{I\in\tilde{\mc{A}}}}{|\tilde{\mc{A}}|^2}
	}.
	\end{equation*} 
	Taking expectation, we obtain
	\begin{equation*}
	\E(\eta_n - \tilde{\eta}_n)^2 \le 2B^2 \E \brk*{
	\frac{\indic{I\in\mc{A}}}{|\mc{A}|^2} }
	= %
	\frac{2B^2}{n} \E \brk*{\frac{1}{|\mc{A}|} },
	\end{equation*}
	where the final transition follows from $\E \brk{\,\indic{I\in\mc{A}} 
	\mid 
	|\mc{A}|} = {|\mc{A}|}/{n}$ (since $I$ is uniform on $[n]$).
	Assume for the moment that $\lambda \le B$. Then we must have 
	$\eta_n \ge 0$ and moreover 
	$|\mc{A}| 
	\ge n\min\{1, \lambda / B\}$ with probability 1. Substituting back 
	into~\eqref{eq:efron-stein}, we get the variance bound
	\begin{equation}\label{eq:var-eta-bound}
	\var \eta_n \le \frac{B^3}{\lambda n}.
	\end{equation}
	Combining~\eqref{eq:var-eta-bound},~\eqref{eq:var-psi-bound} 
	and~\eqref{eq:cauchy-schwarz-ftw} gives the result for $\lambda \le B$. 
	
	In the edge case that $\lambda \ge B$, Eq.~\eqref{eq:lam-edge-case} 
	gives us that
	\begin{flalign*}
	\bLlam(x;n) &= \E \frac{1}{n}\sum_{i\le n} Z_i + \frac{1}{2\lambda}\E 
	\var[Z_1^n]
	= \E Z + \frac{n-1}{2n\lambda} \Var[Z] \\ &= \L(x;P_0) - 
	\frac{1}{2\lambda}\Var[Z] \ge \L(x;P_0) - \frac{B^2}{2\lambda n}.
	\end{flalign*}
	We note that in this case may easily form an unbiased estimator of 
	$\Llam$ by using the standard unbiased variance estimator.
      \end{proof}

\subsubsection{Worst-case tightness of bias bounds}\label{sec:worst-case-bias}

\begin{proposition}\label{prop:lb-bias}
  For $p\in [0, 1]$, let $P_0 = \mathsf{Bernoulli}(p_0)$ and
  $\ell(\x;\s) = B \cdot s$. The following results hold.
  \begin{itemize}
  \item Set $p_0=\alpha$, then
    \begin{equation*}
      \Lcvar(\x;P_0) - \bLcvar(\x;n) \gtrsim 
      \frac{B\sqrt{1-\alpha}}{\sqrt{\alpha n}}.
    \end{equation*}
  \item Set $p_0=(1+2\rho)^{-1}$, then
    \begin{equation*}
      \Lcs(\x;P_0) - \bLcs(\x;n) \gtrsim B\sqrt{\frac{\rho}{n}}.
    \end{equation*}
  \item Set $p_0=\lambda/B \le 1/2$, then
    \begin{equation*}
      \Llam(\x;P_0) - \bLlam(\x;n) \gtrsim \frac{B^2}{\lambda n}.
    \end{equation*}
  \end{itemize}
\end{proposition}

\begin{proof} As before, we treat each case separately.
	  
  \paragraph{CVaR.} 
  First, note that $\Lcvar(x; P_0) = B$ since for $Q$ such 
  that
  $Q(1)=1$ we have $\frac{\d Q}{\d P_0}(s) = \frac{1}{\alpha}\indic{s=1}$ and
  therefore $Q\in\usetCVaR(P_0)$.  Second, for a sample $S_1^n\in\{0,1\}^n$ we
  have
  \begin{equation*}
    \Lcvar(x;S_1^n) = B \max\crl[\Bigg]{1, \frac{1}{\alpha 
	n}\sum_{i\in[n]} S_i}.
  \end{equation*}
  Therefore
  \begin{flalign*}
    \Lcvar(x; P_0) - \bLcvar(x;n) &= \frac{B}{\alpha n} \, \E \prn[\Bigg]{
      n\alpha - \sum_{i\in[n]} S_i }_+ \\& \ge
    \frac{B\sqrt{1-\alpha}}{\sqrt{\alpha n}} \P\prn[\big]{ \binrv(n,\alpha) \le
      n\alpha - \sqrt{n\alpha(1-\alpha)} }\gtrsim
    \frac{B\sqrt{1-\alpha}}{\sqrt{\alpha n}},
  \end{flalign*}
  where the final bound follows from the Berry-Esseen theorem (see
  Lemma~\ref{lem:bin-tail-be}).
  \paragraph{Constrained-$\cs$.} The $\cs$ divergence between two 
  Bernoulli
  random variables is
	\begin{equation*}
	\divcs(\bernrv(q), \bernrv(p)) = \half p \prn*{\frac{q}{p}-1}^2 + 
	\half (1-p) \prn*{\frac{1-q}{1-p}-1}^2 = \frac{(q-p)^2}{2p(1-p)}.
	\end{equation*}
	Therefore, for any $p\in(0,1)$, the element in $\usetCS(\bernrv(p))$ 
	that maximizes $Q(1)$ is $Q=\bernrv(q)$  with 
	$q=\min\crl*{1, p + \sqrt{2\rho}\sqrt{ p(1-p)}}$. 
	Set  $p_0 = \frac{1}{1+2\rho}$ and note that the function 
	$f(p) = p + 
	\sqrt{2\rho}\sqrt{ p(1-p)} = p + \sqrt{(1-p_0)/p_0}\sqrt{p(1-p)}$ 
	satisfies $f(p_0)=1$ and $f'(p)\ge \frac{1}{2p_0}$ for all 
	$p\le p_0$. Therefore, we 
	have
	\begin{equation*}
	\Lcs(x; \bernrv(p)) \le B\brk*{1-\prn*{\frac{p_0 - p}{2p_0}}_{+}}
	\end{equation*}
	for all $p\in(0,1)$, with equality at $p=p_0$.  
	In particular, setting $P_0=\bernrv(p_0)$ implies $\Lcs(x; P_0) = 
	B$  and for a sample $S_1^n\sim P_0^n$ with  
	$\hat{p}=\frac{1}{n}\sum_{i\in[n]} S_i$ we have $\Lcs(x;S_1^n) \le 
	B\prn*{1-\frac{p_0-\hat{p}}{2p_0}}$. Therefore
	\begin{equation*}
	\Lcs(x;P_0)-\bLcs(x;n) \ge \frac{B}{2p_0}\E(p_0-\hat{p})_+
	= \frac{B}{2p_0 n}\cdot\E \prn[\Bigg]{ np_0  - \sum_{i\in[n]} S_i 
	}_+ \stackrel{(\star)}{\gtrsim}
	\frac{B\sqrt{1-p_0}}{\sqrt{p_0 n}} = B\sqrt{\frac{2\rho}{n}},
	\end{equation*}
	where $(\star)$ follows from the CVaR case and for the final equality we
        substitute the definition of $p_0$.
       
      \paragraph{Penalized-$\cs$.}
      For any $p\in(0,1)$ we have
      \begin{flalign*}
      \L(x;\bernrv(p)) &= \sup_{q\in[0,1]} \crl*{qB - \lambda 
     \divcs(\bernrv(q), 
      \bernrv(p))}
      \\& = \sup_{q\in(0,1)} \crl*{qB - \frac{\lambda(q-p)^2}{2p(1-p)}} 
      =\begin{cases}
     pB\prn*{1+\frac{(1-p)B}{2\lambda}}  & p \le \lambda /B \\
     B - \frac{\lambda(1-p)}{2p}  & \mbox{otherwise.} \\
           \end{cases}
      \end{flalign*}  
      Simplifying, we have,
      \begin{equation*}
      \L(x;\bernrv(p)) \le  \frac{B+\lambda}{2} +  \frac{B^2}{2\lambda}\cdot 
      \brk*{\prn*{p - \frac{\lambda}{B}} 
      - \frac{B}{\lambda}\frac{\prn*{p - 
      \frac{\lambda}{B}}_{+}^2}{1+\frac{B}{\lambda}{\prn*{p - 
      \frac{\lambda}{B}}_{+}}}
  },
      \end{equation*}
      with equality at $p=\lambda/B$. 
      Taking $p_0 = \lambda /B$ and  and for a sample $S_1^n\sim P_0^n$ 
      letting 
      $\hat{p}=\frac{1}{n}\sum_{i\in[n]} S_i$, we have
      \begin{equation*}
      \Llam(x;P_0) - \bLlam(x;n) \ge - \frac{B}{2p_0}\E (\hat{p}-p_0)
      + \frac{B}{2p_0^2}\E \frac{\prn*{\hat{p} 
      -p_0}_{+}^2}{1+\frac{1}{p_0}{\prn*{\hat{p} - 
      			p_0}_{+}}}.
      \end{equation*}
      Since $\E \hat{p} = p_0$, we may lower bound this as
      \begin{equation*}
      \Llam(x;P_0) - \bLlam(x;n) \ge \frac{B(1-p_0)}{4np_0}\P( \sqrt{n^{-1} 
      p_0(1-p_0)} \le \hat{p}-p_0 \le p_0).
      \end{equation*}
      We have 
      \begin{flalign*}
      &\P( \sqrt{n^{-1} 
      	p_0(1-p_0)} \le \hat{p}-p_0 \le p_0) \\&\quad\quad
      \ge \P(\binrv(n,p_0) \ge np_0 + 
      	\sqrt{n
      	p_0(1-p_0)})- \P(\binrv(n,p_0) \ge 2np_0) \gtrsim 1
      \end{flalign*}
      by Berry-Esseen and Chernoff, and the result follows by substituting 
      $p_0=\lambda/B$.
\end{proof}

\subsection{Discussion of additional 
assumptions}\label{app:batch-extra-assumptions}

\subsubsection{Smoothness of $\ell$}\label{app:batch-smooth-ell}
The guarantees for the accelerated gradient iterations~\eqref{eq:agd}, 
detailed in Appendix~\ref{app:gangster}, require the objective function be 
smooth, i.e., have Lipschitz gradient. However, the degree of smoothness 
need not be high: as~\citet{Nesterov05b} and subsequent 
work~\cite{Lan12,DuchiBaWa12} observed, even if $\grad \L$ is order 
$G^2/\eps$ Lipschitz, acceleration allows finding an $\epsilon$ accurate 
solution in roughly $GR/\eps$ steps (a quadratic improvement over the 
SGM 
rate), as long as the gradient variance is itself of order $\epsilon$; the 
accelerated rates in Theorem~\ref{thm:batch-complexity} stem from this 
fact. 

By Claim~\ref{claim:smooth-obj}, for $\L$ to have roughly $G^2/\eps$ 
Lipschitz gradient, the loss gradients $\grad\ell$ have to be 
$H=G^2/\eps$ Lipschitz. This is in fact a weak assumption, because every 
$G$-Lipschitz loss $\ell$ has a \emph{smoothed} version $\tilde{\ell}$ 
that satisfies $|\tilde{\ell}(x;s)-\ell(x;s)|\lesssim \eps$ for all $x,s$ and 
that 
$\grad\tilde{\ell}(x;s)$ is $G^2/\eps$ Lipschitz. For example, we may 
replace 
the hinge loss $\ell(x;s)=(1-x^\top s)_+$ with $\tilde{\ell}(x;s)=\epsilon 
\log(1+ \exp([1-x^\top s]/\eps))$. More generally, the 
smoothing~\cite{GuzmanNe15}
\begin{equation}\label{eq:gen-smoothing}
\tilde{\ell}(x;s) = \inf_{y\in\X} \crl*{ \ell(y;s) + 
\frac{G^2}{2\epsilon}\norm{y-x}^2}
\end{equation} 
works for any $G$-Lipschitz $\ell$. 

In practice, we are often at liberty to replace the original loss $\ell$ with its 
smoothed version $\tilde{\ell}$ and minimize the resulting objective 
$\tilde{\L}$ which is guaranteed to be sufficiently smooth and 
approximates $\L$ to accuracy $\epsilon$. Indeed, in the ``statistical 
learning''  
model where we observe the entire $\ell(\cdot;S)$ per sample of $S\sim 
P_0$, we can 
apply the smoothing~\eqref{eq:gen-smoothing} to enforce the smoothness 
requirement without loss of generality. 
Therefore, our smoothness assumption can fail to hold only in situation 
where $\ell$ is non-smooth and $\ell$ and $\grad\ell$ are strict 
black-boxes, so we cannot compute~\eqref{eq:gen-smoothing} without 
multiple black-box queries.

\subsubsection{Lipschitz inverse-cdf}\label{app:batch-icdf}

The inverse-cdf of $\ell(x;S)$ is Lipschitz if and only if the distribution of 
$\ell(x;S)$ has positive density in the interval $[\min_{s\in\ss}\ell(x;s), 
\max_{s\in\ss}\ell(x;s)]$. This is a rather strong assumption that fails 
whenever $\ss$ is discrete or $\ell(x;S)$ is distributed as two separate 
bulks. However, the conclusions of our analysis under the Lipschitz 
inverse-cdf assumption hold under two natural relaxations.

\paragraph{Near-Lipschitz inverse-cdf and discrete loss distributions.}
Note that if $F^{-1}$ satisfies $|F^{-1}(u)-\tilde{F}^{-1}(u)| \le \delta$ for 
all $u\in[0,1]$ and a $\icdfLip$-Lipschitz $\tilde{F}^{-1}(u)$, then we can 
repeat the proof of the bound~\eqref{eq:icdf-bias} to show that 
$\L(x;P_0)-\bL(x;n) \le \delta + \frac{\icdfLip}{n+1}$ for all objectives of 
the form~\eqref{eq:pop-dro-reg-restated}. Moreover, suppose that $P_0$ 
is 
uniform on $N$ elements $s_1^N$ such that $\ell(x;s_i)$ is increasing in 
$i$, and 
suppose that it holds that
\begin{equation}\label{eq:discrete-icdf-assumption}
\ell(x;s_{i+1}) - \ell(x;s_{i}) \le \frac{\icdfLip}{N}.
\end{equation}
That is, the increments in the loss are not too far from uniform. Then, the 
piecewise linear function $\tilde{F}^{-1}$ connecting the steps in $F^{-1}$ 
is $\icdfLip$-Lipschitz and satisfies $|F^{-1}(u)-\tilde{F}^{-1}(u)| \le 
\icdfLip/N$. Therefore, the 
assumption~\eqref{eq:discrete-icdf-assumption} implies that for any 
mini-batch size $n< N$, we have $\L(x;P_0)-\bL(x;n) \le 2\icdfLip/(n+1)$. 
We note also that  the assumption 
$n< N$ is essentially without loss of generality, since for $n=N$ we can 
simply use a full-batch method with no bias at all. 

\paragraph{CVaR bias bounds with locally Lipschitz inverse-cdf.}
The proof of the bound~\eqref{eq:cvar-bias-icdf}
also works if $F^{-1}$ is Lipschitz in a small neighborhood of the CVaR 
cutoff $1-\alpha$, because for values of $u$ that are roughly 
$\sqrt{n^{-1}\alpha}$ far from $1-\alpha$ we may bound $|\indic{u\ge 
1-\alpha} - \softind(u)|$ via  tail bounds, as in the proof of the 
bound~\eqref{eq:cvar-bias}. Therefore, we expect the bias of 
$\Lcvar(x;P_0)-\bL(x;n)$ to vanish with rate $1/n$ whenever the 
distribution of $\ell(x;S)$ has a density at the $1-\alpha$ quantile loss 
value. Prior work shows that, from an asymptotic perspective, the converse 
is also true: when $\ell(x;S)$ does not have a density at the $1-\alpha$ 
quantile, the bias vanishes with asymptotic rate 
$n^{-1/2}$~\cite[cf.][Theorem 2]{TrindadeUrShZr07}.

\subsection{Proofs of variance bounds}\label{prf:prop-variance-grad}
We give a more general statement of the variance bound using the notion of 
$\CSbdd$-$\cs$-bounded objectives (Definition~\ref{def:cs-bounded}); 
Proposition~\ref{prop:variance-grad} follows immediately from 
Claim~\ref{claim:cs-bdd}. 

\begin{customprop}{\ref*{prop:variance-grad}'}\label{prop:variance-extended}
	Let $\L$ be an objective of the form~\eqref{eq:pop-dro-reg-restated}. If 
	$\L$ is 
	$\CSbdd$-$\cs$-bounded, we have that for all $n\in \N$ and  $x\in\X$
	\begin{equation*}
	\Var[ \L(x;S_1^n) ] \le \frac{2(1+C)}{n}B^2.
	\end{equation*}
	If in addition $\phi=0$ and $\psi$ is strictly convex, we have
	\begin{equation*}
	\Var[ \grad\L(x;S_1^n) ] \le \frac{8(1+C)}{n}G^2.
	\end{equation*}
\end{customprop}

\begin{proof}
We first show the bound on the objective variance. 	By the the Efron-Stein 
inequality (see Lemma~\ref{lem:efron-stein}), we have
\begin{equation}\label{eq:obj-efron-stein}
\Var[ \L(x;S_1^n)] \le \frac{n}{2} \E{} \prn{  \L(x;S_1^n) - 
	 \L(x;\tilde{S}_1^n)}^2,
\end{equation}
where $S$ and $\tilde{S}$ are identical 
except in a random entry $I$ for which $\tilde{S}_I$ is an iid copy of 
$S_I$.  Let $q$ and $\tilde{q}$ denote the maximizers 
of~\eqref{eq:emp-dro-reg-restated} for samples $S_1^n$ and 
$\tilde{S}_1^n$ respectively. In addition, let $Z_i = \ell(x;S_i)$ and 
$\tilde{Z}_i = \ell(x;\tilde{S}_i)$. Clearly, $\L(x;S_1^n)$ is convex in $Z$ 
and satisfies $\frac{\del}{\del Z}\L(x;S_1^n) = q$. Therefore, 
\begin{equation*}
\L(x;S_1^n) - \L(x;\tilde{S}_1^n) \le \inner{\frac{\del}{\del 
Z}\L(x;S_1^n) }{Z-\tilde{Z}} = q_I (\ell(x;S_I) - \ell(x;\tilde{S}_I)).
\end{equation*}
Applying the argument again with $S$ and $\tilde{S}$ swapped, we find that
\begin{equation*}
|\L(x;S_1^n) - \L(x;\tilde{S}_1^n)| \le \max\{q_I,\tilde{q}_I\} |\ell(x;S_I) - 
\ell(x;\tilde{S}_I)| \le B\sqrt{q_I^2 + \tilde{q}_I^2}.
\end{equation*}

Therefore, using the fact the $q_I$ and $\tilde{q}_I$ are identically 
distributed, we have
\begin{equation*}
\E \prn{  \L(x;S_1^n) -  \L(x;\tilde{S}_1^n)}^2 \le
2B^2 \E q_I^2 = \frac{2G^2}{n} \E \norm{q}_2^2 
= \frac{4B^2}{n^2}(\CSbdd+1),
\end{equation*}
where the final bound is due to $\norm{q}_2^2 = 
\frac{1}{n}(2\divcs(q,\frac{1}{n}\ones)+1)$ and the 
$\CSbdd$-$\cs$-bounded property of $\L$. Substituting back 
into~\eqref{eq:obj-efron-stein} gives the claimed objective variance bound.
	
 Next, to show the bound on the gradient variance we invoke Efron-Stein 
 elementwise to obtain
  \begin{equation*}\label{eq:grad-efron-stein}
  \Var[\grad \L(x;S_1^n)] \le \frac{n}{2} \E{} \norm{ \grad \L(x;S_1^n) - 
  	\grad \L(x;\tilde{S}_1^n)}^2.
  \end{equation*}
By the 
expression~\eqref{eq:gradient-expression-emp} for $\grad \L$ we have
\begin{flalign*}
\norm{ \grad \L(x;S_1^n) - \grad \L(x;\tilde{S}_1^n)}
&= \norm[\Bigg]{\sum_{i \ne I} (q_i - \tilde{q}_i) \grad \ell(x;S_i) + 
q_I\grad 
\ell(x;S_I) -\tilde{q}_I \grad\ell(x;\tilde{S}_I)}
\\ &
\le G\prn[\Bigg]{\sum_{i \ne I} |q_i - \tilde{q}_i|  + q_I + \tilde{q}_I},
\end{flalign*}
where the bound follows from the triangle inequality and the fact that 
$\ell$ is $G$-Lipschitz. 

Now, observe that $q_i = \frac{1}{n}{\psi^*}'[ (\ell(x;S_i) - \eta)/\lambda ]$ 
for some 
$\eta\in \R$ by Eq.~\eqref{eq:emp-q-opt-expression}. Similarly, $q_i = 
\frac{1}{n}{\psi^*}'[ (\ell(x;S_i) - \tilde{\eta})/\lambda ]$ for some 
$\tilde{\eta}$. Since 
$\psi$ is strictly convex we have that ${\psi^*}'$ is continuous and 
monotonic non-decreasing. Consequently, either $q_i \ge \tilde{q}_i$ for 
all $i\ne I$ (if $\eta \le \tilde{\eta}$), or $q_i \le \tilde{q}_i$ for all $i\ne I$ 
(if $\eta \ge \tilde{\eta}$). In either case, we have
\begin{equation*}
\sum_{i \ne I} |q_i - \tilde{q}_i|  = \abs[\Bigg]{\sum_{i \ne I}(q_i -  
\tilde{q}_i)}
= \abs*{q_I - \tilde{q}_I},
\end{equation*}
where the final equality used the fact that $\sum_{i\le n} q_i = \sum_{i\le 
n} \tilde{q}_i =1$. 
Substituting back, we find that
\begin{equation*}
\norm{ \grad \L(x;S_1^n) - \grad \L(x;\tilde{S}_1^n)} \le 2G\max\{q_I, 
\tilde{q}_I\} \le 2G \sqrt{q_I^2 + \tilde{q}_I^2}.
\end{equation*}
The remainder of the proof is identical to that of the objective variance 
bound, except with $2G$ replacing $B$.
\end{proof}

Proposition~\ref{prop:variance-extended} implies that the variance of the  
$\cs$ constraint 
objective $\Lcs$ is at most $2(1+\rho)B^2/n$. However, our gradient 
variance bound requires $\phi=0$ and therefore does not apply to $\grad 
\Lcs$. The following proposition shows that the requirement $\phi=0$ is 
necessary, since no upper bound of the from $O(1) 
(1+\rho)G^2/n$ 
holds for $\Var[\grad\Lcs]$.

\begin{proposition}[Variance of the mini-batch gradient estimator for  
$\Lcs$]\label{prop:lb-variance}
  For any $n>4$ and $\rho \ge 0$, there exists a distribution $P_0$ over
  $\ss = \crl{0,1,2}$ and a $G$-Lipschitz loss $\ell:[-1,1]\times\ss\to 
  [0,1]$ such that
  \begin{equation*}
    \Var\brk*{\nabla \Lcs(0;\S_1^n)} \gtrsim \frac{\rho^2}{(1+\rho)^2}G^2.
  \end{equation*}
\end{proposition}
\begin{proof}
  We construct $P_0$ as follows,
  \begin{equation*}
  \P(S=2)=p_2=1-2^{1/n}\approx \frac{\log 2}{n}~~\mbox{and}~~
  \P(S=1)=p_1=\frac{1}{1+2\rho}.%
  \end{equation*}
  (Note that we may assume without loss of generality that $\rho \gtrsim 
  1/n$, so that $P(S=0)=1-p_1-p_2>0$, since for $\rho=0$ we already 
  have a standard $1/n$ lower bound on the variance). We set the loss 
  values to 
  be
  \begin{equation*}
  \ell(0;0)=0~~,~~\ell(0;1)=\frac{1}{30n}
  ~~\mbox{and}~~\ell(0;2)=1,
  \end{equation*}
  and the loss gradients as
  \begin{equation*}
  \grad\ell(0;2)=\grad\ell(0;0)=-G~~\mbox{and}~~\ell(0;1)=G.
  \end{equation*}
  
  The source of high variance in this construction is that, for a sample 
  $S_1^n$, the maximizing $q\opt$ behaves very differently when $S_i=2$ 
  for some $i$ and when $S_i \ne 2$ for all $i$. In the former case, we show 
  that $q\opt$ puts significant mass on samples with $S_i\ne 1$, so 
  $\grad\L(0;S_1^n) < G(1-c)$ for some $c\gtrsim \rho/(1+\rho)$. In the 
  latter case, we show that with constant probability $q\opt$ places mass 
  only on samples with $S_i=1$, and so $\grad\L(0;S_1^n) = G$. Since 
  either scenario occurs with constant probability, the variance bound 
  follows.
  
  \newcommand{\Count}{\mathrm{C}}
   \newcommand{\Ev}{\mathfrak{E}}
  
  To provide a detailed proof, let $\Count_k(S_1^n) = \sum_{i\le 
  n}\indic{S_i=k}$ be the number of samples with value $k$, for 
  $k\in\{0,1,2\}$, and consider the events
  \begin{equation*}
  \Ev_a(S_1^n) = \{\Count_2(S_1^n) = 0~\mbox{and}~\Count_1(S_1^n) \ge 
  np_1\} 
  \end{equation*}
  and
  \begin{equation*}
  \Ev_b(S_1^n) = \{\Count_2(S_1^n) = 1~\mbox{and}~\Count_1(S_1^n) <  
  np_1\}.
  \end{equation*}
  Note that we chose $p_2$ such that $\P(\Count_2(S_1^n) = 
  0)=(1-p_2)^n=\half$ and  that $\P(\Count_1(S_1^n) \ge np_1) \gtrsim 1$ 
  since $n p_1$ is roughly the median of $\Count_1(S_1^n)$. 
  Similarly, $\P(\Count_2(S_1^n) = 
  1)=np_0(1-p_0)^{n-1}\approx \frac{\log2 }{2}$ and $\P(\Count_1(S_1^n) 
  < np_1) \gtrsim 1$. Therefore,
  \begin{equation}\label{eq:evs-are-likely}
  \P(\Ev_a(S_1^n))\gtrsim 1~~\mbox{and}~~ \P(\Ev_b(S_1^n))\gtrsim 1.
  \end{equation}
  We bound 
  $\grad 
  \Lcs$ conditional on each event in turn.
  
  Under event $\Ev_a(S_1^n)$, the empirical loss distribution is Bernoulli 
  with 
  parameter $\Count_1(S_1^n)/n \ge p_1 = 1/(1+2\rho)$ and consequently 
  $q\opt$ places mass only on samples with value 1 (see further discussion 
  in the proof of Proposition~\ref{prop:lb-bias}). Therefore, we have
  \begin{equation}\label{eq:grad-ev-a}
  \E\brk{\grad \Lcs(0;S_1^n)\mid \Ev_a(S_1^n)} = G.
  \end{equation}
  
  To bound the gradient under event $\Ev_b(S_1^n)$, assume that without 
  loss of 
  generality that $S_1=2$ is the unique sample with that value. We consider 
  separately the cases $q\opt_1 > 2/3$ and $q\opt_1 \le 2/3$. In the 
  former, we clearly have $\grad \Lcs(0;S_1^n) \le -q\opt_1 G + 
  (1-q\opt_1)G < -G/3$. In the latter case, we recall 
  Eq.~\eqref{eq:cs-opt-q} showing that $q\opt$ is of the form
  \begin{equation*}
  q\opt_i = \frac{(\ell(x;S_i) - \eta\opt)_+}{\sum_{j\le n}(\ell(x;S_j) - 
  \eta\opt)_+}
  \end{equation*}
  for some $\eta^\star\in \R$. The fact that $q\opt_1 \le 2/3$ and that 
  there are at most $n$ samples with value 
  $\ell(0;1)=1/(30n)$ gives the following 
  bound on $\eta\opt$
  \begin{equation*}
 \frac{2}{3}\ge q\opt_1 =  
 \frac{\ell(0;S_1)-\eta\opt}{\sum_{j\le n}(\ell(x;S_j) - 
 	\eta\opt)_+} \ge \frac{1-\eta\opt}{31/30-n\eta\opt}
 \implies \eta\opt \le -\frac{1}{3n}.
  \end{equation*}
  Suppose $S_j=0$ and $S_i=1$, then
  \begin{equation*}
  r = \frac{q\opt_j}{q\opt_i}=\frac{\ell(0;0)-\eta\opt}{\ell(0;1)-\eta\opt}
  = 1 - \frac{\ell(0;1)}{\ell(0;1)-\eta\opt} \ge \frac{7}{8}.
  \end{equation*}
  Assuming that $S_1=2$, we have that the total 
  weight under $q\opt$ of samples with gradient $-G$ is 
  \begin{equation*}
  q\opt_1 + (1-q\opt_1) \frac{r \Count_0(S_1^n)}{\Count_1(S_1^n) + r 
  	\Count_0(S_1^n)} \ge \frac{7}{8} (1-p_1) = \frac{7\rho}{4(1+2\rho)},
  \end{equation*}
  which implies $\grad \Lcs(0;S_1^n) \le -\frac{7\rho}{4(1+2\rho)}G + 
  (1-\frac{7\rho}{4(1+2\rho)})G \le G(1-\frac{7\rho}{2(1+2\rho)})$. We 
  conclude 
  that
  \begin{equation}\label{eq:grad-ev-b}
  \E\brk{\grad \Lcs(0;S_1^n)\mid \Ev_b(S_1^n)} \le  
  G(1-c)~~\mbox{for}~~c=\min\crl*{\frac{4}{3}, 
  \frac{7\rho}{2(1+2\rho)}}\gtrsim 
  \frac{\rho}{1+\rho}.
  \end{equation}
  
  Let $\tilde{S}_1^n$ be an independent copy of $S_1^n$. We combine our 
  conclusions~\eqref{eq:evs-are-likely},~\eqref{eq:grad-ev-a} 
  and~\eqref{eq:grad-ev-b} to form a variance bound as follows,
  \begin{flalign*}
  &\Var[\grad\Lcs(0;S_1^n)] = \frac{1}{2} \E \prn*{\grad\Lcs(0;S_1^n)-
  	\grad\Lcs(0;\tilde{S}_1^n)}^2 
   \\&\quad \quad
   \ge 
   \frac{1}{2} \E \brk*{\prn*{\grad\Lcs(0;S_1^n)-
  	\grad\Lcs(0;\tilde{S}_1^n)}^2 ~\Big\vert~ \Ev_a(S_1^n), 
  	\Ev_b(\tilde{S}_1^n)} \P(\Ev_a(S_1^n),\Ev_b(\tilde{S}_1^n) )
  \\ & \quad \quad \ge 
  \half c^2 G^2 \cdot \P(\Ev_a(S_1^n))\cdot \P(\Ev_b(\tilde{S}_1^n) )\gtrsim 
  \frac{\rho^2}{(1+\rho)^2}G^2.
  \end{flalign*}
\end{proof}

\subsection{Convergence rates of stochastic gradient 
  methods}\label{app:gangster}

We state below the classical convergence
rates for standard and accelerated stochastic gradient methods, under a 
somewhat 
non-standard assumption that the stochastic gradient estimates are 
unbiased for a uniform approximation of the objective function with 
additive error $\delta$.

\arxiv{
\restatePropGangster*
}
\notarxiv{
  \begin{proposition}[Convergence of stochastic gradient
    methods~{\cite[][Corollary 1]{Lan12}}]\label{prop:gangster}
Let $F:\X\to\R$ and $\overline{F}:\X\to\R$ satisfy $0\le 
F(x)-\overline{F}(x)\le
\delta$ for all $\X\in\R$. Assume that $\overline{F}$ is convex and that a
stochastic gradient estimator $\tilde{g}$ satisfies $\E{} \tilde{g}(x)
\in \del \overline{F}(x)$ and $\E{} \norm{\tilde{g}(x)}^2 \le \Gamma^2$ for
all $x\in\X$. For $T\in\N$, the iterate $\bar{x}_T$ in the
sequence~\eqref{eq:sgd} with 
 $\eta \asymp \frac{R}{ T^{1/2}\Gamma}$
satisfies
\begin{equation}\label{eq:sgd-rate}
\E F(\bar{x}_T) - \inf_{x'}F(x') \lesssim \delta + \frac{\Gamma
  R}{\sqrt{T}}.
\end{equation}
If in addition $\grad \overline{F}$ is $\Lambda$-Lipschitz and $\Var
\brk*{\tilde{g}(x)} \le \sigma^2$ for all $x\in\X$, the iterate $y_T$
in the sequence~\eqref{eq:agd} with $\eta \asymp
\min\{\frac{1}{\Lambda}, \frac{R}{T^{3/2}\sigma}\}$ and $\theta_t =
\frac{2}{t+1}$ satisfies
\begin{equation}\label{eq:agd-rate}
\E F(y_T) - \inf_{x'}F(x') \lesssim \delta + \frac{\Lambda R^2}{T^2} +
\frac{\sigma R}{\sqrt{T}}.
\end{equation}
\end{proposition}
}
\begin{proof}
	\cite{Lan12} gives us the rates~\eqref{eq:sgd-rate} 
	and~\eqref{eq:agd-rate} but for $\overline{F}$ rather than $F$. That is, it 
	guarantees that SGM finds $\bar{x}_T$ such that
	\begin{equation*}
	\E \overline{F}(\bar{x}_T) - \inf_{x'}\overline{F}(x') \lesssim 
	\frac{\Gamma R}{\sqrt{T}}.
	\end{equation*}
	To remove the bars, we use $0\le F(x)-\overline{F}(x)\le 
	\delta$ to write 
	\begin{equation*}
	-\inf_{x'}{F}(x') \ge  -\inf_{x'}\overline{F}(x') 
	~~\mbox{and}~~
	{F}(\bar{x}_T) \ge \overline{F}(\bar{x}_T) + \delta.
	\end{equation*}
\end{proof}

\begin{remark}
	In the unconstrained case $\X=\R^d$, the
	recursion~\eqref{eq:agd} reduces to the more familiar form
	\begin{equation}\label{eq:agd-pytorch}
	v_{t+1} = \omega_t v_t - \eta \tilde{g}(x_t),~~x_{t+1} = x_t + 
	\omega_{t+1} v_{t+1} - \eta \tilde{g}(x_t),
	\end{equation}
	where $\omega_t = (1-\theta_{t-1})\frac{\theta_t}{\theta_{t-1}}$ is a
	time-varying ``momentum'' parameter; the sequences $y_t,z_t$ are 
	related to
	$v_t$ via $v_t = \frac{\theta_t}{\omega_t}(z_t - y_t)$ and
	$y_{t+1} = x_t - \eta \tilde{g}(x_t)$. 
\end{remark}

\subsection{Proofs of complexity bounds}\label{prf:thm-complexity-batch}
\restateThmBatchComplexity*
\renewcommand{\err}{\mathsf{err}}
\begin{proof} To prove each bound in the theorem we choose $n$ large 
enough via one of the bounds in Proposition~\ref{prop:batch-bias} and 
then choose $T$ to guarantee $\epsilon$-accurate solution via  
Proposition~\ref{prop:gangster}. For a (potentially random) point 
$\bar{x}\in\X$ and robust risk $\L$,
  we define the shorthand
  \begin{equation*}
    \err(x; \L) \defeq \E\L(\bar{x};P_0) - \inf_{x\in\X}\L(x;P_0).
  \end{equation*}
  We summarize our choices of $n$ and $T$ for different robust objectives,
  under different assumptions in Table~\ref{batch-scaling}. In the statement of
  the theorem, we sometimes  upper bound $a\vee b \defeq \max\crl{a, b}$ 
  by $a+b$ for 
  readability, and state
  the tighter rates here.
  
  \begin{table}\label{batch-scaling}
  	\renewcommand{\arraystretch}{2.25}
  	\begin{center}
  		\begin{tabular}{ccccc}
  			\toprule
  			Loss & $\grad$ est. & $n$ & $T$ &  complexity $=nT$
  			\\
  			\midrule
  			$\Lcvar$& 
  			$\grad \Lcvar$&
  			$ \frac{B^2}{\alpha \epsilon^2}$ &
  			$\frac{(GR)^2}{\epsilon^2}$&
  			$\frac{(GR)^2B^2}{\alpha\epsilon^4}$\\
  			$\Lcs$& 
  			$\grad \Lcs$&
  			$\frac{(1+\rho)B^2}{\epsilon^2}\log \frac{(1+\rho)B^2}{\epsilon^2}$
  			&
  			$\frac{(GR)^2}{\epsilon^2}$&
  			$\frac{(1+\rho)(GR)^2 
  				B^2}{\epsilon^4}\log\frac{(1+\rho)B^2}{\epsilon^2}$\\
  			$\Llam$& 
  			$\grad \Llam$&
  			$\frac{B^2}{\lambda \epsilon}$ &
  			$\frac{(GR)^2}{\epsilon^2}$&
  			$\frac{(GR)^2 B^2 }{\lambda\epsilon^3}$\\
  			any $\L$ in~\eqref{eq:pop-dro-reg}& 
  			$\grad \L$&
  			$ \frac{\icdfLip}{\epsilon}$ &
  			$\frac{(GR)^2}{\epsilon^2}$&
  			$\frac{(GR)^2 \icdfLip }{\epsilon^3}$\\
  			\midrule
  			$\Lcvar$ & 
  			$\grad \LcvarR$&%
  			$\frac{B^2}{\alpha \epsilon^2}$ &
  			$\frac{GR\sqrt{\log\frac{1}{\alpha}+\nu}}{\epsilon}\vee 
  			\frac{(GR)^2}{B^2}$&
  			$\frac{(GR)^2}{\alpha\epsilon^2}\prn*{1 \vee 
  			\frac{B^2}{GR\epsilon}\sqrt{\log\frac{1}{\alpha} + 
  			\nu}}$%
  			\\
  			$\Lcvar$ & 
  			$\grad \LcvarR$&%
  			$\frac{\icdfLip}{\epsilon}$ &
  			$\frac{GR\sqrt{ 
  					\log\frac{1}{\alpha}+\nu}}{\epsilon}\vee\frac{(GR)^2}{\alpha 
  				\icdfLip 
  				\epsilon}$&
  			$ \frac{(GR)^2}{\epsilon^2} \prn*{\frac{1}{\alpha}\vee
  				\frac{\icdfLip}{GR}\sqrt{\log\frac{1}{\alpha} + \nu}}$\\
  			$\Llam$& 
  			$\grad \Llam$&
  			$ \frac{B^2}{\lambda \epsilon}$ &
  			$\frac{GR\sqrt{1+\nu}}{\epsilon}\vee \frac{(GR)^2}{B\epsilon}$&
  			$\frac{GRB}{\lambda\epsilon^2}\prn*{B\sqrt{1+\nu} \vee GR}$\\
  			\bottomrule
  		\end{tabular}
  		
  		\caption{Parameter settings for Theorem~\ref{thm:batch-complexity}. 
  			For $\LcvarR$ we take $\lambda \asymp 
  			\frac{\epsilon}{\log(\frac{1}{\alpha})}$. For $\Llam$ assume 
  			$\lambda 
  			\ge \epsilon$. $\nu \defeq H\epsilon/G^2$. We use the shorthand 
  			$a\vee b$ for $\max\{a,b\}$.
  		}
  	\end{center}
  \end{table}

  \paragraph{CVaR.} We distinguish between the different possible assumptions on
  the loss $\ell$ and distribution $P_0$ as they yield different rates.
  \begin{enumerate}[label=(\alph*),leftmargin=*]
  \item \underline{Non-smooth $\ell$:} let $\bar{x}_T$ be the iterates
    of~\eqref{eq:sgd}, the sub-optimality guarantee of~\eqref{eq:sgd-rate} and
    the bias bound of Proposition~\ref{prop:batch-bias} yield
  \begin{equation*}
    \err(\bar{x}_T; \Lcvar) \lesssim \frac{B}{\sqrt{\alpha n}} + \frac{GR}{\sqrt{T}}.
  \end{equation*}  
  In that case, setting $n \asymp \frac{B^2}{\alpha \epsilon^2}$ guarantees that
  the bias is smaller than $\epsilon$ and setting
  $T \asymp \frac{(GR)^2}{\epsilon^2}$ yields that
  $\err(\bar{x}_T;\Lcvar) \lesssim \epsilon$.

\item \underline{Smooth $\ell$:} if $\ell$ is $H$-smooth, we consider the
  $\LcvarR$ objective with $\lambda =
  \frac{\epsilon}{\log\frac{1}{\alpha}}$. This guarantees that, for all $x\in\X$
  \begin{equation*}
    \LcvarR(\x;P_0) \le \Lcvar(\x;P_0) \le \LcvarR(\x;P_0) + \epsilon.
  \end{equation*}
  $\bLcvarR$ being $(\tfrac{G^2\log(1/\alpha)}{\epsilon}+H)$-smooth, the final
  iterate of the sequence~\eqref{eq:agd} achieves
  \begin{equation*}
    \err(y_T;\Lcvar) \lesssim \epsilon + \frac{B}{\sqrt{\alpha n}}
    + \frac{(GR)^2(\log\frac{1}{\alpha}+\nu)}{\epsilon T^2}
    + \frac{GR}{\sqrt{\alpha nT}}.
  \end{equation*}
  To make sure that the second and third terms are smaller than $\epsilon$, we
  set
  $T = \frac{(GR)^2}{\alpha n \epsilon^2}\vee
    \frac{GR}{\epsilon}\sqrt{\log\tfrac{1}{\alpha} + \nu}$. To guarantee small bias,
  we set $n \asymp \frac{B^2}{\alpha\epsilon^2}$; the resulting complexity 
  is
  \begin{equation*}
    nT \asymp \frac{(GR)^2}{\alpha\epsilon^2}
    \max\crl*{1, \frac{B^2}{GR\epsilon}\sqrt{\log\frac{1}{\alpha} + \nu}}.
  \end{equation*}
\item \underline{Smooth $\ell$ and inverse cdf Lipschitz:} in this case, the
  regret guarantees of the iterates of~\eqref{eq:agd-rate} is
    \begin{equation*}
      \err(y_T;\Lcvar) \lesssim \epsilon + \frac{\icdfLip}{n} +
      \frac{(GR)^2(\log\frac{1}{\alpha}+\nu)}{\epsilon T^2}
      + \frac{GR}{\sqrt{\alpha nT}}.
    \end{equation*}
    We once again set $T = \frac{(GR)^2}{\alpha n \epsilon^2}\vee
    \frac{GR}{\epsilon}\sqrt{\log\tfrac{1}{\alpha} + \nu}$, and choosing
    $n\asymp \frac{\icdfLip}{\epsilon}$ yields the result.
  \end{enumerate}

  \paragraph{Penalized-$\cs$.} We distinguish between whether or not $\ell$
  is smooth.
  \begin{enumerate}[label=(\alph*),leftmargin=*]
  \item \underline{Non-smooth $\ell$:} for the sequence of iterates 
  of~\eqref{eq:sgd},
  we have
  \begin{equation*}
    \err(\bar{x}_T;\Llam) \lesssim \frac{B^2}{\lambda n} + \frac{GR}{\sqrt{T}},
  \end{equation*}
  and setting $n\asymp \frac{B^2}{\lambda \epsilon}$ and
  $T\asymp \frac{(GR)^2}{\epsilon^2}$ yields the fist rate.
  \item \underline{Smooth $\ell$:} We now turn to acceleration,
  we have
  \begin{equation*}
    \err(y_T; \Llam) \lesssim \frac{B^2}{\lambda n} +
    \frac{R^2\prn*{\frac{G^2}{\lambda}+ H}}{T^2} +
    GR\sqrt{\frac{1+\frac{B}{\lambda}}{nT}}.
  \end{equation*}
  First, noting that $\lambda \ge \epsilon$ guarantees that
  $R^2\prn{\frac{G^2}{\lambda}+ H} \le \frac{(GR)^2(1+\nu)}{\epsilon}$.
  Furthermore, we simplify the variance term since $B/\lambda \ge 1$.  We thus
  set
  $T \asymp \frac{GR}{\epsilon}\sqrt{1+\nu}\vee \frac{(GR)^2B}{\lambda n \epsilon^2}$
  and choose $n \asymp \frac{B^2}{\lambda \epsilon^2}$. This yields the final result
  \begin{equation*}
    nT \lesssim \frac{GRB}{\lambda \epsilon^2}\prn*{B\sqrt{1+\nu}
      \vee (GR)}.
    \end{equation*}
  \end{enumerate}
  \paragraph{Constrained-$\cs$.} This case is straightforward---without 
  any bound on
  the variance in the worst-case, we turn to the basic SGM 
  guarantee~\eqref{eq:sgd-rate}; we have
  \begin{equation*}
    \err(\bar{x}_T;\Lcs) \lesssim B\sqrt{1+2\rho}\sqrt{\frac{\log n}{n}}
    + \frac{GR}{\sqrt{T}}.
  \end{equation*}
  We set $T \asymp \frac{(GR)^2}{\epsilon^2}$ and
  $n\asymp\frac{(1+2\rho)B^2}{\epsilon^2}\log\prn{(1+2\rho)B^2\epsilon^{-2}}$.
  We then have
  \begin{equation*}
    B\sqrt{1+2\rho}\sqrt{\frac{\log n}{n}} =
    \epsilon\sqrt{
      1+\frac{\log\log\prn{\frac{(1+2\rho)B^2}{\epsilon^2}}}
      {\log\prn{\frac{(1+2\rho)B^2}{\epsilon^2}}}} \le \sqrt{2}\epsilon,
  \end{equation*}
  and this concludes the proof.
  
  \paragraph{Lipschitz inverse-cdf.} The sequence of iterates~\eqref{eq:sgd}
  yield error
  \begin{equation*}
    \err(\bar{x}_T;\L) \le \frac{\icdfLip}{n} + \frac{G R}{\sqrt{T}},
  \end{equation*}
  and setting $n\asymp \frac{\icdfLip}{\epsilon}, 
  T\asymp\frac{(GR)^2}{\epsilon^2}$
  concludes the proof of the theorem.
\end{proof}

\section{Proofs of Section~\ref{sec:multilevel}}\label{sec:prf-multilevel}
We now provide additional discussion of the multilevel Monte 
Carlo estimator for general functions $\funct$, whose form we restate here 
for convenience
\begin{equation}\label{eq:ml-def-restate}
\ml[\funct] \defeq \funct(x;S_1^{n_0}) + \frac{1}{q(J)} 
\mld_{2^J n_0},~\mbox{where}~
\mld_k \defeq \funct(x;S_1^{k}) - \frac{
	\funct(x;S_1^{k/2}) +\funct(x;S_{k/2+1}^{k})
}{2}.
\end{equation}
Section~\ref{sec:prf-mlmc-bounds} provides upper bounds on the  
moments of $\ml$ for estimating $\LcvarR$, $\Llam$, and their gradients, 
proving Claim~\ref{claim:mlmc} and Proposition~\ref{prop:ml-mom}. In 
that section we also prove that similar second moment bounds do not 
always hold for $\grad\Lcs$. In Section~\ref{prf:mlmc-complexity} we 
prove the complexity guarantees in Theorem~\ref{thm:ml}, and we conclude 
in Section~\ref{app:compare-with-blanchet} with a comparison of some of 
our design choices to the original proposal of~\citet{BlanchetGl15}.

\subsection{Proofs of moment bounds}\label{sec:prf-mlmc-bounds}

\begin{customclaim}{\ref*{claim:mlmc}'}
	  For any function $\funct$, the estimator $\ml[\funct]$ with parameters $n = 2^{\jmax} n_0$ satisfies
	\begin{equation*}
	\E \ml[\funct] = \E \funct(S_1^{\kmax}),
	~
	\text{requiring expected sample size}~\E 2^J n_0 =
	n_0(1+\log_2(\kmax/n_0)).
	\end{equation*}
\end{customclaim}
\begin{proof}
	For any even $k$, $\E \mld_k = \E \funct(S_1^k) - 
	\E\funct(S_1^{k/2})$. Therefore, the expectation of $\ml[\funct]$ 
	telescopes: 
	$\E\ml[\funct] = \E\brk{\funct(S_1^{n_0})} + \sum_{j=1}^{\jmax} 
	\E\mld_{2^j n_0} = 
	\E\brk{\funct(S_1^{\kmax})}$. The expected number of samples follows 
	from direct calculation: $\E[2^J] = \sum_{j=1}^{\jmax}2^j \P(J=j) = 
	\jmax + 1$.
\end{proof}

We  have the following bound on the second moment of the estimator,
\begin{equation}\label{eq:ml-second-moment}
\E\,\norm[\Big]{\ml\brk[\big]{\funct}}^2
\le 2\norm[\Big]{\funct(S_1^{n_0})}^2 + \sum_{j=1}^{\jmax} 
\frac{2}{q(j)}\E\norm*{\mld_{2^jn_0}}^2
\le 
2\norm[\Big]{\funct(S_1^{n_0})}^2 + \sum_{j=1}^{\jmax} 
2^{j+1}\E\norm*{\mld_{2^jn_0}}^2.
\end{equation}
For $\cs$-bounded (Definition~\ref{def:cs-bounded}) pure-penalty losses 
such as $\LcvarR$ and $\Llam$, we argue that 
$\E\norm{\mld_k}^2 \lesssim 1/k$, so that $2^{j} \E\norm{\mld_{2^j 
		n_0}}^2 \lesssim 1/n_0$. Substituting into the 
bound~\eqref{eq:ml-second-moment} gives the following guarantees, from 
which Proposition~\ref{prop:ml-mom} follows immediately via 
Claim~\ref{claim:cs-bdd}. 

\begin{customprop}{\ref*{prop:ml-mom}'}\label{prop:ml-mom-extended}
  Let $\L$ be an objective of the form~\eqref{eq:pop-dro-reg-restated} 
  with $\phi=0$ and strictly convex $\psi$. 
  If $\L$ is $\CSbdd$-$\cs$-bounded, we have that for all
  $x\in\X$, the multi-level Monte Carlo estimator with parameters $n$ and
  $n_0$ satisfies %
  \begin{equation*}
    \E{}\, \prn[\Big]{\ml\brk[\big]{\L}}^2 \le
    2B^2\prn*{1 + \frac{2C}{n_0}\log_2(n/n_0)}
    \mbox{~~and~~}
    \E{}\,\norm*{\ml\brk[\big]{\grad\L}}^2
    \le 2G^2\prn*{1 + \frac{2C}{n_0}\log_2(n/n_0)}.
  \end{equation*}
\end{customprop}
\begin{proof}
  The proof follows similarly to the proof of
  Proposition~\ref{prop:variance-extended}, where the key step is to bound
  $\mld_k$ for $k \in 2\N$. We distinguish between estimating the gradient and
  the loss as, for the latter, one needs to account for estimating the
  regularizer $\mathrm{D}_\psi$.

  \paragraph{Gradient estimator.} We start with the proof of the
  second moment of the gradient estimator. Let $k \in 2\N$ and let $q, q'$ 
  and $q''$ be the maximizer
  of~\eqref{eq:emp-dro-reg-restated} for $S_1^k, S_1^{k/2}$ and 
  $S_{k/2+1}^k$
  respectively. We have %
  \begin{flalign*}
    \norm{\mld_k} & = \norm[\Bigg]{ \sum_{i\le k} \prn*{q_i  - 
    \tfrac{1}{2}q'_i 
    \indic{i\le k/2} - \tfrac{1}{2}q''_{i-k/2} \indic{i> k/2}}\nabla \ell(x;S_i)} \\&
    \le G\sum_{i\le k/2}|q_i - \tfrac{1}{2}q'_i| + G\sum_{i > k/2}
    |q_i - \tfrac{1}{2}q''_{i-k/2}|.
  \end{flalign*}
   For
  $i \in \crl{1, \ldots, k/2}$, it holds that
  $q_i = \frac{1}{n}{\psi^{\ast}}'[(\ell(x;S_i) - \eta)/ \lambda]$ and
  $q'_i =\frac{2}{n} {\psi^{\ast}}'[(\ell(x;S_i) - \eta') / \lambda]$ for
  $\eta, \eta' \in \R$. Since that $\psi$ is strictly convex, ${\psi^\ast}'$ is
  increasing and $q_i - \tfrac{1}{2}q_i'$ is of constant sign for
  $i \in \crl{1, \ldots, k/2}$. Therefore,
  \begin{equation*}
    \sum_{i\le k/2}\abs*{q_i - \half q'_i} = \left\vert \sum_{i\le k/2} 
    \prn*{q_i - \half q_i'}\right\vert
    = \left\vert \sum_{i\le k/2} q_i - \frac{1}{2} \right\vert.
  \end{equation*}
  By symmetry, it thus holds that
  \begin{equation*}
    \E{}\,\norm{\mld_k}^2 \le 4G^2\E{}\,\prn*{\sum_{i\le k/2}q_i - \frac{1}{2}}^2
    \stackrel{(i)}{\le} \frac{2}{k}\divcs(q, \tfrac{1}{k}\ones)
    \stackrel{(ii)}{\le} \frac{2CG^2}{k},
  \end{equation*}
  where $(i)$ is due to Lemma~\ref{lem:subset-cs} and $(ii)$ follows from 
  the assumption
  that $\L$ is $\CSbdd$-$\cs$-bounded. Substituting 
  into~\eqref{eq:ml-second-moment}, we have
  \begin{align*}
    \E{}\, \norm{\ml\brk[\big]{\nabla \L}}^2
    & \le 2G^2 + 4CG^2\sum_{j\le 1}^{\jmax}\frac{1}{q(j)2^jn_0} \\
    & \le 2G^2 + \frac{4CG^2}{n_0}\prn*{\jmax - \frac{1}{2}} \\
    & \le G^2\prn*{2 + \frac{4C}{n_0}\log_2(n/n_0)}.
  \end{align*}
  This concludes the argument for the gradient.
  
  \paragraph{Loss estimator.} %
  With the same notation, let us
  define $\tilde{q} \defeq [\half q', \half q''] \in \Delta^k$. We first prove
  that $\mld_k \ge 0$. Indeed, we have
  \begin{align*}
    \L(\x;S_1^k) =\sum_{i\le k} q_i \ell(x;S_i) - \lambda\divpsi(q, \tfrac{1}{k}\ones)
    & \stackrel{(i)}{\ge} \sum_{i\le k} \tilde{q}_i \ell(x;S_i) -
      \lambda\divpsi(\tilde{q}, \tfrac{1}{k}\ones) \\
    & \stackrel{(ii)}{=} \frac{1}{2}\L(\x; S_1^{k/2})
      + \frac{1}{2}\L(\x;S_{k/2+1}^k),
  \end{align*}
  where $(i)$ is because $q$ is the maximizer for $S_1^k$ and $(ii)$ 
  because the
  $\psi$-divergence tensorizes, i.e.,
  $\divpsi(\tilde{q}, \tfrac{1}{k}\ones) = \frac{1}{2}\divpsi(q',
  \tfrac{2}{k}\ones) + \frac{1}{2}\divpsi(q'', \tfrac{2}{k}\ones)$.  This
  guarantees that $\mld_k = \L(\x;S_1^k) - \frac{1}{2}\L(\x; S_1^{k/2})- 
  \frac{1}{2}\L(\x;S_{k/2+1}^k)\ge 0$. 
  
  Let us now upper bound $\mld_k$.  To that
  end, we define $\tilde{q}' = 2q_1^{k/2} + \delta$ where $\delta \in \R^{k/2}$
  is a fixed-sign vector such that $\tilde{q}'$ lies in $\Delta^{k/2}$. More 
  precisely, if
  $\tilde{q}'^\top\ones > 1$, $\delta$ decreases the mass of the largest
  coordinate until $\tilde{q}'_{(1)} = \frac{2}{k}$ and iterates along the
  sorted coordinates until $\tilde{q}' \in \Delta^{k/2}$. If
  $\tilde{q}'^\top\ones < 1$, $\delta$ similarly increases the smallest
  coordinate to $\frac{2}{k}$ until $\tilde{q}' \in \Delta^{k/2}$. 
  Without loss of generality, we can assume
  that $\psi$ attains its minimum at $t=1$ (otherwise may 
  replaced it by $\psi(t)-\psi'(1)(t-1)$ without changing the objective). 
  Therefore, since
  $\tilde{q}'$ is closer to $\tfrac{2}{k}\ones$ than $2q_1^{k/2}$, it holds that
  \begin{equation*}
    \divpsi(\tilde{q}', \tfrac{2}{k}\ones) =
    \frac{2}{k}\sum_{i\le k/2}\psi(\tfrac{k\tilde{q}'}{2})
    \le \frac{2}{k}\sum_{i\le k/2}\psi(kq_i)
  \end{equation*}
  Finally, we know that $q'$ is optimal for $S_1^{k/2}$ and so
  \begin{align*}
    \L(\x; S_1^{k/2})
    & \ge \sum_{i \le k/2}\tilde{q}'_i \ell(x;S_i)
      - \lambda\divpsi(\tilde{q}', \tfrac{2}{k}\ones) \\
    & \ge 2\sum_{i\le k/2}q_i \ell(x;S_i) 
      - \sum_{i\le k/2}[-\delta_i]_+B  -  \lambda\frac{2}{k} \sum_{i\le 
      	k/2}\psi(kq_i)\\
    & =  2\sum_{i\le k/2}q_i \ell(x;S_i) 
      - 2\brk*{\sum_{i\le k/2}q_i - \frac{1}{2}}_+B
      - \lambda\frac{2}{k} \sum_{i\le k/2}\psi(kq_i).
  \end{align*}
  The same argument for the indices $\crl{k/2+1, \ldots, k}$ yields (recall that $\divpsi(q, \frac{1}{k}\ones) = \frac{1}{k}\sum_{i\le k} \psi(kq_i)$)
  \begin{equation*}
    \mld_k \le 2B\crl*{\brk*{\sum_{i=1}^{k/2}q_i - \frac{1}{2}}_+
      + \brk*{\sum_{i=k/2+1}^{k}q_i - \frac{1}{2}}_+} = 2B \abs*{
      \sum_{i=1}^{k/2}q_i - \frac{1}{2}
    }.
  \end{equation*}
  Therefore, we have
  \begin{equation*}
    \E(\mld_k)^2 \le 4B^2\E\brk*{\sum_{i\le k/2}q_i
      - \frac{1}{2}}^2 \stackrel{(i)}{\le} \frac{2CB^2}{k},
  \end{equation*}
  where $(i)$ follows from Lemma~\ref{lem:subset-cs} and the 
  $\CSbdd$-$\cs$-boundedness of $\L$. Substituting 
  into~\eqref{eq:ml-second-moment} yields the desired bound on 
  $\ml\brk{\L}$.
\end{proof}

Having established the gradient estimator upper bounds for pure-penalty 
objectives, we demonstrate that similar bounds \emph{do not} extend to 
the case of $\cs$ constraint.
  
      \begin{proposition}[Lower bound in the case of constrained-$\cs$]
        \label{prop:lb-cs-mlmc}
        For every $\rho\ge 1$, $n_0$ and $n\ge 4$,
        there exists a
        distribution $P_0$ over $\ss=\{0,1,2\}$ and a $G$-Lipschitz loss
        $\ell: [-1,1]\times\ss \to \R_+$ such that the multi-level Monte Carlo 
        gradient
        estimator with parameters $n_0$ and $n$ satisfies
	\begin{equation*}
	\E\norm[\Big]{\ml\brk[\big]{\grad\Lcs}}^2 \gtrsim 
	 \frac{n}{n_0}  G^2.
	\end{equation*}
      \end{proposition}

  \providecommand{\Count}{\mathrm{C}}
  \providecommand{\Ev}{\mathfrak{E}}
  
      \begin{proof}
      	We reuse the construction and notation in the proof of  
      	 Proposition~\ref{prop:lb-variance} and so do not repeat it.
    	 For a sample $S_1^n$, we consider the event $\Ev_a(S_1^{n/2})$ 
    	 where $S_i \ne 2$ for all $i\le n/2$ and there are at least $n p_1 /2$ 
    	 samples with value 1. We argue in the proof of 
    	 Proposition~\ref{prop:lb-variance} (Eq.~\eqref{eq:grad-ev-a}) that 
    	 under this event we have
    	 \begin{equation*}
    	 \E\brk{\grad \Lcs(0;S_{1}^{n/2})\mid \Ev_a(S_1^{n/2})} = G.
    	 \end{equation*}
    	 Moreover, we have
    	 \begin{equation*}
    	 \grad \Lcs(0;S_{n/2+1}^{n}) \ge - G
    	 \end{equation*}
    	 with probability 1, so overall
    	 \begin{equation*}
    	 \E\brk*{\half\grad \Lcs(0;S_{1}^{n/2})+ \half\grad 
    	 \Lcs(0;S_{n/2+1}^{n})~\Big\vert~ \Ev_a(S_1^{n/2})} \ge 0.
    	 \end{equation*}
    	 
    	 We also consider the event $\Ev_b(S_1^n)$ that there is exactly one 
    	 sample with value 2 and less the $np_1$ samples with value 1. As per 
    	 the proof of Proposition~\ref{prop:lb-variance} 
    	 (Eq.~\eqref{eq:grad-ev-b}) we have
    	 \begin{equation*}
    	 \E\brk{\grad \Lcs(0;S_{1}^{n})\mid \Ev_b(S_1^{n})} \le -\frac{1}{6}G,
    	 \end{equation*}
    	 where we used $\rho \ge 1$. 
    	 
    	 Moreover, by the arguments in the proof of 
    	 Proposition~\ref{prop:lb-variance}, we have
    	 \begin{equation*}
    	 \P\prn*{\Ev_a(S_1^{n/2}) \cap \Ev_b(S_1^n)} \gtrsim 1.
    	 \end{equation*}
    	 Therefore, since $\mld_n = \grad \Lcs(0;S_{1}^{n}) - \half\grad 
    	 \Lcs(0;S_{1}^{n/2})- \half\grad \Lcs(0;S_{1}^{n/2})$, we have
    	 \begin{flalign*}
    	 \E \norm{\mld_n}^2 &\ge \E \brk*{ \norm{\mld_n}^2 \mid 
    	 \Ev_a(S_1^{n/2}) \cap 
    	 \Ev_b(S_1^n)} \P\prn*{\Ev_a(S_1^{n/2}) \cap \Ev_b(S_1^n)}
    	 \\&\ge 
    	 \frac{G^2}{36}\P\prn*{\Ev_a(S_1^{n/2}) \cap \Ev_b(S_1^n)} \gtrsim 
    	 G^2.
    	 \end{flalign*}
    	 The proof is complete by noting that
      	\begin{equation*}
      	\E \abs*{\ml[ \grad \Lcs(0;\cdot)]}^2 \ge 2^{\jmax-1} \E 
      	\norm{\mld_n}^2 
      	= \frac{n}{2n_0} \E \norm{\mld_n}^2  \gtrsim  \frac{n}{n_0}G^2.
      	\end{equation*}
      \end{proof}
      Since the number $T$ of SGM iterations must be proportional to the 
      second moment of the gradient estimator, 
      Proposition~\ref{prop:lb-cs-mlmc} tells us that in the worst case we 
      might have to set $T\asymp n (GR)^2/\eps^2$, in which case we 
      might as well use a mini-batch estimator with batch size $n$ and run 
      $(GR)^2/\eps^2$ SGM steps.
      
      \subsection{Proof of complexity bounds}\label{prf:mlmc-complexity}
      \restateThmMLMC*
      \begin{proof}
        The convergence guarantee of Proposition~\ref{prop:gangster} 
        and
        the second moment bound of Proposition~\ref{prop:ml-mom} directly 
        give
        that iterates of the form~\eqref{eq:sgd} with the MLMC gradient
        estimator guarantees a regret smaller than $\epsilon$ for
        $n\asymp \frac{B^2}{\alpha \epsilon^2}$,
        $1\lesssim n_0 \lesssim \frac{\log n}{\alpha}$ and
        $T\asymp\frac{(GR)^2}{n_0\alpha\epsilon^2}\log^2{n}$. However, 
        since the
        multilevel estimator randomizes the batch size, it remains to show that
        the number of samples concentrates below the claimed bound. Let 
        $K_t = n_0 2^{J_t}$ be the batch size at time $t$, and note that
        \begin{flalign*}
        \E K_1 &= n_0 \log_2 \frac{2n}{n_0},\\
        \E K_1^2 &= 3n_0 n - 2n_0^2\le 3n_0 n,~~\mbox{and}\\
         K_1 &\le  n~~\mbox{with probability 1}.
        \end{flalign*} 
        Therefore, since $K_1^T$ are iid, a one-sided Bernstein 
        bound~\cite[Prop. 2.14]{Wainwright19} implies that
        \begin{equation*}
          \P\brk*{\sum_{t \le T}K_t \ge n_0\log_2(2n/n_0)T + \delta}
          \le \exp\prn*{-\frac{\delta^2}{6n_0nT + \tfrac{n\delta}{3}}}.
        \end{equation*}
        Solving in $\delta$ for the RHS to be equal to $\frac{1}{n}$ yields
        $\delta = \frac{n\log n}{3}\prn{1+\sqrt{1+216\tfrac{T n_0}{n\log n}}}$.
        We replace $n, n_0$ and $T$ by their values and conclude the proof.
      \end{proof}
      
\subsection{Comparison 
  with~\citet{BlanchetGl15}}\label{app:compare-with-blanchet} There are 
  two differences between our MLMC estimator and the proposal of 
  \citet{BlanchetGl15}. First, we take $J$ to be a truncated 
  $\mathsf{Geo}(1/2)$ random variable while they suggest $J\sim 
  \mathsf{Geo}(2^{-3/2})$ without truncation---as we further discuss 
  below, this modification is crucial for ensuring a useful second moment 
  bound in our setting. The second difference is that we allow for a 
  minimum sample size $n_0>1$ as opposed to $n_0=1$ in~ 
  \cite{BlanchetGl15}. This modification is somewhat less important, as 
  $n_0=1$ suffices for optimal gradient complexity, but choosing slightly 
  larger $n_0$ is helpful in practice and can provably reduce the sequential 
  depth of SGM by logarithmic factors.

  Let us discuss in more detail the choice $p=1/2$ in our construction of 
  $J\sim \min\{\mathsf{Geo}(p), \jmax\}$. Inspection of 
  Claim~\ref{claim:mlmc} shows that $p<1/2$ implies that the expected 
  sample cost is $\E 2^J n_0 \le \frac{n_0}{1-2p}$ independent of 
  $n=2^{\jmax}n_0$, so in principle we could compute unbiased estimates 
  even for $\E \funct(S_1^\infty)$, i.e., the population objective. However, 
  any $p<1/2$ would result in overly large second moments: substituting  
  $q(j)\propto p^{-j}$ and $\E\norm{\mld_k}^2 \asymp 1/k$ 
  in~\eqref{eq:ml-second-moment} would result in bounds scaling with 
  $(n/n_0)^{\log_2 1/(2p)}$. Therefore, $p=1/2$ is the only value for which 
  both the second moment and expected number of samples are 
  sub-polynomial in $n$. In contrast,~\citet{BlanchetGl15} 
  apply 
  the MLMC estimator to more regular functionals for which 
  $\norm{\mld_k}^2 \lesssim 1/k^2$, and consequently can use a smaller 
  value 
  for $p$. %
\section{Lower bound proofs}\label{sec:prf-lowerbounds}

This section proves our lower bounds, which we restate for ease of 
reference.

\restateThmLowerBounds*

Since our proofs for CVaR and $\cs$ penalty are quite different, we 
present them separately in Theorems~\ref{thm:cvar-lb} and 
\ref{thm:lam-lb}, respectively. 

\subsection{CVaR lower bound}
To prove the CVaR lower bound we use the following standard Le Cam 
reduction from stochastic optimization to hypothesis testing.

\begin{lemma}{\cite[Chapter 5]{Duchi18}}\label{lem:opt-to-test}
	Let $\mc{P}$ be a set of distributions and $P_{-1}, P_1 \in \mc{P}$ and 
	define
	\begin{equation*}
	\mathsf{d}_\mathrm{opt}(P_{1}, P_{-1}) \defeq \sup\crl*{\delta' \ge 0
		\;\middle|\; \mbox{no $x\in\X$ is $\delta'$-optimal for both 
		$\L(\cdot; P_{-1})$ and $\L(\cdot; P_{1})$}}.
	\end{equation*}
		Then for any 
	measurable mapping $\hat{x}_n:\ss^n\to\mc{X}$ we have
	\begin{equation*}
	\sup_{P\in\mc{P}}
	\E_{S_1^n \sim P^n}\crl*{\L(\hat{x}_n(S_1^n); P)} - 
	\inf_{x'\in\mc{X}}\L(x'; P)
	\ge \frac{\mathsf{d}_\mathrm{opt}(P_{1}, P_{-1})}{2}\prn*{1-\sqrt{\frac{ 
	n}{2} \divkl(P_{-1}, P_{1})}}, 
	\end{equation*}
\end{lemma}

Armed with Lemma~\ref{lem:opt-to-test}, we state and prove the
lower bound for CVaR.

\begin{customthm}{\ref*{thm:lb}a}[CVaR lower bound]\label{thm:cvar-lb}
  Let $G,R,\alpha>0$, $\epsilon\in(0,GR/64)$, $\ss=[-G,G]$, 
  $\xset=[-R,R]$, 
  and
 $\ell(x,s) = x\cdot s$. For any (potentially randomized) mapping
 $\hat{x}_n : \ss^n \to \xset$ there exists a distribution $P_0$ over $\ss$
 such that, 
 \begin{equation*}
   n \le \frac{(GR)^2}{2048\alpha \epsilon^2}
   \mbox{~~implies~~} \E\Lcvar(\hat{x}_n(S_1^n);P_0) - 
   \inf_{\x\in\X}\Lcvar(\x; P_0)
   \ge \epsilon.
 \end{equation*}
\end{customthm}

  \newcommand{\cvarrm}{\mathrm{CVaR}^\alpha} 

\begin{proof}  %
  Let us first assume that $\alpha \le \frac{1}{2}$. For
  $\delta \le \min\crl{\alpha, 1-2\alpha}$, $\mu>0$ and $v\in\{-1,1\}$ we
  consider the distributions $P_v$ such that for $\S_v \sim P_v$ we
  have %
  \begin{equation}\label{eq:cvar-hard-distributions}
    \S_v = G\cdot\begin{cases}
      \mu & \mbox{~~with probability~~} \alpha + \delta v \\
      -1 & \mbox{~~with probability~~} 1 - \alpha - \delta v
    \end{cases}.
  \end{equation}
  For $\x\in[-R, R]$, we let
  $\ell(\x;s) = x \cdot s$. Since the CVaR objective is positively
  homogeneous, we have 
  \begin{equation*}
  \Lcvar(x;P_v) = |x| \cdot \Lcvar(\sign(x);P_v). 
  \end{equation*}
  It 
  therefore suffices to compute $\Lcvar(\pm 1; P_{\pm 1})$.  A quick
  calculation yields
  \begin{flalign*}
   & \Lcvar(1;S_1) = G\mu,~~ \Lcvar(-1,S_{1}) = G,\\
    &\Lcvar(1;S_{-1}) =
    G\mu\prn*{1-\frac{\delta}{\alpha}} - G\frac{\delta}{\alpha}
    ~~\mbox{and}~~
     \Lcvar(-1;S_{-1})  = G.
  \end{flalign*}
  We thus have a closed-form expression for the CVaR objective: for $P_1$ 
  we have
    \begin{equation*}
  \Lcvar(\x;P_1) = -Gx\indic{x\le 0} + Gx\mu\indic{x\ge 0},
  \end{equation*}
  which clearly
  attains its minimum at $x=0$ where it has value $0$. 
  Choosing $\mu$ such that  
  \[
  \mu = 
  \frac{\delta}{2\alpha}\prn*{1-\frac{\delta}{2\alpha}}^{-1}
  \]
  gives  $\Lcvar(1;S_{-1}) = -G\mu$ and
  \begin{equation*}
  \Lcvar(\x;P_{-1}) = -Gx\indic{x\le 0} -Gx\mu\indic{x\ge 0},
  \end{equation*}
 which attains its minimum at $x=R$ where it has value
  $-GR\mu$.
  We therefore have that
  \begin{equation}\label{eq:sep-bound}
  \mathsf{d}_\mathrm{opt}(P_{1}, P_{-1}) = \frac{GR\mu}{2}  \ge 
  \frac{GR\delta}{4\alpha}.
  \end{equation}
  Moreover, we have $t\log t -t + 1 \le (t-1)^2$ for all $t\ge 0$, so that 
  $\divkl(Q,P)\le 2\divcs(Q,P)$ for all $Q,P$, and in particular
  \begin{equation}\label{eq:kl-bound}
  \divkl(P_{-1}, P_{1})  \le 2\divcs(P_{-1}, P_{1}) = 
  \frac{4\delta^2}{(1-\alpha-\delta)(\alpha + \delta)} \le 
  \frac{8\delta^2}{\alpha},
  \end{equation}
  where that last transition used $\delta \le \alpha$ and $\alpha \le 
  1/2$. 
  
  We take
  \begin{equation*}
  \delta = \sqrt{\frac{\alpha}{16(n+\alpha^{-1})}},
  \end{equation*}
  where so that  $\divkl(P_{-1}, P_{1})\le 1/(2n)$ and 
  Lemma~\ref{lem:opt-to-test} combined with~\eqref{eq:sep-bound} 
  and~\eqref{eq:kl-bound} gives 
  \begin{equation*}
  \sup_{P\in\mc{P}}
  \E_{S_1^n \sim P^n}\crl*{\L(\hat{x}_n(S_1^n); P)} - 
  \inf_{x'\in\mc{X}}\L(x'; P)
  \ge \frac{GR}{32\sqrt{\alpha n + 1}},
  \end{equation*}
  and the result follows from substituting $n \le 
  \frac{(GR)^2}{2048\alpha \epsilon^2}$. When $\alpha \ge 1/2$ the result 
  follows from the standard lower bound for
  stochastic convex optimization (e.g.~\cite[Thm. 5.2.10]{Duchi18}).
\end{proof}

\subsection{Penalized-$\cs$ lower bound}
Computation of $\Llam(x;P_{\pm})$ for the CVaR lower bound 
construction~\eqref{eq:cvar-hard-distributions} shows that the argument 
does not easily transfer to the penalized-$\cs$ objective because---as 
opposed to constrained-$\cs$ and CVaR---it is not positive 
homogeneous in $x$. 

Sidestepping this difficulty, we prove our lower bound using the  
different machinery of high-dimensional hard instances for  oracle-based 
optimization~\cite{NemirovskiYu83}. We consider two standard oracles. 
First is the deterministic first-order oracle, that for a function
$f:\R^d \to \R$ and a query $x$ returns
\begin{equation*}
  \mathsf{O}^\mathrm{D}_f(x) \defeq \prn{f(x), \nabla f(x)},
\end{equation*}
where we recall that $\nabla f(x)$ is an arbitrary element of $\partial
f(x)$. Second is the stochastic oracle, that for a loss function 
$\ell:\X\times\ss \to \R$ and distribution 
$P_0$ returns the randomized mapping
\begin{equation*}
  \mathsf{O}^{\mathrm{S}}_{\ell, P_0}(x)
  \defeq \prn{\ell(\x;\S), \nabla\ell(\x;\S)}, \mbox{~~for~~} S\sim P_0.
\end{equation*}

We construct the  hard instance for $\Llam$ based on the standard hard 
instance for non-stochastic convex optimization, whose properties are as 
follows.

\begin{proposition}[\citet{BraunGuPo17}, Theorem 
V.1]\label{prop:standard-oracle-lower-bound}
  Let $\epsilon, G, R > 0$. There exists
  $d_\epsilon \lesssim (GR)^2\epsilon^{-2}\log\frac{GR}{\epsilon}$ such 
  that the following holds for $\X=\{x\in\R^{d_\epsilon}\mid \norm{x}\le 
  R\}$. For
  any (possibly randomized) algorithm there exists $f_\epsilon:\X\to [0, 
  GR]$ 
  convex
  and $G$-Lipschitz such the query $x_T$ 
  $\mathsf{O}^\mathrm{D}_{f_\epsilon}$ at iteration $T$ 
   satisfies
  \begin{equation*}
    T \le c\frac{(GR)^2}{\epsilon^2}
    \mbox{~~implies~~} \E f_\epsilon(x_T) - \inf_{\norm{x'} \le R}f_\epsilon(x') \ge \epsilon,
  \end{equation*}
  for a numerical constant $c > 0$.
\end{proposition}

In other words, any ``dimension-free'' algorithm needs to interact
$\Omega(\epsilon^{-2})$ times with the deterministic oracle to obtain an 
$\epsilon$-suboptimal
point. With this result, we prove our lower bound for optimizing $\Llam$.

\begin{customthm}{\ref*{thm:lb}b}[Penalized-$\cs$ lower 
bound]\label{thm:lam-lb}
	Let $G,R,\lambda>0$ and $\epsilon\in(0, GR)$. There exists
	$d_\epsilon \lesssim (GR)^2\epsilon^{-2}\log\frac{GR}{\epsilon}$ such
	that the following holds for
	$\X=\{x\in\R^{d_\epsilon}\mid \norm{x}\le R\}$ and
	$\ss \subseteq [-1,1]$.  For every algorithm there exists a distribution
	$P_0$ over $\ss$ and $\ell:\X\times\ss\to[-GR,GR]$ convex and 
	$G$-Lipschitz
	in $x$, such that the query $x_T$ to $\mathsf{O}^{\mathrm{S}}_{\ell, 
	P_0}$ at iteration $T$ satisfies
	\begin{equation*}
	T \le c\frac{(GR)^3}{\lambda\epsilon^2} \mbox{~~implies~~}
	\E\brk{\Llam(x_T;P_0)}-\min_{x'\in\X}\Llam(x';P_0) > 
	\epsilon,
	\end{equation*}
	for $c>0$ independent of $G,R,\lambda$ and $\epsilon$.
\end{customthm}
\begin{proof}
 Consider any convex and $G$ Lipschitz $f:\X\to[0,GR]$, 
  define 
  $\ss \defeq \crl{0, 1}$ and $P_0 = \bernrv(\frac{\lambda}{GR})$, 
  and construct the following loss
  \begin{equation*}
    \ell(\x; \S) \defeq \begin{cases}
      f(x) & \mbox{~~if~~} \S = 1 \\
      -GR & \mbox{~~if~~} \S = 0.
    \end{cases}
  \end{equation*}
  (If $\lambda > GR$ the 
  result follows from the standard $(GR)^2/\epsilon^2$ lower bound for 
  convex optimization). 
 Expressing the 
  resulting objective $\Llam$ with the dual form~\eqref{eq:lam-dual} 
gives
  \begin{align*}
    \Llam(\x; P_0) 
                   & = \inf_{\eta\in\R} \crl*{\frac{\lambda}{2} + \eta
                     +                 
\frac{1}{2\lambda}\brk*{\frac{\lambda}{GR}\prn*{f_\epsilon(x) 
                     - \eta}^2_+
                     + \prn*{1-\frac{\lambda}{GR}}(-GR-\eta)_+^2}}
                 \\&=
                 f(x) - \frac{GR-\lambda}{2},
  \end{align*}
  since $\eta\opt = f(x)-GR \ge -GR$. We get that minimizing $\Llam$ is  
  equivalent to optimizing $f$. 
  
  Fix an algorithm interacting with $\mathsf{O}^{\mathrm{S}}_{\ell, P_0}$ 
  and note that it implies a (randomized) algorithm interacting with 
  $\mathsf{O}^{\mathrm{D}}_{f}$. Therefore we may take $f=f_{\epsilon}$, 
  the hard function for this algorithm that 
Proposition~\ref{prop:standard-oracle-lower-bound} guarantees.
  Note that an algorithm interacting with $\mathsf{O}^{\mathrm{S}}_{\ell, 
  P_0}$ receives information on $f_\epsilon$ only when $S=1$. Therefore, 
  the worst-case expected optimality gap when minimizing $\Llam$ with 
  $T$ queries to $\mathsf{O}^{\mathrm{S}}_{\ell, P_0}$ is identical to the 
  worst-case expected optimality gap  when minimizing $f_\epsilon$ with 
  $\binrv(T, \frac{\lambda}{GR})$ queries. Therefore, 
  Proposition~\ref{prop:standard-oracle-lower-bound} tells us that for 
  some $c'>0$,
  \begin{align*}
    \E[\L(x_T;P_0) - \inf_{\norm{x'}\le R} \L(x';P_0)]
    & \ge \epsilon  \cdot
      \P\prn*{ \binrv \prn[\Big]{T, \frac{\lambda}{GR}} \le c'\cdot 
      \frac{(GR)^2}{\epsilon^2}}.
  \end{align*}
  Substituting $T\le \frac{c}{4} \cdot \frac{(GR)^3}{\lambda\epsilon^2}$ 
  gives that $\P( \binrv(T, \frac{\lambda}{GR}) \le c'\cdot 
  \frac{(GR)^2}{\epsilon^2}) \ge \frac{1}{2}$ by a standard Chernoff bound. 
  The result follows by properly adjusting the constant factors (e.g., 
  replacing $\epsilon$ with $2\epsilon$). 
\end{proof} \arxiv{%
\section{Doubling schemes proofs}\label{app:doubling}

We now complete the proofs of the claims in Section~\ref{sec:doubling}.

\restateCSDualRange*
\begin{proof}
	Le
	$x\opt, \lambda\opt = \argmin_{x\in\X,\lambda \ge 0}  
	\crl*{f_\rho(x,\lambda)
	}$, noting that $\min_{x'\in\X} \Lcs(x';P_0) = 
	f_\rho(x\opt,\lambda\opt)$. 
	For any $x,\lambda$ let $Q\opt_{x,\lambda}$ be the maximizing $Q$ 
	in~\eqref{eq:pop-dro-reg-restated} for these values of $x,\lambda$. 
	Moreover, let $\div(x,\lambda) = \divcs(Q\opt_{x,\lambda}, P_0)$. By 
	Claim~\ref{claim:cs-bdd}, for all $\lambda > B/\rho$ we have that 
	$\div(x,\lambda) < \rho$, and consequently $\lambda > \lambda\opt$, 
	i.e., $\lambda\opt \le B/\rho$, and hence that upper bound has no 
	impact on accuracy.
	
	When in addition we have $\lambda\opt \ge \epsilon/(2\rho)$ then 
	clearly $\min_{x\in\X} \LcsRange{\frac{\epsilon}{2\rho}, 
		\frac{B}{\rho}}(x;P_0) = \LcsRange{0, \infty}(x;P_0)
	= \min_{x'\in\X} \Lcs(x';P_0)$. Otherwise, if $\lambda\opt < 
	\epsilon/(2\rho)\eqdef \lambda_\eps$ we may write
	\begin{flalign*}
	\min_{x\in\X} \LcsRange{\frac{\epsilon}{2\rho}, \frac{B}{\rho}}(x;P_0)
	&\stackrel{(i)}{\le} 
	f_\rho(x\opt, \lambda_\eps) \stackrel{(ii)}{\le}  
	f_\rho(x\opt,\lambda\opt)
	+ \brk*{\tfrac{\del}{\del\lambda}f_\rho(x\opt, \lambda_\eps)} 
	(\lambda_\eps 
	-\lambda\opt)
	\\&= \min_{x'\in\X} \Lcs(x';P_0) +\brk*{\rho - 
		\div(x\opt,\lambda_\eps)}(\lambda_\eps - \lambda\opt)
	\\&\stackrel{(iii)}{\le} 
	\min_{x'\in\X} \Lcs(x';P_0) + \lambda_\eps \rho = 
	\min_{x'\in\X} \Lcs(x';P_0) + \frac{\eps}{2}.
	\end{flalign*}
	Where we used $(i)$ that $x\opt$ and $\lambda_\eps$ are feasible 
	points in the joint minimization of $f_\rho(x,\lambda)$ over $x\in\X$ 
	and 
	$\lambda\in [\lambda_\eps, B/\rho]$; $(ii)$ the convexity of $f$ in 
	$\lambda$; and $(iii)$ the fact that $\div(x\opt,\lambda_\eps) \ge 0$ 
	and $\lambda\opt \le \lambda_\eps$.
\end{proof}

\restateLemLambdaMom*
\begin{proof}
	Recall the definition~\eqref{eq:ml-def-restate} of the MLMC estimator of  
	a general $\funct$ and the expression~\eqref{eq:ml-second-moment} 
	for its second moment. Suppose that $\funct(\cdot) = \funct_1(\cdot) + 
	\funct_2(\cdot) + c$, where $c$ is a constant. Then
	\begin{equation*}
	\E \norm{\mld_k \brk[\big]{ \funct}}^2 
	= \E \norm{\mld_k \brk[\big]{ \funct_1 +\funct_2}}^2 
	\le 2\E \norm{\mld_k \brk[\big]{ \funct_1 }}^2 + 2\E \norm{\mld_k 
		\brk[\big]{ \funct_2 }}^2.
	\end{equation*}
	Consequently, by~\eqref{eq:ml-second-moment}, we have
	\begin{equation}\label{eq:ml-sum-decomp}
	\E \norm{\ml \brk[\big]{ \funct}}^2 \le 2c^2 + 2\E \norm{\ml \brk[\big]{ 
			\funct_1}}^2 + 2 \E \norm{\ml \brk[\big]{ \funct_2}}^2.
	\end{equation}
	
	We apply this observation to $\ml\brk[\big]{\tfrac{\del}{\del \lambda} 
		\Llam^{\lambda}(x;\cdot)+\rho}$ by noting that
	\begin{equation*}
	\tfrac{\del}{\del \lambda} \Llam^{\lambda}(x;S_1^n) = -\divcs(q\opt; 
	\tfrac{1}{n}\ones) =  
	\frac{1}{\lambda}\prn[\Bigg]{\Llam^\lambda(x;S_1^n) 
		-\frac{1}{n}\sum_{i\le n} 
		q\opt_i \ell(x;S_i)}.
	\end{equation*}
	Proposition~\ref{prop:ml-mom-extended} gives us the bound $\E 
	\prn*{\ml \brk[\big]{ \Llam^\lambda(x;S_1^n) }}^2\lesssim B^2\prn*{1 + 
		\frac{B}{\lambda n_0} 
		\log\frac{n}{n_0}}$. Moreover, we have that
	\begin{equation*}
	\E \prn*{\ml \brk[\Big]{ \frac{1}{n}\sum_{i\le n} 
			q\opt_i \ell(x;S_i)}}^2 \lesssim B^2\prn*{1 + \frac{B}{\lambda n_0} 
		\log\frac{n}{n_0}}
	\end{equation*}
	By exactly the same argument that proves the gradient second moment 
	bound in Proposition~\ref{prop:ml-mom-extended}. The result then 
	follows by substituting into~\eqref{eq:ml-sum-decomp}.
\end{proof}

\restateLemXLamSGD*
\begin{proof}
	We take $n\asymp \frac{B}{\lambdaL \epsilon}$ to guarantee bias below 
	$\epsilon/2$ by Proposition~\ref{prop:batch-bias}, and we take $n_0 
	\asymp \frac{B}{\lambdaL}\log n$ to guarantee that 
	\begin{equation*}
	\Gamma_x^2 \defeq \sup_{x\in\X, \lambda\in[\lambdaL,\lambdaH]}
	\E \norm*{
		\ml\brk[\big]{\grad \Llam^{\lambda}(x; \cdot)}}^2 \lesssim G^2
	\end{equation*}
	and, by Lemma~\ref{lem:ml-lambda-mom},
	\begin{equation*}
	\Gamma_\lambda^2 \defeq \sup_{x\in\X, 
		\lambda\in[\lambdaL,\lambdaH]}
	\E \prn*{\ml\brk[\big]{\tfrac{\del}{\del \lambda} 
			\Llam^{\lambda}(x;\cdot)+\rho}}^2 \lesssim 
	\frac{B^2}{\lambdaL^2} + \rho^2.
	\end{equation*}
	Let $\bar{\lambda}_T = \sum_{t\le T} \lambda_t$ be the average of the 
	$\lambda$ iterates in~\eqref{eq:x-lambda-sgd}. 
	By appropriate choice of $\eta$ and $\eta'$ we guarantee (via 
	Proposition~\ref{prop:gangster}) that
	\begin{equation*}
	\E \LcsRange{\lambdaL,\lambdaH}(\bar{x}_T;P_0)   
	\le \E f_\rho(\bar{x}_T, \bar{\lambda}_T)
	\le \min_{x\in\X, \lambda\in[\lambdaL, \lambdaH]}f_\rho(x,\lambda) + 
	\mathrm{err}_T
	= \min_{x\in\X}\LcsRange{\lambdaL,\lambdaH}(x;P_0)  
	+ \mathrm{err}_T,
	\end{equation*} 
	where
	\begin{equation*}
	\mathrm{err}_T \lesssim  \frac{\eps}{2} + \frac{\Gamma_x R + 
		\Gamma_\lambda (\lambdaH - \lambdaL)}{\sqrt{T}}.
	\end{equation*}
	Therefore, by taking 
	\begin{equation*}
	T\asymp \frac{\Gamma_x^2 R^2 + 
		\Gamma_\lambda^2 \lambdaH^2}{\epsilon^2} \asymp\frac{(GR)^2 + 
		B^2 \lambdaH^2 / 
		\lambdaL^2 + 
		\lambdaH^2 \rho^2}{\epsilon^2} 
	\end{equation*}
	we guarantee that $\mathrm{err}_T 
	\le \eps$, and the complexity bound follows from substituting $n_0, 
	n$ and $T$ in the high probability upper bound  
	$n_0\log_2\prn*{\frac{n}{n_0}}T +
	5\sqrt{(n\log n)^2 + n_0n T \log n}$
	shown in Theorem~\ref{thm:ml}.
\end{proof}

\restateThmDoubling*
\begin{proof}
	By Lemma~\ref{lem:x-lambda-sgd}, finding an $\epsilon$ approximate 
	solution in the interval $[\lambda\pind{i+1},\lambda\pind{i}]$ requires  
	\begin{equation*}
	\lesssim \prn*{1+\frac{\rho B}{2^{K-i}\epsilon}}\frac{(GR)^2 + B^2  
	}{\epsilon^2}\log^2\prn*{1+ 
		\frac{\rho B}{\epsilon^2}}
	\end{equation*}
	gradient computations, where we have used 
	$\lambda\pind{i}/\lambda\pind{i+1} \le 2$, $\lambda\pind{i} \le 
	\frac{B}{\rho}$, and $\lambda\pind{i+1} \ge \frac{\eps}{2\rho} 2^{K-i}$. 
	Summing over $i$ (and applying a union bound) gives the claimed 
	guarantee. 
	Since the minimizer of $f_\rho(x,\lambda)$ over $x\in \X$ and 
	$\lambda\in [\frac{\eps}{2\rho}, \frac{B}{\rho}]$ is equivalent is identical 
	to its minimizer in one of the intervals $[\lambda\pind{i+1}, 
	\lambda\pind{i}]$ for $i\le K$, the result follows from 
	Lemma~\ref{lem:cs-dual-range}.
\end{proof}

}
\notarxiv{%
\section{A doubling scheme for minimizing $\Lcs$}\label{sec:doubling}

In this section, we obtain a stronger guarantee for minimizing the
$\cs$-constraint robust objective. Namely, we leverage duality relationships to
approximate the constrained
objective\notarxiv{ $\Lcs$ }\arxiv{~\eqref{eqn:chi-square-def} }via its penalized
counterpart\arxiv{~\eqref{eqn:chi-square-reg-def}}, $\Llam$.  We adjust notation
to make the dependence of $\Llam^\lambda$ on $\lambda$ explicit.

Our starting point is the recognition that, by duality
(cf.~\cite[Sec.~3.2]{Shapiro17}),
\begin{equation*}
  \Lcs(x; P_0) = \inf_{\lambda \ge 0} \left\{ \Llam^\lambda(x; P_0)
  + \lambda \rho \right\}
  = \inf_{\lambda \ge 0} \sup_{Q \ll P_0}
  \left\{\E_Q \loss(x; S) - \lambda
  \left[\divcs(Q, P_0) - \rho \right] \right\}
\end{equation*}
for any distribution $P_0$. For $0 \le \lambdaL \le \lambdaH$, we
may thus consider the approximation
\begin{equation*}
  \LcsRange{\lambdaL, \lambdaH}(x;P_0) \defeq 
  \min_{\lambda\in [\lambdaL, \lambdaH]} 
  f_\rho(x,\lambda)
  ~~\mbox{where}~~
  f_\rho(x,\lambda) \defeq \Llam^\lambda(x;P_0) + \lambda \rho.
\end{equation*}
By restricting $\lambda$ to an appropriate range, we can then approximate 
$\Lcs$ by
its truncated version, as the next lemma 
shows.

\begin{restatable}{lemma}{restateCSDualRange}\label{lem:cs-dual-range}
  For all $P_0$, $\rho$ and $\epsilon$,
  \begin{equation*}
    \min_{x\in\X} \LcsRange{\frac{\epsilon}{2\rho}, \frac{B}{\rho}}(x;P_0)
    \le \min_{x'\in\X} \Lcs(x';P_0) + \frac{\epsilon}{2}.
  \end{equation*}
\end{restatable}

\begin{proof}
	Le
	$x\opt, \lambda\opt = \argmin_{x\in\X,\lambda \ge 0}  
	\crl*{f_\rho(x,\lambda)
	}$, noting that $\min_{x'\in\X} \Lcs(x';P_0) = 
	f_\rho(x\opt,\lambda\opt)$. 
	For any $x,\lambda$ let $Q\opt_{x,\lambda}$ be the maximizing $Q$ 
	in~\eqref{eq:pop-dro-reg-restated} for these values of $x,\lambda$. 
	Moreover, let $\div(x,\lambda) = \divcs(Q\opt_{x,\lambda}, P_0)$. By 
	Claim~\ref{claim:cs-bdd}, for all $\lambda > B/\rho$ we have that 
	$\div(x,\lambda) < \rho$, and consequently $\lambda > \lambda\opt$, 
	i.e., $\lambda\opt \le B/\rho$, and hence that upper bound has no 
	impact on accuracy.
	
	When in addition we have $\lambda\opt \ge \epsilon/(2\rho)$ then 
	clearly $\min_{x\in\X} \LcsRange{\frac{\epsilon}{2\rho}, 
		\frac{B}{\rho}}(x;P_0) = \LcsRange{0, \infty}(x;P_0)
	= \min_{x'\in\X} \Lcs(x';P_0)$. Otherwise, if $\lambda\opt < 
	\epsilon/(2\rho)\eqdef \lambda_\eps$ we may write
	\begin{flalign*}
	\min_{x\in\X} \LcsRange{\frac{\epsilon}{2\rho}, \frac{B}{\rho}}(x;P_0)
	&\stackrel{(i)}{\le} 
	f_\rho(x\opt, \lambda_\eps) \stackrel{(ii)}{\le}  
	f_\rho(x\opt,\lambda\opt)
	+ \brk*{\tfrac{\del}{\del\lambda}f_\rho(x\opt, \lambda_\eps)} 
	(\lambda_\eps 
	-\lambda\opt)
	\\&= \min_{x'\in\X} \Lcs(x';P_0) +\brk*{\rho - 
		\div(x\opt,\lambda_\eps)}(\lambda_\eps - \lambda\opt)
	\\&\stackrel{(iii)}{\le} 
	\min_{x'\in\X} \Lcs(x';P_0) + \lambda_\eps \rho = 
	\min_{x'\in\X} \Lcs(x';P_0) + \frac{\eps}{2}.
	\end{flalign*}
	Where we used $(i)$ that $x\opt$ and $\lambda_\eps$ are feasible 
	points in the joint minimization of $f_\rho(x,\lambda)$ over $x\in\X$ 
	and 
	$\lambda\in [\lambda_\eps, B/\rho]$; $(ii)$ the convexity of $f$ in 
	$\lambda$; and $(iii)$ the fact that $\div(x\opt,\lambda_\eps) \ge 0$ 
	and $\lambda\opt \le \lambda_\eps$.
      \end{proof}
      
Our strategy is therefore to jointly minimize
$f_\rho(x,\lambda)=\Llam^{\lambda}(x;P_0)+ \lambda \rho$ over both $x \in
\mc{X}$ and $\lambda \in [\lambdaL, \lambdaH]$ (rather than $[0, \infty]$),
using the approximation guarantee in Lemma~\ref{lem:cs-dual-range} to argue
that the restriction of $\lambda$ will have limited effect on the quality of
the resulting solution. We iterate the projected stochastic gradient method
with the multi-level Monte Carlo (MLMC) gradient estimator~\eqref{eqn:mlmc}
via \notarxiv{
  \begin{equation}
  \label{eq:x-lambda-sgd}
  x_{t+1} = \Pi_{\X}\prn*{x_{t} - \gamma_x \ml\brk[\big]{\grad 
      \Llam^{\lambda_t}(x_t)}},~
  \lambda_{t+1} = \Pi_{[\lambdaL,\lambdaH]}\prn*{\lambda_t - 
    \gamma_\lambda
    \ml\brk[\big]{\tfrac{\del}{\del \lambda} 
      \Llam^{\lambda_t}(x_t)+\rho}}.
\end{equation}
}
\arxiv{
  \begin{equation}
    \label{eq:x-lambda-sgd}
    \begin{split}
      x_{t+1} & = \Pi_{\X}\prn*{x_{t} - \gamma_x \ml\brk[\big]{\grad 
          \Llam^{\lambda_t}(x_t)}} \\
      \lambda_{t+1} & = \Pi_{[\lambdaL,\lambdaH]}\prn*{\lambda_t - 
        \gamma_\lambda
        \ml\brk[\big]{\tfrac{\del}{\del \lambda} 
          \Llam^{\lambda_t}(x_t)+\rho}}.
    \end{split}
  \end{equation}
}
If we can bound the moments of the MLMC-approximated gradients
$\ml$, we can then leverage standard stochastic gradient
analyses to prove convergence. We use the following
bound.
\begin{restatable}{lemma}{restateLemLambdaMom}\label{lem:ml-lambda-mom}
	We have 
	\begin{equation*}
	\E \prn*{\ml\brk[\big]{\tfrac{\del}{\del \lambda} 
			\Llam^{\lambda}(x;\cdot)+\rho}}^2 \lesssim 
	\frac{B^2}{\lambda^2}\prn*{
		1+\frac{B \log\frac{n}{n_0}}{\lambda n_0}} + \rho^2.
	\end{equation*}
\end{restatable}
\begin{proof}
	Recall the definition~\eqref{eq:ml-def-restate} of the MLMC estimator of  
	a general $\funct$ and the expression~\eqref{eq:ml-second-moment} 
	for its second moment. Suppose that $\funct(\cdot) = \funct_1(\cdot) + 
	\funct_2(\cdot) + c$, where $c$ is a constant. Then
	\begin{equation*}
	\E \norm{\mld_k \brk[\big]{ \funct}}^2 
	= \E \norm{\mld_k \brk[\big]{ \funct_1 +\funct_2}}^2 
	\le 2\E \norm{\mld_k \brk[\big]{ \funct_1 }}^2 + 2\E \norm{\mld_k 
		\brk[\big]{ \funct_2 }}^2.
	\end{equation*}
	Consequently, by~\eqref{eq:ml-second-moment}, we have
	\begin{equation}\label{eq:ml-sum-decomp}
	\E \norm{\ml \brk[\big]{ \funct}}^2 \le 2c^2 + 2\E \norm{\ml \brk[\big]{ 
			\funct_1}}^2 + 2 \E \norm{\ml \brk[\big]{ \funct_2}}^2.
	\end{equation}
	
	We apply this observation to $\ml\brk[\big]{\tfrac{\del}{\del \lambda} 
		\Llam^{\lambda}(x;\cdot)+\rho}$ by noting that
	\begin{equation*}
	\tfrac{\del}{\del \lambda} \Llam^{\lambda}(x;S_1^n) = -\divcs(q\opt; 
	\tfrac{1}{n}\ones) =  
	\frac{1}{\lambda}\prn[\Bigg]{\Llam^\lambda(x;S_1^n) 
		-\frac{1}{n}\sum_{i\le n} 
		q\opt_i \ell(x;S_i)}.
	\end{equation*}
	Proposition~\ref{prop:ml-mom-extended} gives us the bound $\E 
	\prn*{\ml \brk[\big]{ \Llam^\lambda(x;S_1^n) }}^2\lesssim B^2\prn*{1 + 
		\frac{B}{\lambda n_0} 
		\log\frac{n}{n_0}}$. Moreover, we have that
	\begin{equation*}
	\E \prn*{\ml \brk[\Big]{ \frac{1}{n}\sum_{i\le n} 
			q\opt_i \ell(x;S_i)}}^2 \lesssim B^2\prn*{1 + \frac{B}{\lambda n_0} 
		\log\frac{n}{n_0}}
	\end{equation*}
	By exactly the same argument that proves the gradient second moment 
	bound in Proposition~\ref{prop:ml-mom-extended}. The result then 
	follows by substituting into~\eqref{eq:ml-sum-decomp}.
\end{proof}

\noindent
Therefore, 
we may find an $\epsilon$ approximate minimizer with complexity roughly 
$B^3\lambdaH^2/(\lambdaL^3\epsilon^2)$:
\begin{restatable}{lemma}{restateLemXLamSGD}\label{lem:x-lambda-sgd}
  Fix $\epsilon\in (0,B)$ and $\lambdaH \ge \lambdaL > 0$. For a suitable 
  setting of the parameters $n_0,n,T,\gamma_x$ and $\gamma_\lambda$, 
  the average $\bar{x}_T = \sum_{t\le T} x_t$ 
  of the iterates~\eqref{eq:x-lambda-sgd} satisfies
  $\E \LcsRange{\lambdaL,\lambdaH}(\bar{x}_T;P_0) \le \min_{x\in\X}  
  \LcsRange{\lambdaL,\lambdaH}(x;P_0) + \epsilon$, with complexity 
  \begin{equation*}
    \lesssim 
     \prn*{1+\frac{B}{\lambdaL}}\frac{(GR)^2 + B^2 \lambdaH^2 / 
      \lambdaL^2 + 
      \lambdaH^2 \rho^2}{\epsilon^2}\log^2\prn*{1+ 
      \frac{B}{\lambdaL\epsilon}}~~\mbox{with probability 
    }\ge 1-\frac{\eps^2}{B^2}.
  \end{equation*}
\end{restatable}
\begin{proof}
	We take $n\asymp \frac{B}{\lambdaL \epsilon}$ to guarantee bias below 
	$\epsilon/2$ by Proposition~\ref{prop:batch-bias}, and we take $n_0 
	\asymp \frac{B}{\lambdaL}\log n$ to guarantee that 
	\begin{equation*}
	\Gamma_x^2 \defeq \sup_{x\in\X, \lambda\in[\lambdaL,\lambdaH]}
	\E \norm*{
		\ml\brk[\big]{\grad \Llam^{\lambda}(x; \cdot)}}^2 \lesssim G^2
	\end{equation*}
	and, by Lemma~\ref{lem:ml-lambda-mom},
	\begin{equation*}
	\Gamma_\lambda^2 \defeq \sup_{x\in\X, 
		\lambda\in[\lambdaL,\lambdaH]}
	\E \prn*{\ml\brk[\big]{\tfrac{\del}{\del \lambda} 
			\Llam^{\lambda}(x;\cdot)+\rho}}^2 \lesssim 
	\frac{B^2}{\lambdaL^2} + \rho^2.
	\end{equation*}
	Let $\bar{\lambda}_T = \sum_{t\le T} \lambda_t$ be the average of the 
	$\lambda$ iterates in~\eqref{eq:x-lambda-sgd}. 
	By appropriate choice of $\eta$ and $\eta'$ we guarantee (via 
	Proposition~\ref{prop:gangster}) that
	\begin{equation*}
	\E \LcsRange{\lambdaL,\lambdaH}(\bar{x}_T;P_0)   
	\le \E f_\rho(\bar{x}_T, \bar{\lambda}_T)
	\le \min_{x\in\X, \lambda\in[\lambdaL, \lambdaH]}f_\rho(x,\lambda) + 
	\mathrm{err}_T
	= \min_{x\in\X}\LcsRange{\lambdaL,\lambdaH}(x;P_0)  
	+ \mathrm{err}_T,
	\end{equation*} 
	where
	\begin{equation*}
	\mathrm{err}_T \lesssim  \frac{\eps}{2} + \frac{\Gamma_x R + 
		\Gamma_\lambda (\lambdaH - \lambdaL)}{\sqrt{T}}.
	\end{equation*}
	Therefore, by taking 
	\begin{equation*}
	T\asymp \frac{\Gamma_x^2 R^2 + 
		\Gamma_\lambda^2 \lambdaH^2}{\epsilon^2} \asymp\frac{(GR)^2 + 
		B^2 \lambdaH^2 / 
		\lambdaL^2 + 
		\lambdaH^2 \rho^2}{\epsilon^2} 
	\end{equation*}
	we guarantee that $\mathrm{err}_T 
	\le \eps$, and the complexity bound follows from substituting $n_0, 
	n$ and $T$ in the high probability upper bound  
	$n_0\log_2\prn*{\frac{n}{n_0}}T +
	5\sqrt{(n\log n)^2 + n_0n T \log n}$
	shown in Theorem~\ref{thm:ml}.
\end{proof}
Directly substituting $\lambdaL  = \frac{\epsilon}{2\rho}$ and 
$\lambdaH=\frac{B}{\rho}$ results in a guarantee scaling as 
$\epsilon^{-5}$, which is worse than the mini-batch rate of 
$\epsilon^{-4}$. To improve on this, we divide $[\frac{\epsilon}{2\rho}, 
\frac{B}{\rho}]$ into $K = \log_2 \frac{B}{\epsilon}$ sub-intervals 
$[\lambda\pind{i+1},\lambda\pind{i}]$ satisfying 
$\lambda\pind{i+1}/\lambda\pind{i} = 2$.
We then perform the stochastic gradient method~\eqref{eq:x-lambda-sgd}
on each of these intervals $[\lambda^{(i + 1)}, \lambda^{(i)}]$ in turn,
yielding estimates $\bar{x}^{(i)}$ that
are each $\lesssim\epsilon$-suboptimal for the approximate
objective $\LcsRange{\lambda\pind{i+1},\lambda\pind{i}}$.
Using the bounded ratio $\lambda\pind{i+1} / \lambda\pind{i} = 2$,
this requires complexity roughly
$1/(\lambda\pind{i+1} \epsilon^2) \lesssim \rho /\epsilon^3$,
giving the following theorem.
\begin{restatable}{theorem}{restateThmDoubling}\label{thm:doubling}
	Fix $\epsilon\in (0,B)$, and for $i\in\N$ set 
	$\lambda\pind{i}=\frac{B}{\rho}2^{-i+1}$ and let $\bar{x}\pind{i}$ be an 
	$\epsilon/2$-approximate minimizer of 
	$\LcsRange{\lambda\pind{i+1},\lambda\pind{i}}$ computed via 
	stochastic 
	gradient iterations according to Lemma~\ref{lem:x-lambda-sgd}. 
	Then, 
	for $1+ K=\ceil{\log_2 \frac{2B}{\epsilon}}$ and some $i\opt \le K$ we 
	have
	$\E \Lcs(\bar{x}\pind{i\opt}; P_0) \le \min_{x\in\X} \Lcs(x;P_0) + 
	\epsilon$. 
	Computing $\bar{x}\pind{1},\ldots, \bar{x}\pind{K}$ requires a total 
	number of $\grad \ell$ evaluations 
	\begin{equation*}
	  \lesssim \frac{(GR)^2 (\rho B + \epsilon \log_2 
		\frac{B}{\epsilon})}{\epsilon^3} \log^2 
	\prn*{1+\frac{\rho 
			B}{\epsilon^2}}~~\mbox{with probability 
	}\ge 1-\frac{\epsilon}{B}.
	\end{equation*}
\end{restatable}
\begin{proof}
	By Lemma~\ref{lem:x-lambda-sgd}, finding an $\epsilon$ approximate 
	solution in the interval $[\lambda\pind{i+1},\lambda\pind{i}]$ requires  
	\begin{equation*}
	\lesssim \prn*{1+\frac{\rho B}{2^{K-i}\epsilon}}\frac{(GR)^2 + B^2  
	}{\epsilon^2}\log^2\prn*{1+ 
		\frac{\rho B}{\epsilon^2}}
	\end{equation*}
	gradient computations, where we have used 
	$\lambda\pind{i}/\lambda\pind{i+1} \le 2$, $\lambda\pind{i} \le 
	\frac{B}{\rho}$, and $\lambda\pind{i+1} \ge \frac{\eps}{2\rho} 2^{K-i}$. 
	Summing over $i$ (and applying a union bound) gives the claimed 
	guarantee. 
	Since the minimizer of $f_\rho(x,\lambda)$ over $x\in \X$ and 
	$\lambda\in [\frac{\eps}{2\rho}, \frac{B}{\rho}]$ is equivalent is identical 
	to its minimizer in one of the intervals $[\lambda\pind{i+1}, 
	\lambda\pind{i}]$ for $i\le K$, the result follows from 
	Lemma~\ref{lem:cs-dual-range}.
\end{proof}
The index $i^\star$ is independent of randomness in our procedure, but we 
do not know it in advance. Instead, we may estimate the minimized objective for each $i$ and select the index 
 with the lowest estimate. Let $\hat{\lambda}\pind{i}$ be the average of the $\lambda$ iterations of our stochastic gradient method~\eqref{eq:x-lambda-sgd} for a particular interval $[\lambda\pind{i+1},\lambda\pind{i}]$. Our bias and variance bounds on $\Llam$ (Proposition~\ref{prop:batch-bias} and Proposition~\ref{prop:variance-extended} in the appendix) imply the we can estimate\footnote{
	To obtain an estimate that has error $\lesssim \epsilon$ with high probability, we can use the median of a logarithmic number of iid copies of the batch estimator.}  $f_\rho(\bar{x}\pind{i},\hat{\lambda}\pind{i})$ to accuracy $\lesssim \epsilon$ with a sample of size $\asymp B^2 / (\lambda\pind{i} \epsilon^2) \asymp  2^{i-K} B^3 \rho\epsilon^{-3}$.
Taking $i\opt$ to be the index $i$ minimizing this estimate, it is straightforward to argue that $\E \Lcs(\bar{x}\pind{i\opt}; P_0) - \min_{x\in\X} \Lcs(x;P_0) \lesssim 
\epsilon$. 
Therefore, the cost of selecting the best $i$ is at most the cost of performing the optimization.

Theorem~\ref{thm:doubling} provides a rigorous guarantee on the 
complexity of minimizing $\Lcs$ with a fixed constraint $\rho$ by 
optimizing the parameter $\lambda$ of $\Llam^\lambda$. In practice, we 
usually have no prior knowledge of $\rho$, so it will often make sense to 
directly tune $\lambda$ according to validation criteria rather than a target 
$\rho$.
We also note that~\citet{DuchiNa20} prove a lower bound of order
${\rho}{\epsilon^{-2}}$, which is smaller than our  
${\rho}{\epsilon^{-3}}$ rate. Establishing  the optimal rate for this 
problem 
remains an open question.

}
\section{Experiments}\label{app:experiments}

In this section we give a detailed description of our experiments. We begin with
a description of the problems we study 
(Section~\ref{app:experiments-setup}) followed by our hyperparameter 
settings (Section~\ref{app:experiments-tuning}) and brief remarks about 
our PyTorch implementation (Section~\ref{app:experiments-pytorch}). 
Then, in
Sections~\ref{app:experiments-results} and~\ref{app:experiments-discussion} we
present and discuss our results in detail, including speed-up factors over
full-batch optimization, a study of the generalization impacts of the DRO
objective, and direct empirical evaluation of the bias $\L -\bL$ which we 
bound in Proposition~\ref{prop:batch-bias}.

\subsection{Dataset description}\label{app:experiments-setup}

\paragraph{Digits.} We consider the MNIST handwritten digit recognition 
dataset with the standard train/test
split into  with $6\cdot 10^4$ and
$10^4$ training and test images, respectively. 
There are $10$ classes corresponding to the ten
digits. We augment the training set with $N_{\mathrm{typed}} = 600$ 
randomly chosen digits from the characters 
dataset~\cite{deCamposBaVa09}, i.e.,  $1\%$ of the 
hand-written digits. Our test set includes the MNIST test set as well as a 
class-balanced sample of
8K typed digits not included in the training data. Creating an 8K image test 
set requires that we disregard the original test/train split 
of~\cite{deCamposBaVa09}, but is important in order to make estimates of 
per-class accuracy reliable. To featurize our data, we train a small 
convolutional Neural Network (two convolutional layers, two 
fully-connected layers with ReLU activation function) with a standard ERM 
objective and 10 epochs of SGM on the MNIST training set (with no typed digits). For 
both handwritten and typed digits, we use the activations of the last layer 
as the feature vector. 

We perform DRO to learn a linear classifier $x$ on our features, taking the loss
$\ell$ to be multi-class logarithmic loss with a quadratic regularization term
on $x$ (the weight part only, not the bias), namely, for a data point
$s = (z, y)$ with $z \in \R^d, y\in[C]$ (with $C$ the number of classes) and
regularization strength $\mu \ge 0$, we use
\begin{equation*}
  \ell([x, b]; (z, y)) \defeq \log\prn*{\sum_{c=1}^C \exp(\tri{x_c - x_y, z} + b_c - b_y)}
  + \frac{\mu}{2}\sum_{c=1}^C\norm{x_c}_2^2,
\end{equation*}
where $x \in \R^{C\times d}, b\in\R^C$ and $x_c$ denotes the $c$-th row of
$x$. As the generalization metric, we report accuracy and log loss on the worst
sub-group of the data---where a sub-group corresponds to a tuple (subpopulation,
class), e.g., (typed, 9).

\paragraph{ImageNet.} The ImageNet dataset
comprises of $1.2\cdot 10^6$ training images and $5\cdot 10^{4}$ test 
images with 
$1000$ different classes. We featurize the dataset using a pre-trained 
ResNet-50~\cite{HeZhReSu16} (trained on ImageNet itself with an ERM 
objective). We use those features as the input to a linear classifier, with 
regularized multi-class logarithmic loss as in the previous experiment. As 
the robust generalization metric, we report 
the average loss and accuracy on the 10 classes with highest test loss. 

\subsection{Hyperparameter tuning}\label{app:experiments-tuning}

We fix the budget of our algorithms to $300$ epochs for
Digits and $30$ epochs for ImageNet, where an epoch corresponds to $N$ 
computations of $\grad\ell$, where $N$ is the training set size.
For all (mini)-batch methods we use Nesterov 
acceleration~\eqref{eq:agd-pytorch} with constant momentum $\omega = 
0.9$; we did not carefully tune this parameter but did observe it performs 
better than no momentum. For MLMC using no momentum  
($\omega=0$) performs slightly better than momentum 0.9, so we use no momentum in this case.
We also perform iterate averaging with the 
scheme
of~\citet{ShamirZh13} with parameter $3$ (roughly
averaging over the last third of the iterates).
Our experiments with CVaR use $\grad \Lcvar$ rather than $\grad 
\LcvarR$, in contrast to our theory; we leave empirical exploration of 
entropy smoothing for CVaR to future work.

\paragraph{Stepsizes.} We tune our stepsizes with a coarse-to-fine strategy.
More precisely, for each stepsize in $\crl{10^i}_{-5\le i \le 0}$, we perform a
single run of the experiment, and pick the best two
stepsizes in terms of the final training value. For these two stepsizes, we
evaluate $\frac{\eta}{2}, 2\eta$ and select the stepsize that
gives the best value of the training loss. For this final stepsize, we repeat
the experiments with $5$ different seeds (affecting weight initialization and
mini batch samples but not the dataset structure) and report the minimum and
maximum across seeds at each iteration. We select all the stepsizes in our
experiments using this strategy, except for batch size $n=10$ in 
ImageNet where we extrapolated the stepsize from other batch sizes. 
Table~\ref{table:step-sizes} summarizes our step size choices---for batch 
sizes up to 5K we see a clear linear relationship between the batch size and 
optimal step size.

\paragraph{{$\ell_2$}-regularization and parameters of the robust loss} We 
choose the strength of the regularizer in the set $\crl{0, 10^{-5}, \ldots, 
  10^{-1}}$.
For each robust loss, we consider an appropriate grid of either the size of the
uncertainty set ($\alpha$ and $\rho$) or the strength of the penalty
($\lambda$). We evaluate each configuration ($\ell_2$ regularization and robust
loss parameters) with the stepsizes from the coarse grid and pick the
configuration that achieves a good trade-off in terms worst-subgroup and
average-case generalization. For simplicity, we choose the same regularization
strenght for all the robust losses---$\mu = 10^{-3}$ for ImageNet and
$\mu = 10^{-2}$ for Digits. For ERM, we choose the two values of $\ell_2$
regularization that optimize either worst subgroup loss or worst subgroup
accuracy. That is, for ImageNet we tune the $\ell_2$ regularization for the best
result on either the worst 10 classes loss and worst 10 classes accuracy
respectively, and for Digits we choose the values that optimize loss/accuracy on
the hardest typed class---for both experiments, this results in
$\mu \in \crl{10^{-4}, 10^{-3}}$ for ERM.

\newcommand{\lrf}[2]{$#1\cdot 10^{-#2}$}

\begin{table}
  \begin{center}
    \begin{tabular}{clcccccc}
      \toprule
      & & \multicolumn{3}{c}{ImageNet}
      & \multicolumn{3}{c}{Digits} \\
      \cmidrule(lr){3-5}
      \cmidrule(lr){6-8}
      \multicolumn{2}{c}{}
      & $\Lcvar$%
      & $\Lcs$%
      & $\Llam$%
      & $\Lcvar$%
      & $\Lcs$%
      & $\Llam$\\%
      \multicolumn{2}{c}{Algorithm}
      & $\alpha = 0.1$
      & $\rho=1$
      & $\lambda = 0.4$
      & $\alpha=0.02$
      & $\rho=1$
      & $\lambda=0.05$ \\
      \midrule
      Batch
      & $n=10$ & $1\cdot 10^{-4}$ & $2\cdot 10^{-4}$ & $2\cdot 10^{-4}$ 
      & \lrf{1}{4} & \lrf{5}{5} & \lrf{5}{5}\\
      & $n=50$ & $5\cdot 10^{-4}$ & \lrf{1}{3} & \lrf{1}{3} & \lrf{1}{4} & 
      \lrf{2}{4} & \lrf{1}{4}\\
      & $n=500$ & \lrf{5}{3} & \lrf{1}{2} & \lrf{1}{2} & \lrf{1}{3} & \lrf{2}{3} & \lrf{1}{3} \\
      & $n=5K$ & \lrf{5}{2} & \lrf{1}{1} & \lrf{1}{1} & \lrf{5}{3} & \lrf{2}{2} & \lrf{1}{2}\\
      & $n=50K$ & \lrf{2}{1} & \lrf{5}{1} & \lrf{2}{1} & -- & -- & --\\
      & $n=150K$ &  \lrf{5}{1} & \lrf{5}{1} & \lrf{5}{1} & -- & -- & --\\                    
      \midrule
      MLMC
      & $n_0 = 10$ & \lrf{1}{3}%
       & \lrf{2}{3}%
       & \lrf{2}{3}%
      & \lrf{5}{4}%
      & \lrf{5}{4}%
      & \lrf{5}{4}%
      \\
      \midrule
      \multicolumn{2}{c}{Full-batch} & \lrf{5}{1} & \lrf{5}{1} & \lrf{5}{1} & \lrf{1}{2} & \lrf{2}{2} & \lrf{1}{2} \\
      \bottomrule
    \end{tabular}
    \caption{Stepsizes for the experiments we present in this work. We use 
    momentum 0.9 for all configurations except MLMC, where we do not use 
    momentum. We
      select the stepsizes according to the `coarse-to-fine' strategy we
      describe in this section.
    }\label{table:step-sizes}
  \end{center}
\end{table}

\subsection{PyTorch Integration}\label{app:experiments-pytorch}
Figure~\ref{fig:pytorch} illustrates our integration of DRO into PyTorch. 
Users simply define the robust loss they wish to use (in the
example $\Lcs$ with $\rho=1$) and feed the loss for the examples in the batch to
the robust layer.
While our current implementation only supports the robust objectives we 
analyze---namely, CVaR,
KL-regularized CVaR, constrained-$\cs$ and penalized-$\cs$---it is easy 
to extend to other choices of $\phi$ and $\psi$.

\begin{figure}
	\begin{center}
		\includegraphics[scale=0.8]{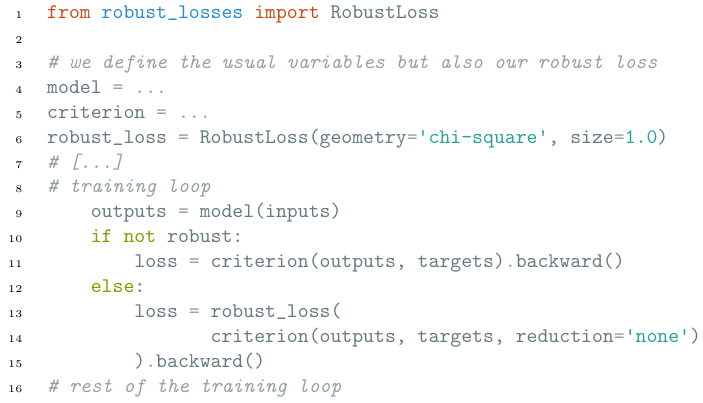}
		\caption{\label{fig:pytorch}An example training loop in PyTorch where 
		one 
			can decide to use
			the robust training objective at the cost of three extra lines of code 
			(lines
			1, 6 and 13).}
	\end{center}
\end{figure}

\subsection{Experiment results}\label{app:experiments-results}
We complement the training curves in Figure~\ref{fig:main-experiment} 
with comparisons of robust generalization 
metrics and training efficiency. 
In Figures~\ref{fig:gen-digits} and \ref{fig:gen-imagenet} we show the 
training curves of Figure~\ref{fig:main-experiment} along with two 
``robust'' generalization metrics and two ``average'' performance metrics.
For Digits, we consider the loss and accuracy on the worst 
sub-group---typically the typed digit 9---as the robust generalization 
metrics. For ImageNet, we look at the average loss (resp. accuracy) on the 
10 labels with highest loss (resp. lowest accuracy). In each figure we also 
show the values achieved by ERM with two different regularization strengths 
chosen to optimize either loss or accuracy on the worst-subgroup.
In Tables~\ref{tab:speed-up-digits} and~\ref{tab:speed-up-imagenet} we 
compare the number of epochs the various algorithms require to reach a 
training loss within 2\% of the minimal value found across all runs. To 
achieve such convergence with the full batch method we run it for much 
longer: 30K epochs for Digits and 1K epochs for ImageNet.

\begin{figure}
  \begin{center} 
  \includegraphics[width=\linewidth]{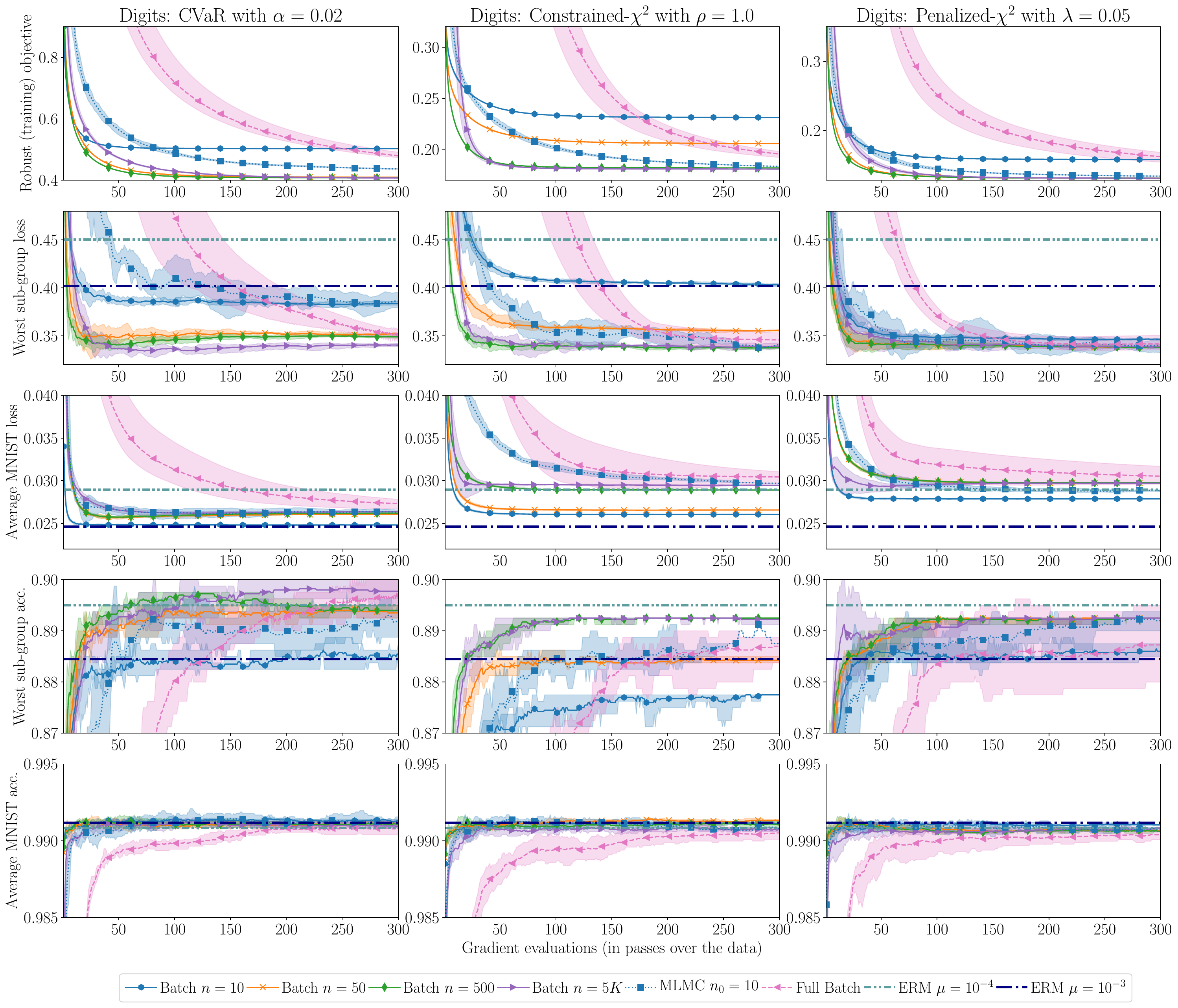}
    \caption{Detailed results from our digit recognition experiment. Shaded 
    areas indicate range of variability
      across 5 repetitions (minimum to maximum), and the zoomed-in 
      regions
      highlight the (often very small) ``bias floor'' of small batch sizes.}
\label{fig:gen-digits}
\end{center}
\end{figure}

\begin{figure}
  \begin{center} 
  \includegraphics[width=\linewidth]{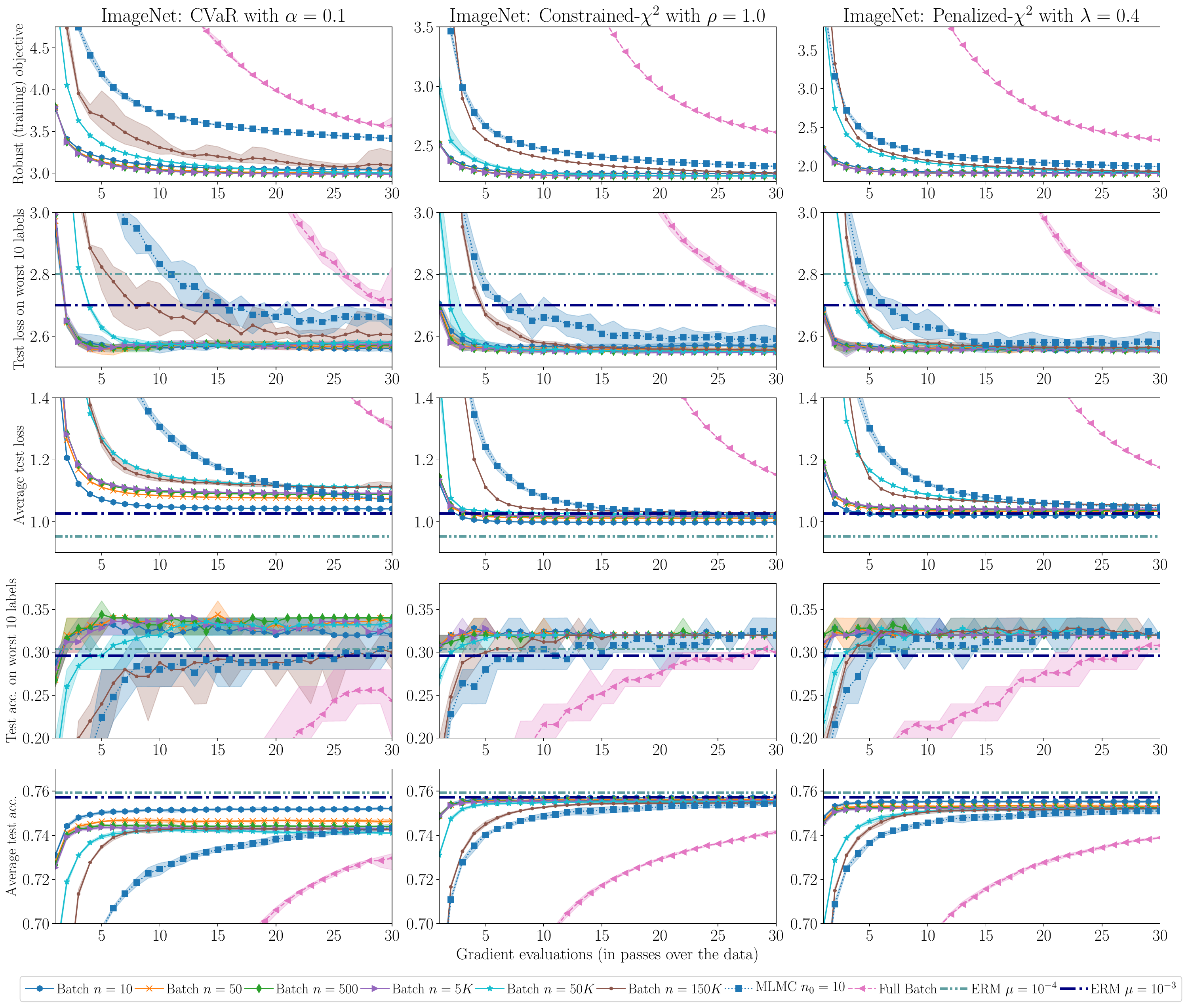}
    \caption{Detailed results from our ImageNet classification experiment. 
    Shaded areas indicate range of variability
      across 5 repetitions (minimum to maximum), and the zoomed-in 
      regions
      highlight the (often very small) ``bias floor'' of small batch sizes.}
\label{fig:gen-imagenet}
\end{center}
\end{figure}

\subsection{Discussion}\label{app:experiments-discussion}

\subsubsection{Generalization performance}

We now take a closer look at the curves presented in
Figures~\ref{fig:gen-digits} and \ref{fig:gen-imagenet}. We first note that, in
the context of machine learning, one does not wish to reach the minimum of the
training objective but rather find a model that achieves good generalization
performance. From that perspective, we observe that mini-batch methods 
achieve
their best generalization performance in a shorter time than necessary to
converge on the training objective, e.g., less than 50 epochs for CVaR
on Digits when the training objective always requires more than 115 epochs.

In the case of Digits, we observe that DRO achieves a better trade-off than 
ERM in all settings. More precisely, DRO achieves better worst sub-group 
loss and accuracy than either of the ERM runs with no visible degradation in 
average accuracy and slightly worse average loss. We observe a similar 
trend in the case of ImageNet, albeit with a more visible degradation in 
average loss and accuracy. 

We note that in the Digits experiment batch size $n=10$ has 
generalization performance more similar to ERM. This is an expected 
by-product of the bias inherent in small batch size, as in the edge case  
$n=1$, the mini-batch method degenerates to ERM.

\citet{HuNiSaSu18} observe that applying DRO objectives of the 
form~\eqref{eq:pop-dro-reg-restated} directly on the 0-1 loss amounts to 
a simple monotonic transformation of the average accuracy, and is 
therefore equivalent to minimizing average accuracy. Thus, in as far as the 
logarithmic 
loss is a surrogate to the 0-1 loss (which is arguably the case in near 
realizable-settings), DRO might not provide improvements in robust 
accuracy. This is consistent with the observations in our experiments, 
where we see only small effects on the accuracy in the Digits experiments 
(which is close to realizable), and a somewhat more pronounced but still 
modest effect on ImageNet (which is not quite realizable, as the training 
accuracy is below 90\%). Nevertheless, these observation do not preclude 
DRO from improvement the subpopulation test loss itself, as we see in our 
experiments:  for Digits DRO provides between between 17.5\% and 27\% 
reduction in worst subgroup loss compared to ERM, and for ImageNet the 
reduction is a more modest 5.6\% and 9\%. 
While the common practice in machine learning is to view accuracy as the 
more important performance metric, logarithmic loss is also operationally 
meaningful, as it measures the calibration of the model predictions. Thus, 
DRO is potentially helpful in situations where robust precise uncertainty 
estimates are important.

We remark that approaches that explicitly target the subgroups on which 
we measure the generalization~\cite[e.g.,][]{SagawaKoHaLi20} will likely 
perform better than DRO. However, in contrast to these methods DRO is   
agnostic to the subgroup definition---except that we use a subgroup  
validation set in order to tune its uncertainty set size---and therefore 
requires less data annotation.

\subsubsection{Optimization performance}

\begin{table}
	\begin{center}
		\begin{tabular}{lcccc|c}
			\toprule
			&\multicolumn{4}{c|}{Number of epochs to 2\% of opt} & Speed-up 
			\\
			& $n=50$
			& $n=500$
			& $n=5K$
			& Full-batch
			& vs.\ full-batch\\
			\midrule
			$\Lcvar, \alpha=0.02$
			& $189\pm 3$
			& $\boldsymbol{115}\pm 1$
			& $193 \pm 4$
			& $1035$
			& $9.0\times$\\
			$\Lcs, \rho = 1$
			& $\infty$
			& $74\pm 1$
			& $\boldsymbol{60} \pm 3$
			& $570$
			& $9.5\times$ \\
			$\Llam, \lambda=0.05$
			& $107\pm 1$
			& $\boldsymbol{104} \pm 1$
			& $131 \pm 5$
			& $1680$
			& $16.2\times$ \\
			\bottomrule
		\end{tabular}
		\caption{Empirical complexity for the Digits experiment in terms of 
			number of epochs required to reach within 2\% of
			the optimal training objective value, averaged across 5 seeds $\pm$ 
			one 
			standard deviation. (For the full-batch experiments we only ran one 
			seed).  The ``speed-up'' column gives the ratio between 
			the full batch complexity and the best mini-batch complexity.}
		\label{tab:speed-up-digits}
	\end{center}
\end{table}

\begin{table}
	\begin{center}
		\begin{tabular}{lccccccc|c}
			\toprule
			&\multicolumn{7}{c|}{Number of epochs to 2\% of opt} & Speed-up 
			\\
			$n=$
			& $10$
			& $50$
			& $500$
			& $5K$
			& $50K$
			& $150K$ 
			& Full-batch
			& vs.\ full-batch\\
			\midrule
			$\Lcvar, \alpha=0.1$
			& $20$
			& $10$
			& $\boldsymbol{9}$
			& $\boldsymbol{9}$
			& $19$
			& -
			& $245$ %
			& $27\times$ \\
			$\Lcs, \rho = 1$
			& $6$
			& $\boldsymbol{5}$
			& $\boldsymbol{5}$
			& $\boldsymbol{5}$
			& $8\pm 1$ 
			& $23$
			& $160$ %
			& $32\times$ \\
			$\Llam, \lambda=0.4$
			& $7$
			& $\boldsymbol{5}$
			& $\boldsymbol{5}$
			& $\boldsymbol{5}$
			& $22$
			& $26$
			& $180$ %
			& $36\times$ \\
			\bottomrule
		\end{tabular}
		\caption{Empirical complexity for the ImageNet experiment in terms 
		of 
			number of epochs required to reach within 2\% of
			the optimal training objective value, averaged across 5 seeds $\pm$ 
			one 
			standard deviation, whenever it is not zero. (For the full-batch 
			experiments we only ran one seed).  The ``speed-up'' column gives 
			the ratio between 
			the full batch complexity and the best mini-batch 
			complexity.}\label{tab:speed-up-imagenet}
	\end{center}
\end{table}

As Figure~\ref{fig:main-experiment} and Tables~\ref{tab:speed-up-digits} 
and~\ref{tab:speed-up-imagenet} indicate,  mini-batch methods converge 
significantly faster than full-batch. We also see that, 
while theoretically optimal, MLMC methods are slower to converge. 
Furthermore, the bias is empirically much 
smaller than what the theory predicts and setting the batch size as small as 
50 guarantees negligible bias; we investigate this further below. As the 
theory predicts, the MLMC method (for 
corresponding values of $n_0$) effectively counteracts this bias, and is able 
to converge to the optimal value even when $n_0$ is 10. 

We also note that the effect of batch size on the depth of the algorithm 
(number of iterations) is remarkably consistent with the theoretical 
prediction of the variance-based analysis in Section~\ref{sec:batch}: for 
smaller batch sizes the number of steps is roughly inversely proportional to 
the batch size, and the total amount of work is constant. The best stepsize 
also grows linearly with the batch size (see Table~\ref{table:step-sizes}). As 
batch sizes grow, the best 
stepsize plateaus and the number of steps required for convergence 
also stops decreasing with the batch size, making the total work become 
larger.

\notarxiv{%
\paragraph{Runtime comparison.}
In Table~\ref{tab:compute} we report the gradient complexity and 
wallclock time to reach
accuracy within 2\% of the optimal value. For brevity, we show it
for a single robust objective (penalized-$\cs$), but we observe that
similar results across robust objectives.
We note that for small batch sizes the time per epoch is significantly larger 
than for larger batch sizes, this due in part to parallelization in evaluating 
$\ell$ and $\grad\ell$ and in part to logging and Python interpreter 
overhead, which increase linearly with the number of iterations. However, 
these effects diminish as the batch size grows, and for batch size 5K the 
wallclock time to reach an accurate solution is an order of magnitude 
smaller than with the full-batch method.
We run
our experiments with 4 Intel Xeon E5-2699 CPUs and 12--32Gb
of memory.
Increasing the number of CPUs or using GPUs would allow for greater 
parallelism and improve the runtime at greater batch sizes. However, 
increasing the model complexity (e.g., to a deep neural network) would 
have the opposite effect. Using 4 CPUs for linear classification gives 
roughly the same range of feasible batch sizes as a ResNet-50 on large GPU 
arrays.

\begin{table}
	\begin{center}
		\smaller
		\begin{tabular}{clcccccc}  
			\toprule
			& & \multicolumn{3}{c}{ImageNet times [minutes]} & 
			\multicolumn{3}{c}{Digits times [minutes]} \\
			\cmidrule(lr){3-5}
			\cmidrule(lr){6-8}
			\multicolumn{2}{c}{Algorithm} 
			& per epoch & to 2\% of opt &\# epochs  
			& per epoch & to 2\% of opt &\# epochs \\
			\midrule
			Batch
			& $n=10$ & $120 \pm 5$ & $850 \pm 30$ & 7 & $0.80 \pm 0.1$ & 
			$\infty$  & $\infty$ \\
			& $n=50$ & $23 \pm 0.7$ & $116 \pm 4$ & $\boldsymbol{5}$ & 
			$0.23 
			\pm 0.01$ & $24 \pm 1$  & $107\pm 1$ \\
			& $n=500$ & $5.9 \pm 0.2$ & $29 \pm 1$ & $\boldsymbol{5}$ & 
			$0.056 \pm 0.004$ & $5.8 \pm 0.4$  & $\boldsymbol{104} \pm 
			1$\\
			& $n=5K$ & $3.3 \pm 0.04$ & $\boldsymbol{16.5} \pm 0.2$ & 
			$\boldsymbol{5}$ & $0.033 \pm 0.004$ & 
			$\boldsymbol{4.4} \pm 0.7$  & $131 \pm 6$\\
			& $n=50K$ & $2.2 \pm 0.03$ & $50 \pm 0.9$ & $22$ & -- & --  & 
			--\\
			& $n=150K$ & $2.1 \pm 0.03$ & $55 \pm 0.7$ & $26$ & -- & -- & 
			--\\
			\midrule
			MLMC
			& $n_0 = 10$ & $16\pm 1$ & $\infty$ & $\infty$ & $0.34 \pm 
			0.02$ & 
			$\infty$ & $\infty$ \\
			\midrule
			\multicolumn{2}{c}{Full-batch} & $2.1$ & $380$ & $180$ & $0.022$ 
			& 
			$37.0$ & $1680$ \\
			\bottomrule
		\end{tabular}
		\caption{Comparison wallclock time (in minutes) of the different
			algorithms, in terms of time per epoch and time to reach within 2\% 
			of the
			best training loss. In the last two columns, we report the number of
			epochs required to reach within 2\% of the best training loss. We 
			report 
			$\infty$ for configurations that do not reach the 
			sub-optimality 
			goal for the duration of the experiment, and 
			omit standard deviations when then they are $0$.}
		\label{tab:compute}
	\end{center}
\end{table} %
}

\begin{figure}
	\begin{center}
		 \includegraphics[width=0.48\linewidth]{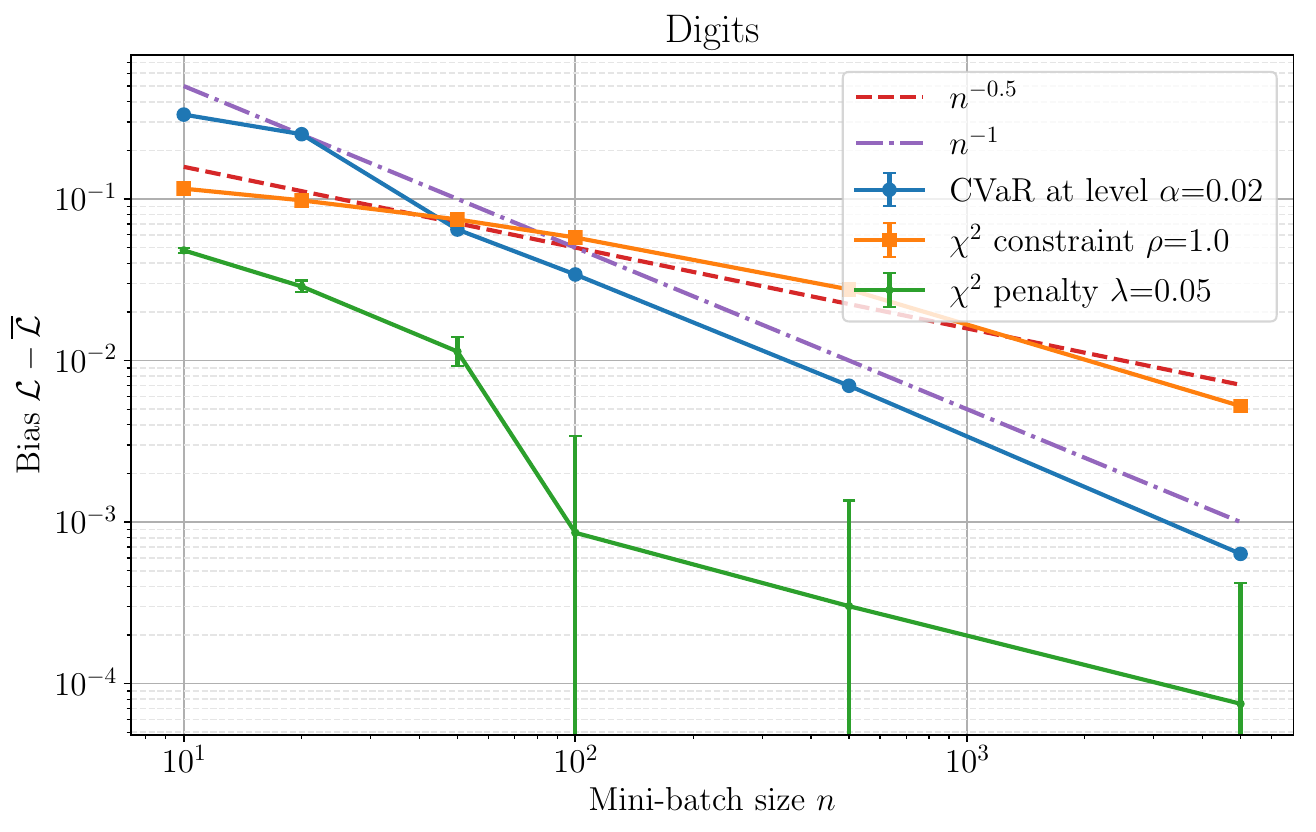}
		 ~~
		 \includegraphics[width=0.48\linewidth]{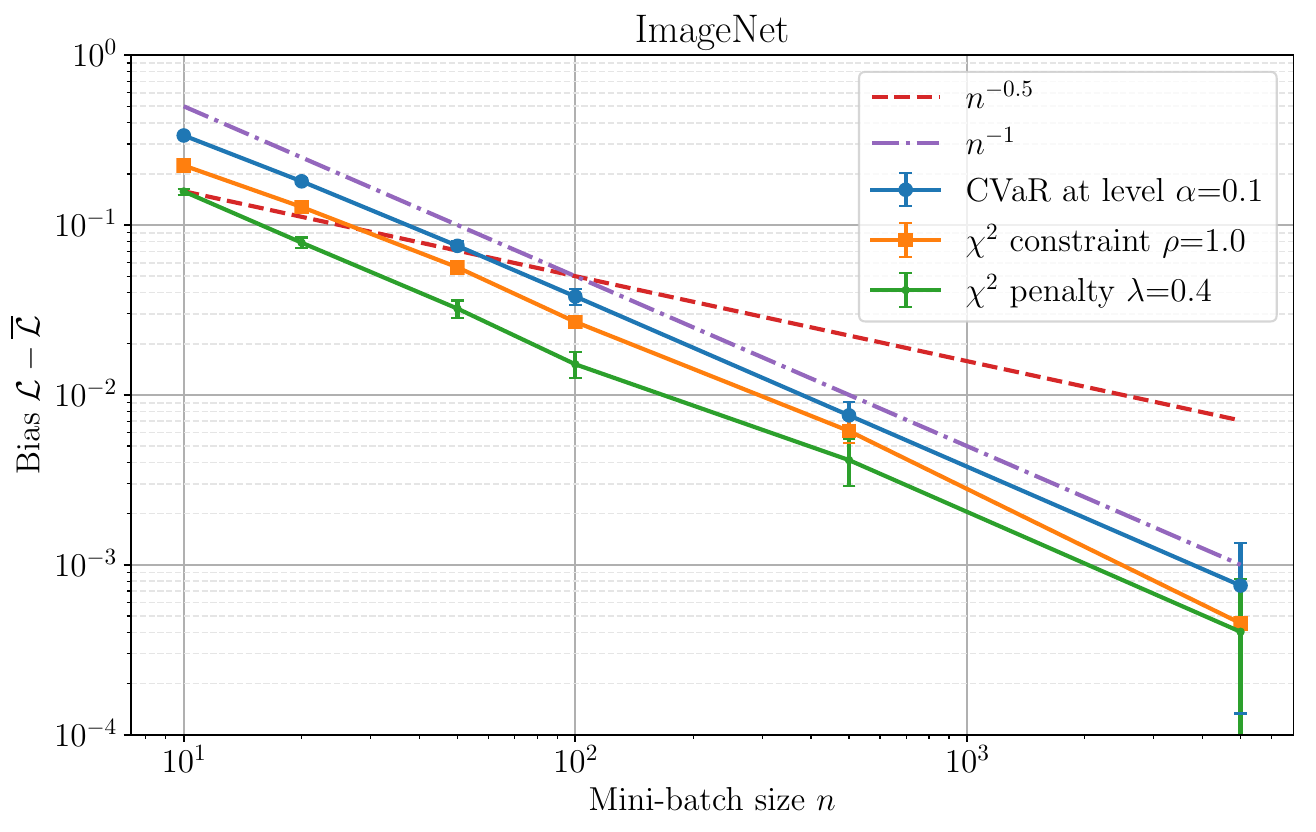}
		\caption{Evaluation of the bias $\L(\bar{x}_T;P_0) - \bL(\bar{x}_T;n)$ 
		at the last iterate $\bar{x}_T$ of the experiments in 
		Figure~\ref{fig:main-experiment}, for different batch sizes $n$. (These 
		batch sizes $n$ are not the same as the mini-batch size used to 
		compute $\bar{x}_T$; we take the latter to be 10). Error bars indicate a 95\% confidence interval 
		computed 
		using the bootstrap.}
		\label{fig:bias}
	\end{center}
\end{figure}

\paragraph{Bias analysis.}
Figure~\ref{fig:main-experiment} shows that even for small batch 
sizes---where the guarantees of Proposition~\ref{prop:batch-bias} are 
essentially vacuous---stochastic gradient steps with the mini-batch 
gradient estimator find solutions very close to optimal. There could be two 
explanations for this finding: (a) $\L$ and $\bL$ are actually much closer to 
each other than the theory predicts, or (b) $\L$ and $\bL$ are far apart as 
expected, but still their minimizers are close. 

To test hypothesis 
(a), we examine the loss values at the last iterate $\bar{x}_T$ of our Digits 
and ImageNet experiments with mini-batch size 10. For each objective, 
we estimate $\bL(\bar{x}_T;n)$ for various values of $n$ by averaging 50K 
evaluations of $\L(\bar{x}_T; S_1^n)$, and use it to compute an estimate of 
the bias $\L(\bar{x}_T;P_0) - \bL(\bar{x}_T;n)$.\footnote{
	For CVaR it is fact possible to compute $\bL(\bar{x}_T;n)$ in closed form 
	via~\eqref{eq:cvar-sample-closed-form}; we do that for the Digits 
	experiment. Scaling the computation to ImageNet is nontrivial, so there 
	we use an empirical estimate instead.}
In Figure~\ref{fig:bias} we plot the bias estimate against the mini-batch 
size $n$. We see that hypothesis (a) is false: for both ImageNet and Digits, 
the difference $\L(\bar{x}_T;P_0) - \bL(\bar{x}_T;n)$ is quite large at small 
$n$, as our upper bounds and matching lower bounds in the 
Bernoulli case would suggest. We also see that the bias decays as $1/n$ in 
all cases except for $\cs$ constraint in Digits; this is again consistent with 
our theory as we expect the inverse-cdf assumption to be relevant in 
practice and particularly for CVaR where it only needs to hold around the 
$1-\alpha$ quantile. 
We conclude that despite the significant bias at small batch size $n$, 
approximate minimizers of $\bL(x;n)$ are also approximate minimizers of 
$\L(x;P_0)$. This is possibly due to the fact that the bias $\L(x;P_0) - 
\bL(x;n)$ is nearly constant as a function of $x$. We leave further study of 
this hypothesis to future work.

\subsection{Comparison with alternative optimization methods}\label{app:experiments-baselines}

We complement the worst-case complexity comparison in Table~\ref{table:summary} by repeating our experiments with two alternative optimization methods: dual SGM and primal-dual methods.

\subsubsection{Comparison with dual SGM}
\paragraph{Experiment description.} Recall the dual SGM method we describe and
analyze in Section~\ref{app:dual-sgm}. The complexity guarantees of dual SGM depend quadratically on the size
of the uncertainty set---scaling with $\alpha^{-2}$ for CVaR and with 
$\lambda^{-2}$ for the penalized version of the $\cs$ objective. In
contrast, our theory predicts that the method we propose have an optimal
linear dependence on the size of the uncertainty set. Here we
empirically test this prediction on the Digits experiment. To do so, we compare the
performance of our proposed mini-batch method with dual SGM for uncertainty sets of
increasing size. For CVaR we consider
\begin{equation*}
  \alpha \in \crl{0.02, 0.006, 0.002, 0.0006, 0.0002},
\end{equation*}
and for penalized $\cs$ we consider
\begin{equation*}
\lambda \in \crl{0.05, 0.015, 0.005, 0.0015, 0.0005}.
\end{equation*}

\paragraph{Parameter tuning.}
 For each uncertainty set size, we jointly tune
the stepsizes $\gamma_x$ and $\gamma_\eta$ over the following grids
\begin{equation*}
  \gamma_x \in \crl{1 \cdot 10^{-i}, 3 \cdot 10^{-i}}_{3\le i \le 6}, \, \gamma_\eta \in
  \crl{1 \cdot 10^{-i}}_{2\le i\le 5}.
\end{equation*}
We choose a coarser grid for $\gamma_\eta$ as we noticed that the value of
$\gamma_\eta$ had a marginal influence on the final performance. For both the mini-batch algorithm and dual SGM, we pick the batch size $n = 500$ We follow the same
averaging scheme and momentum as in our previous experiments.

\begin{figure}
  \begin{center}
    \includegraphics[width=\linewidth]{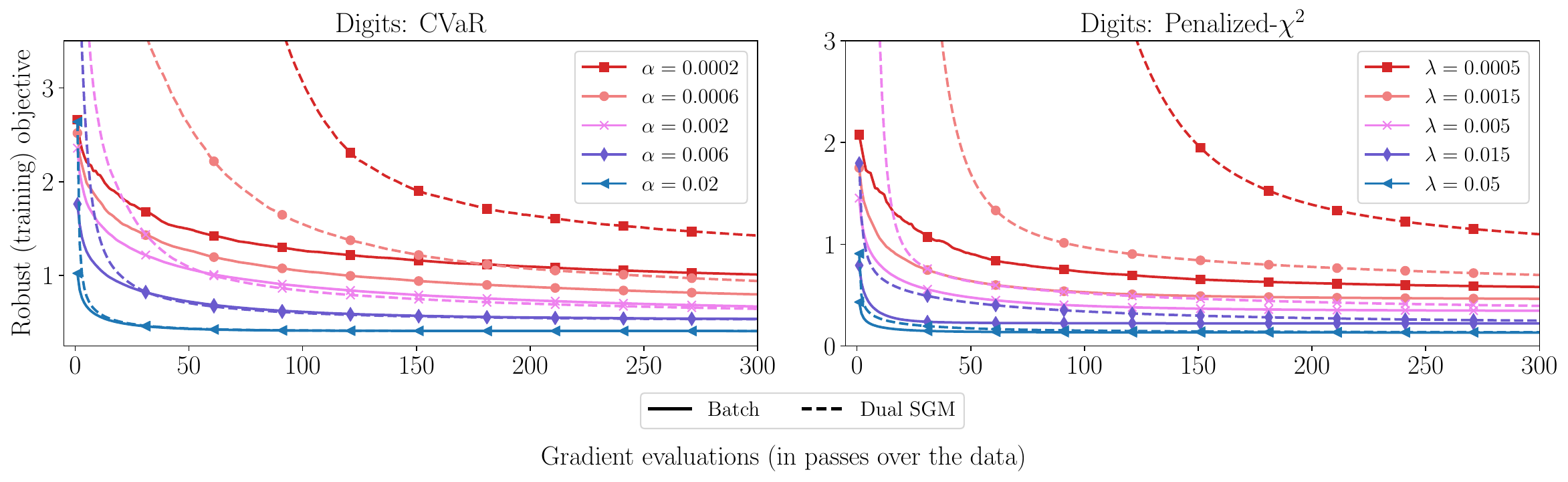}
    \caption{Comparison of batch methods to dual SGM on the digits experiments
      for increasing sizes of uncertainty set sizes or regularization. We
      observe that as the size grows, dual SGM performs increasingly worse.}
    \label{fig:dual-sgm}
  \end{center}
\end{figure}

\paragraph{Discussion of results.}
We plot the results of the experiment in Figure~\ref{fig:dual-sgm}. As the theory predicts, when the size of the
uncertainty set grows, dual SGM performs significantly worse than batch
methods. Conversely, as expected, for small uncertainty sets dual SGM performs on par with the mini-batch
method. We empirically observe that the performance of dual SGM depends only weakly on the choice of $\gamma_\eta$. As a result, dual SGM is not much more difficult to tune than the mini-batch method.

\subsubsection{Comparison with primal-dual methods}
\paragraph{Experiment description.} 
We now turn to primal-dual methods, whose complexity guarantees scale as $\epsilon^{-2}$ but are linear in $N$, and are therefore expected to become less efficient as the size of the training set grows. To test this prediction, we repeat our Digits and ImageNet experiments (with $N=60.6$K and $N=1.2$M, respectively) using these alternative methods for the contrained-$\cs$ and CVaR objectives. We then compare their performance to that of  gradient methods with our mini-batch estimator.

\paragraph{Method description.}
Primal-dual methods maintain an iterate sequence $\{x_t,q_t\}_{t\in\N}$, where $q_t \in \uset(P_0)\subset \Delta^N$ represent an online estimate of the distribution $q$ attaining the maximum in~\eqref{eq:finite-rob} at $x_1,\ldots, x_t$. 
To compute $x_{t+1},q_{t+1}$, we sample a batch of $n$ indices $J_1^n$ drawn independently from $q_{t}$, and (denoting $S_i = s_{J_i}$) estimate the gradient of $\sum_{i=1}^N q_i \ell(x;s_i)$ with respect to $x$ and $q$ as follows:
\begin{equation*}
\tilde{g}^{x}_t = \frac{1}{n}\sum_{i=1}^n \grad \ell(x_t; S_i)
~~\mbox{and}~~
[\tilde{g}^{q}_t]_j = \frac{1}{n}\sum_{i=1}^n \frac{1} {q_{J_i}}\ell(x_t; S_i) \indic{J_i = j}.
\end{equation*}
To compute $x_{t+1}$ from $x_t$ and $\tilde{g}^{x}_t$ we apply the same stochastic gradient scheme we use in our previous experiments (Nesterov momentum 0.9).\footnote{
	The performance of the primal-dual method appears fairly insensitive to the use of momentum so we keep the parameter the same as in our previous experiments for simplicity.}
 We also use the averaging scheme in~\cite{ShamirZh13} with parameter 3 as before.
To compute $q_{t+1}$  from $q_t$ and $\tilde{g}^{q}_t$ we apply a mirror descent step. For the constrained $\cs$ problem the step is of the form
\begin{equation}\label{eq:pd-step-cs}
q_{t+1} =
\argmax_{q: \divcs(q,\frac{1}{N}\ones) \le \rho}
	\crl*{ \inner{q}{\gamma_q \tilde{g}^{q}_t}
		+ \half \norm{q-q_t}^2}=
\argmin_{q\in\Delta^N: \norm{q-\frac{1}{N}\ones}^2 \le 2\rho/n}
	\norm{ q - (q_t + \gamma_q \tilde{g}^{q}_t)}^2,
\end{equation}
i.e., a Euclidean projection of the unconstrained gradient step on $q_t$ to the uncertainty set. For the CVaR problem, the step is of the form
\begin{equation}\label{eq:pd-step-cvar}
q_{t+1} =
\argmax_{q\in\Delta^N: \linf{q} \le \frac{1}{\alpha N}}
\crl*{ \inner{q}{\mathrm{clip}(\gamma_q \tilde{g}^{q}_t)}
	+ \divkl(q, q_t)},
\end{equation}
where $\mathrm{clip}(x)$ is the Euclidean projection of $x$ to $[-1,1]^N$. 

The $\cs$ step is essentially the same as in~\cite{NamkoongDu16}, while the CvaR step is different from the proposal by~\citet{CuriLeJeKr19}. Nevertheless, local norms regret analysis~\cite{ClarksonHaWo12,Shalev12} readily shows that with appropriate $\gamma_x$ and $\gamma_q$ the step~\eqref{eq:pd-step-cvar} allows us to find $\epsilon$-optimal solutions within $\lesssim \frac{N\log\frac{1}{\alpha}B^2 + G^2 R^2}{\epsilon^2}$ iterations, similarly to the guarantee that~\citet{CuriLeJeKr19} show for a computationally intractable determinantal point process scheme. They also propose a tractable approximation for this scheme, but do not prove that it converges to the solution of the CVaR problem.

\paragraph{Parameter tuning.}
For every training task we jointly tune the parameters $\gamma_x$ and $\gamma_q$. We tune $\gamma_x$ over the values $10^{-i}$, $2\cdot10^{-i}$ and $5\cdot 10^{-i}$ for $i\ge 1$ (similarly to our previous experiments) and we tune $\gamma_q$ over the values $10^{-i}$ and $3\cdot 10^{-i}$ for $i\ge 1$. The best-performing values of $(\gamma_x, \gamma_q)$ are $(0.02, 0.003)$ for Digits/CVaR; $(0.02, 3\cdot 10^{-7})$ for Digits/$\cs$; $(0.05, 3\cdot 10^{-5})$ for ImageNet/CVaR; and $(0.02, 3\cdot 10^{-11})$ for ImageNet/$\cs$. 
We use batch size $n=500$ throughout. 

\begin{figure}
	\begin{center}
		\includegraphics[width=0.9\linewidth]{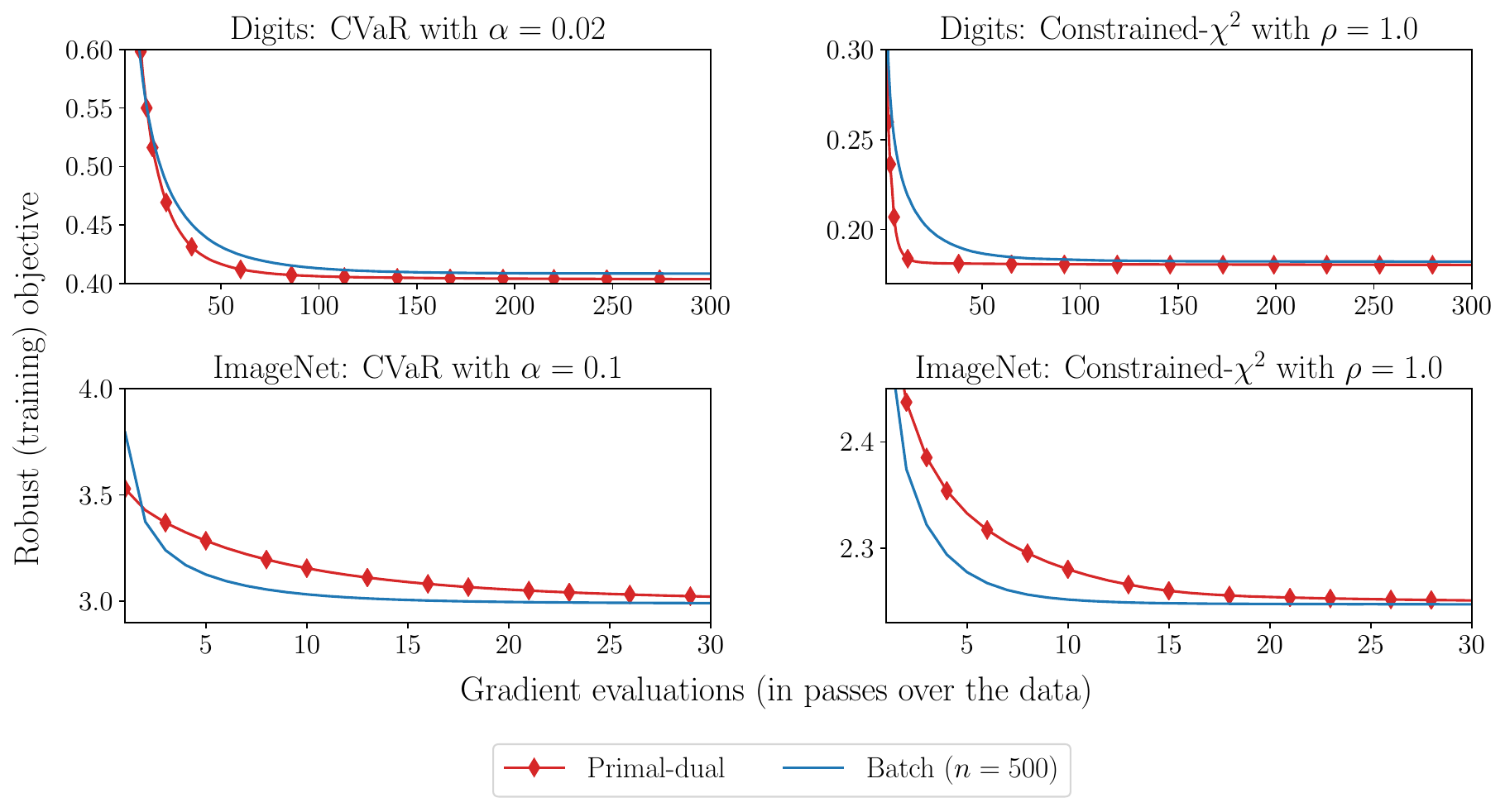}
		\caption{Comparison of batch methods to primal-dual methods. We observe that the primal-dual methods are more efficient on the Digits experiment, but the trend reverses on the large-scale ImageNet experiment.}
		\label{fig:primal-dual}
	\end{center}
\end{figure}

\paragraph{Discussion of results.}
Figure~\ref{fig:primal-dual} compares primal-dual and mini-batch primal methods with the best-performing hyperparameters, for two datasets and two objectives. For the Digits experiment, the primal-dual method perform better that the primal-only method (for $\cs$ significantly so). 
This may appear surprising, since the primal-dual complexity guarantees are larger by an additional factor of $N=60.6$K for this dataset. 
However, a closer look at the analysis of primal-dual methods shows that the term $NB^2$ is actually an upper bound on $\sum_{i=1}^N [\ell(x;s_i)]^2$ at $x=x_1,x_2,\ldots$. As the method converges, many data points are correctly classified with high confidence and therefore have very low value of $[\ell(x;s_i)]^2$. Hence, a more realistic complexity estimate would replace $N$ by the number of incorrectly classified training points, which for Digits is quite small (less than 100). Moreover, we observe that the optimal value of $\gamma_x$ for primal-dual methods is significantly larger than the corresponding step size for the primal-only method, likely because $\tilde{g}^{x}$ gives uniform weights to each $s_i$ as opposed to the adversarial weight of the primal-only method. The larger step sizes enable more rapid optimization over $x$.

For the larger-scale ImageNet experiment, the primal-only method significantly outperforms the primal-dual method. This is consistent with the above discussion, since here the number of misclassified training examples is large (more than 100K).

As an additional illustration of the superior scalability of primal-only method, consider a  thought experiment where we replicate each element in our dataset $m$ times to form a new dataset of size $mN$. Clearly, this will have no impact on the primal-only method. In contrast, the norm of $\tilde{g}^{q}$ will grow by a factor of $m$, and we may expect the complexity of the method to increase by that factor as well.

Finally, we remark that tuning the primal-dual method is considerably more difficult than tuning the primal-only method. In addition to having two learning rates to search over, using an overly large value for $\gamma_q$ typically causes the algorithm to converge to a suboptimal point rather than diverge. Therefore, the common procedure of decreasing the learning rate until divergence no longer occurs will fail for the primal-dual method.
 
\end{document}